\documentclass[12pt]{article}

\usepackage{amsmath,amssymb,amscd,amsthm,verbatim,thmtools}
\usepackage{marvosym}
\usepackage{wasysym}
\usepackage[labelsep=none]{caption}
\usepackage{subcaption}
\usepackage{enumerate}
\usepackage{bm}
\usepackage{titlesec}

\usepackage{tikz, tikz-cd}
\usepackage{graphicx}

\usepackage[vcentermath]{youngtab}
\usepackage[all,2cell]{xy}

\usepackage[ocgcolorlinks, linkcolor=black]{hyperref}
\usepackage{url}

\usepackage[top=1in,bottom=1in,left=1in,right=1in]{geometry}

\usepackage[T1]{fontenc}

\usepackage{cmbright}
\usepackage{eulervm}
\usepackage{titlesec}

\declaretheoremstyle[spaceabove=3pt, spacebelow=3pt, headfont=\normalfont\bfseries, notefont=\mdseries, notebraces={(}{)},
bodyfont=\normalfont,
postheadspace=0.5em]{mystyle}

\declaretheoremstyle[spaceabove=3pt, spacebelow=3pt, headfont=\normalfont\bfseries, notefont=\mdseries, notebraces={(}{)},
bodyfont=\normalfont\itshape,
postheadspace=0.5em]{mytheorem}

\declaretheorem[style=mystyle,qed=$\lhd$,numberwithin=section]{definition}
\declaretheorem[style=mystyle,qed=$\lhd$,sibling=definition]{example}
\declaretheorem[style=mystyle,qed=$\lhd$,sibling=definition]{remark}
\declaretheorem[style=mytheorem,sibling=definition]{theorem}

\theoremstyle{mystyle}
\newtheorem{questionx}[theorem]{Question}
\newtheorem{conjecturex}[theorem]{Conjecture}

\theoremstyle{mytheorem}
\newtheorem{main}{Theorem}

\newtheorem{proposition}[theorem]{Proposition}

\newtheorem{lemma}[theorem]{Lemma}

\newtheorem{corollary}[theorem]{Corollary}

\newtheorem{claim}[theorem]{Claim}

\theoremstyle{definition}

\newtheorem{xample}[theorem]{Example}

\newtheorem{sublemma}[theorem]{Sub-lemma}

\newcommand{\nc}{\newcommand}
\nc{\dmo}{\DeclareMathOperator}

\nc{\I}{\mathcal{I}}
\nc{\K}{\mathcal{K}}

\nc{\Q}{\mathbb{Q}}
\nc{\Ad}{\mathbb{A}}
\nc{\R}{\mathbb{R}}
\nc{\Z}{\mathbb{Z}}
\nc{\C}{\mathbb{C}}
\nc{\N}{\mathbb{N}}
\nc{\F}{\mathbb{F}}
\nc{\cR}{\mathcal{R}}
\nc{\cM}{\mathcal{M}}
\nc{\cC}{\mathcal{C}}
\nc{\fS}{\mathfrak{S}}
\nc{\cS}{\mathcal{S}}
\nc{\bk}{\mathbb{K}}
\nc{\ck}{\mathcal{k}}
\nc{\XyX}{\mathcal{S}}
\nc{\Schur}{\mathbb{S}}
\nc{\rohit}[1]{{\color{blue}[\text{Rohit: #1}]}}

\dmo{\GL}{GL}
\dmo{\Mat}{Mat}
\dmo{\PSL}{PSL}
\nc{\gin}{i}
\nc{\ga}{\Gamma}
\dmo{\Out}{Out}
\dmo{\Aut}{Aut}
\dmo{\Stab}{Stab}
\dmo{\wt}{weight}
\nc{\ddH}{\Delta H}
\dmo\zMod{\Z-Mod}
\dmo{\adj}{Ad}

\dmo\im{im}
\dmo{\Rep}{Rep}
\dmo\id{id}
\dmo\SL{SL}
\dmo\Sp{Sp}
\dmo\Mod{Mod}
\dmo \range{range}
\dmo\PMod{PMod}
\dmo\fd{fd}
\dmo\IA{IA}
\dmo\IAut{IA}
\nc{\IAn}{\IA_n}
\dmo\Sym{Sym}
\dmo\Ind{Ind}
\dmo\Res{Res}
\dmo\res{res}
\dmo\tr{tr}
\dmo\gr{gr}
\dmo\Free{Free}
\dmo\spn{span}
\dmo\End{End}
\dmo\Conf{Conf}
\dmo\conf{conf}
\dmo\op{op}
\dmo\coker{coker}

\dmo\Homeo{Homeo}
\dmo\Inj{Inj}
\dmo\Teich{Teich}

\dmo\rk{rank}
\dmo\ab{ab}
\dmo\rank{rank}
\dmo\Emb{Emb}
\dmo\Map{Map}
\dmo\chr{char}
\dmo\SQ{SQ}
\dmo\SD{SD}
\dmo\TE{TE}
\dmo\upTE{{_{\uparrow}TE}}
\dmo\length{length}
\dmo\CW{CW}
\dmo\Ext{Ext}

\dmo\Tor{\mathcal{T}}
\dmo\Torsion{Tor}

\nc{\bwedge}{\textstyle{\bigwedge}}

\dmo\FI{FI}
\dmo\Tr{Tr}
\dmo\pr{pr}
\dmo\FIMod{FI-Mod}
\dmo\FIModgen{FI-Mod^{gen}}
\dmo\coFIMod{co-FI-Mod}
\dmo\dMod{-Mod}
\dmo\Vect{Vect}
\dmo\dVect{-Vect}

\dmo\dRep{-Rep}

\dmo\alg{alg}

\nc\kT{k[T]}
\def\FIsharp{\FI\sharp}

\dmo\coFI{co-FI}
\dmo\hTop{hTop}
\dmo\Range{Range}
\dmo\lcm{lcm}

\nc\y[1]{{\tiny\yng(#1)}}
\nc\x{\hspace{0.1em}}

\newcommand{\ra}{\rightarrow}

\newcommand{\beq}{\begin{displaymath}}
\newcommand{\eeq}{\end{displaymath}}
\newcommand{\beqn}{\begin{equation}}
\newcommand{\eeqn}{\end{equation}}

\newcommand{\inj}{\hookrightarrow}
\newcommand{\surj}{\twoheadrightarrow}

\DeclareMathOperator{\Hom}{Hom}

\renewcommand{\epsilon}{\varepsilon}
\nc{\coloneq}{\mathrel{\mathop:}\mkern-1.2mu=}
\nc{\margin}[1]{\marginpar{\scriptsize #1}}
\nc{\para}[1]{\medskip\noindent\textbf{#1.}}
\nc{\hide}[1]{#1}

\nc{\arXiv}[1]{\href{http://arxiv.org/abs/#1}{\nolinkurl{arXiv:#1}}}
\nc{\arXivV}[2]{\href{http://arxiv.org/abs/#1}{\nolinkurl{arXiv:#1v#2}}} 

\nc\p{\mathfrak{p}}
\nc\X{\mathbf{X}}

\let\OLDthebibliography\thebibliography
\renewcommand\thebibliography[1]{
  \OLDthebibliography{#1}
  \setlength{\parskip}{0pt}
  \setlength{\itemsep}{0.3ex}
}

\usetikzlibrary{arrows}

\title{$\FI$-modules and the cohomology of modular representations of symmetric groups}

\author{Rohit Nagpal }
\date{}

\begin{document}

\maketitle

\begin{abstract}
An $\FI$-module $V$ over a commutative ring $\bk$ encodes a sequence $(V_n)_{n \geq 0}$ of representations of the symmetric groups $(\fS_n)_{n \geq 0}$ over $\bk$. In this paper, we show that for a "finitely generated" $\FI$-module $V$ over a field of characteristic $p$, the cohomology groups $H^t(\fS_n, V_n)$ are eventually periodic in $n$. We describe a recursive way to calculate the period and the periodicity range and show that the period is always a power of $p$. As an application, we show that if $\mathcal{M}$ is a compact, connected, oriented manifold of dimension $\geq 2$ and $\conf_n(\mathcal{M})$ is the configuration space of unordered $n$-tuples of distinct points in $\mathcal{M}$ then the mod-$p$ cohomology groups $H^{t}(\conf_n(\mathcal{M}),\bk)$ are eventually periodic in $n$ with period a power of $p$.
\end{abstract}

\tableofcontents

\section{Introduction} 
\label{section:introduction}

Nakaoka showed in his 1960 paper  \cite{N} that the cohomology groups $H^t(\fS_n , V)$ stablizes in $n$ when $V$ is a finitely generated abelian group with a trivial action of $\fS_n$. More precisely he showed that
\begin{equation} H^t(\fS_n , V) \cong H^t(\fS_{n-1}, V), \qquad \text{if } n>2 t.\label{eqn:nakaoka} \end{equation} 
Immediately afterwards, A. Dold published his paper~\cite{D}  significantly simplifying Nakaoka's arguments. Over the next two decades the cohomology groups $H^t(\fS_n , \mathbb{Z}/p \mathbb{Z})$ were of great research interest. Nakaoka and others studied theses groups extensively determining the structure of the groups and the cohomology operations on them; see \cite{N1}, \cite{N2}, \cite{Mg}, \cite{May}, \cite{M}.

It is natural to ask what happens if we allow $V$ to vary with $n$ or let $\fS_n$ act nontrivially on $V$. But, to make sense out of this we need coherent sequences of $\fS_n$-representations. The coherent sequences of $\fS_n$-representations were first observed as a ubiquitously occurring phenomena by Church and Farb in \cite{CF}. Soon after, the theory of $\FI$-modules was introduced and developed by Church, Ellenberg and Farb in \cite{CEF} to study such sequences. The notion of coherence of such a sequence was made precise by the natural notion of finite generation of an $\FI$-module.

In \cite{CEF}, it was shown that the finitely generated $\FI$-modules over a characteristic $0$ field form an abelian category and several new theorems in topology, algebra, combinatorics and algebraic geometry were proved as an application. Almost at the same time, Sam and Snowden studied a category equivalent to the category of finitely generated $\FI$-modules in \cite{SS} and provided a detailed analysis of the algebraic structure of the category in characteristic $0$.   

In \cite{CEFN}, Church-Ellenberg-Farb and the current author studied the category of finitely generated $\FI$-modules over a general Noetherian ring showing that the category is abelian and that there is an inductive description of a finitely generated $\FI$-module $V=(V_n)_{n \geq 0}$: there exists some $N \geq 0$ such that if $n \geq N$ then $V_n$ can be described in terms of $V_j$, $j \leq N$. Such an inductive description had appeared before in characteristic $0$ (see Putman's paper \cite{P}) and has several applications. In this paper, the techniques introduced in \cite{CEFN} are used extensively.

Recently, $\FI$-modules and the related theories have been very fruitful. Church-Ellenberg-Farb has related it to the point counts on varieties defined over finite fields via Grothendieck-Lefschetz fixed point theorem (\cite{CEF2}). Related theory of $\FI$-groups has been developed by Church and Putman in \cite{CP} to obtain results about the Johnson filtration on the mapping class groups. The theory has also been generalized to coherent sequences of representations of classical Weyl groups by Wilson which yields another set of applications (\cite{Wil}). 

In this paper, we generalize Nakaoka's stability theorem and show that the cohomology groups $H^t(\fS_n, V_n)$, where the sequence of $\fS_n$-representations $(V_n)_{n\geq 0}$ is given by a finitely generated $\FI$-module $V$, are eventually periodic in $n$. A map $f: U \ra V$ of finitely generated FI-modules determine $\fS_n$-equivariant maps $f_n: U_n \ra V_n$ for each $n \geq 0$. We show that kernel and image of $f_{\star, n}: H^t(\fS_n, U_n) \ra H^t(\fS_n, V_n)$ are eventually periodic in $n$. Similarly, if \[0 \ra U \ra V \ra W\ra 0\] is an exact sequence of finitely generated $\FI$-modules then the kernel and the image of the connecting homomorphisms in the cohomology long exact sequence are eventually periodic in $n$. 

The simplest example where periodicity can be observed is constructed as follows. Working over $\mathbb{F}_2$, let $(V_n)_{n\geq 0}$ be the sequence consisting of permutation representations  ($V_n = (\mathbb{F}_2)^n$)  and  $(U_n)_{n \geq 0}$ be the sequence consisting of trivial representations ($U_n = \mathbb{F}_2$). These sequences form finitely generated $\FI$-modules $V$ and $U$ respectively and the map $\phi: V \ra U$ given by \[\phi_n(a_1, a_2, \ldots, a_n) = \sum_{i=1}^n a_i \qquad \forall n\] is an $\FI$-module homomorphism. Now $W \coloneq \ker{\phi} = (\ker{\phi_n})_{n \geq 0}$ forms an $\FI$-module with $H^0(\fS_n, W_n)= W_n^{\fS_n}$ which is nontrivial only when $n\geq 2$ and $2 \mid n$.

We remark that our results do not provide another proof of Nakaoka's theorem. Rather we build on Nakaoka's theorem addressing a seemingly orthogonal set of difficulties. Unlike Nakaoka's stability theorem our periodicity result is nontrivial even in the case when $t=0$ (see example~\ref{example:1}).

As a direct application to our results, we show that the mod-$p$ cohomology of the unordered configuration space of a compact, connected and oriented  manifold of dimension $\geq 2$ is periodic. The precise statement and a short history of related works on configuration spaces is provided in \S\ref{subsection:applicationintro}.

\para{Related work on other categories} Parallel theories for the twisted cohomology (or homology) of several different sequences of groups exist or are being worked out. Some examples include: cohomology of free groups via polynomial functors \cite{DPV}; rational cohomology of classical groups via polynomial functors \cite{Touze}, \cite{Kuhn}; homology groups $H_t(\GL_n(R), T(\adj_n(R))$ where $T\colon \zMod \ra \zMod$ is a functor of finite degree \cite{dwyer}; homology groups $H_t(\fS_n, V_n)$ where $V$ is a finitely generated $\FI$-module \cite{Wah} (note here that the universal coefficient theorem cannot be used to deduce the cohomology version from the homology version and vice-versa because the coefficients are nontrivial representations of $\fS_n$); stability of twisted homology groups $H_t(\SL_n(R), M_n)$ and $H_t(\Sp_{2 n}(R), M_n)$ where $R$ is a finite ring \cite{PS}. 

A general twisted homological stability result for sequences of representations of wreath product groups $G\wr \fS_n$, where $G$ is a polycyclic-by-finite group, is proved in \cite{gmaps}. Andrew Snowden and the current author are developing a cohomological version for sequences of representations of the groups $G \wr \fS_n$ \cite{periodicity}. In particular, when $G$ is the trivial group, \cite{periodicity} provides an alternate proof of periodicity of the cohomology groups $H^t(\fS_n ,V_n)$ with better bounds on the periodicity range but without optimal bounds on the period.
%
%
%
%

\para{Outline of the paper} We start with an overview of $\FI$-modules in \S\ref{subsection:overview} and the rest of the \S\ref{section:introduction} presents precise statements of our results and an outline of the proofs. In \S\ref{section:filteredFI}, we prove results about the structure of finitely generated $\FI$-modules over a Noetherian ring. Our main theorem is proved in \S\ref{section:periodicityofinv} and the application to configuration spaces is given in \S\ref{section:application}. In \S \ref{section:questions-comments}, we provide some related open questions.

\para{Acknowledgements}
First, I thank my advisor Jordan Ellenberg for  introducing me to the problem. For his precisely formulated questions and innumerous suggestions that encouraged me throughout and helped me prove the most general results. For sharing his vision that made me investigate the spectral sequence for the covering map from the configuration spaces of ordered points to the configuration spaces of unordered points which lead to the main application. Next, I thank Thomas Church and Benson Farb for inviting me to present this work at the Geometry and Topology seminar at University of Chicago in Fall 2013 when the work was only in its rudimentary form. I thank Benson Farb for inviting me again to present this work at University of Arkansas in Spring 2015 and for his helpful conversations on configuration spaces. I am very grateful to Thomas Church for showing me a direct short proof of Lemma~\ref{lemma:homotopy} and for several hours of crucial discussions we had in coffee shops of Chicago that greatly influenced the paper as it is now. I thank Eric Ramos and Jenny Wilson for going through parts of the paper and for their comments. This work is a part of my thesis at University of Wisconsin Madison and I thank the institution for the support throughout.
\subsection{An overview of Church-Ellenberg-Farb's theory of \texorpdfstring{$\FI$}{FI}-modules} \label{subsection:overview}

Let $\FI$ be the category whose objects are finite sets and a morphism between finite sets $A$ and $B$ is an injection $f \colon A \to B$. The category $\FI$ is equivalent to its full subcategory whose objects are the sets $\{1,2, \ldots, n\}$, $n \geq 0$. For simplicity we denote $\{1,2, \ldots, n\}$ by $[n]$ and  $\{-1,-2, \ldots, -n\}$ by $[-n]$. The empty set is denoted by $[0]$.

An $\FI$-module over a commutative ring $\bk$ is a covariant functor from the category $\FI$ to the category of $\bk$-modules. The category of $\FI$-modules over $\bk$ is denoted by $\FIMod_\bk$. For an $\FI$-modules $V$, we denote the $\bk$-module $V([n])$ by $V_n$. Since the group $\End_{FI}([n])$ of endomorphism of $[n]$ is naturally isomorphic to the symmetric group $\fS_n$, any $\FI$-module $V$ determines a sequence of $\bk[\fS_n]$-modules $(V_n)_{n \geq 0}$ with linear
maps between them respecting the group actions. For an $\FI$-morphism $f \colon [m] \ra [n]$, we denote the map $V(f) : V_m \ra V_n$ by $f_{\star}$.

It is known that the category of functors from any small category to an abelian category is abelian (\cite[Remark~2.1.2]{CEF}, \cite[A.4.3]{Wei}), so $\FI$-modules form an abelian category. Moreover, notions such as kernel,
cokernel, subobject, quotient object, injection, or surjection are all defined "pointwise", meaning that
a property holds for an $\FI$-module $V$ if and only if it holds for each $V_n$. For example, we say a map
$V \ra W$ of $\FI$-modules is an injection if and only if the maps $V_n \ra W_n$ are injections for all $n$.

An $\FI$-module $V$ is \emph{finitely generated} if there is a finite set $S$ of elements in $\bigsqcup_i V_i$ so that no
proper sub-$\FI$-module of $V$ contains $S$.

\begin{definition}[{\bf The  $\FI$-module $M(m)$ }{\cite[Definition~2.2.3]{CEF}, \cite[Definition~2.2]{CEFN}}]
\label{def:Mm}
For any $m\geq 0$, the $\FI$-module $M(m)$ takes a finite set $S$ to the free $\bk$-module $M(m)_S$ on the set of injections $[m]\inj S$. In other words, $M(m)=\bk[\Hom_{\FI}([m],{-})]$; by the Yoneda lemma, $M(m)$ is uniquely determined by the natural identification
\[\Hom_{\FIMod}(M(m),V)\cong V_m.\]


An $\FI$-module V is said to be {\bf free} if \[V = \bigoplus_{i\in I} M(m_i). \qedhere \]
\end{definition}

\begin{definition}[{\bf The  $\FI$-module $M(W)$ }{\cite[Definition~2.2.2]{CEF}}]
\label{def:MW}
For any $\bk[\fS_m]$-module $W$, the $\FI$-module $M(W)$ takes a finite set $S$ of size $n$ to the $\bk[\fS_n]$-module $M(W)_S$ given by $M(W)_S \coloneq \bk[\Hom_{\FI}([m],{S})] \otimes_{\bk[\fS_m]} W$. In other words, \[M(W)_n = \Ind_{\fS_m \times \fS_{n-m}}^{\fS_n} W \boxtimes \bk = \bk[\fS_n] \otimes_{\bk[\fS_m] \times \bk[\fS_{n-m}]} (W \boxtimes \bk).\] By the Yoneda lemma, $M(W)$ is uniquely determined by the natural identification
\[\Hom_{\FIMod}(M(W),V)\cong \Hom_{\fS_m}(W, V_m).\]
Note here that $M(\bk[\fS_m])$ is same as $M(m)$.
\end{definition}

\begin{remark}
It is shown in \cite[Theorem~4.1.5]{CEF} that $\FI$-modules  of the form \[V = \bigoplus_{i\in I} M(W_i)\] for some $\bk[\fS_{m_i}]$-modules $W_i$ are precisely the $\FI$-modules with $\FIsharp$  (see \cite[Definition~4.1.1]{CEF}) structure on them. For our purpose, an $\FIsharp$-module is an $\FI$-module admitting a decomposition as above.

If an $\FIsharp$-module $V$ is finitely generated as an $\FI$-module. Then it follows (for example, from the argument in Proposition~\ref{prop:boundedness}) that the cohomology groups $H^t(\fS_n, V_n)$ stabilizes and hence are periodic with period $1$.
\end{remark}

We have the following characterization of finite generation in terms of $\FI$-modules $M(m)$.

\begin{definition}[{\bf Finitely generated $\FI$-modules} {\cite[Proposition~2.3.4]{CEF}, \cite[Definition~2.1]{CEFN}}]
\label{def:finitegeneration}

Let $V$ be an $\FI$-module. We say that $V$ is {\bf finitely generated} if there exists a surjection
\[\bigoplus_{i=1}^d M(m_i)\surj V\] for some integers $m_i\geq 0$.
We say that $V$ is {\bf generated in degree $\leq m$} if there exists a surjection
\[\Pi \colon \bigoplus_{i\in I} M(m_i)\surj V\qquad\text{ with all }m_i\leq m.\] Here the sum may be infinite. The data of the integers $m_i$, $i \in I$ is the {\bf degree structure} of $\Pi$ and $\max_{i \in I} m_i$ is the {\bf degree} of $\Pi$.
\end{definition}

We need the following Noetherian property of finitely generated $\FI$-modules.

\begin{theorem}[{\bf Noetherian property} {\cite[Theorem~A]{CEFN}}]
\label{thm:noetherianProp}
If $V$ is a finitely-generated $\FI$-module over a Noetherian ring $\bk$, and $W$ is a sub-$\FI$-module of $V$, then $W$ is finitely generated.
\end{theorem}

The positive shift functors were used extensively to prove the above Noetherian property. Shift functors are essential for our purpose as well.

\begin{definition}[{\bf Positive shift functor $\cS_{+a}$}  {\cite[Definition~2.8]{CEFN}}]
\label{def:shiftpos}
Given an $\FI$-module $V$ and an integer $a\geq 1$, the functor $\cS_{+a}\colon \FIMod \to \FIMod$ is defined by \[(\cS_{+a}V)_S \coloneq V_{S \sqcup [a]} \]
It follows that as a $\bk[\fS_n]$-module $\cS_{+a}V_n \cong \Res_{\fS_{n}}^{\fS_{n+a}} V_{n+a}$ Since kernels and cokernels are computed pointwise, $\cS_{+a}$ is an exact functor. Also note that there is a natural isomorphism \[\cS_{+(a+b)}V \cong \cS_{+a} \cS_{+b} V. \qedhere \]
\end{definition}

\begin{definition} 
\label{def:torsionsubmodule}
 The natural inclusion $S \hookrightarrow S \sqcup [a]$ is a morphism in the category $\FI$, which induces a map $X_a(V): V \ra \cS_{+a} V$ of $\FI$-modules. The {\bf torsion submodule} of $V$ denoted $T(V)$ is a sub-$\FI$-module of $V$ given by:\[T(V) = \bigcup_{a \geq 0} \ker X_a(V).\] Working over a Noetherian ring $\bk$, $T(V)$ is finitely generated by Theorem~\ref{thm:noetherianProp} and hence for each $a \geq 0$, $(X_a(V))_n$ is injective for large enough $n$ (see \cite[Lemma~2.15]{CEFN} for more details). 
\end{definition}

In some fortunate situations we can define a meaningful map in the opposite direction; $\cS_{+a}V \to V$. An example where this holds is when $V$ is free, which can be seen from the following lemma.

\begin{lemma}[{\cite[Proposition~2.12]{CEFN}}]
\label{lemma:SaMd}
For any $a\geq 0$ and any $d\geq 0$, there is a natural decomposition
\begin{equation}
\label{eq:Mderivative}
\cS_{+a}M(d)=M(d)\oplus Q_a
\end{equation}
where $Q_a$ is a free $\FI$-module finitely generated in degree $\leq d-1$. In particular, if $V$ is a free $\FI$-module then $V$ is a direct summand of $\cS_{+a} V$.
\end{lemma}

\begin{remark} \label{remark:coFIMod}
 When $a=1$ we have $\cS_{+ 1} M(m) = M(m) \oplus M(m-1)^{\oplus m}$. The decomposition of $\cS_{+a} M(m)$ can be computed using the case $a=1$ and the natural isomorphism $\cS_{+(a+b)}V \cong \cS_{+a} \cS_{+b} V$. We generalize the above decomposition in Lemma~\ref{lemma:zero} which shows in particular that if $V$ is an $\FIsharp$-module and $a\geq 0$, then $V$ is a direct summand of $\cS_{+a} V$.
\end{remark}

\subsection{Introduction to \texorpdfstring{$\sharp$}{sharp}-filtered \texorpdfstring{$\FI$}{FI}-modules}
In \S\ref{section:filteredFI}, we show that if $V$ is a finitely generated $\FI$-module over a Noetherian ring $\bk$ then $V$ is built out of $\FI$-modules of the form $M(W)$ which, by Remark~\ref{remark:coFIMod} is a member of a strictly smaller category. To make it precise we need a couple of definitions.

\begin{definition}[{\bf $\sharp$-filtered $\FI$-modules}]
\label{def:filteredfimod}A {\bf $\sharp$-filtered $\FI$-module} is a surjection \[\Pi : \bigoplus_{i=1}^d M(m_i) \twoheadrightarrow V\] of $\FI$-modules such that the filtration \[0 = V^0 \subset V^1 \subset \ldots \subset V^d=V \] given by \[V^r \coloneq \Pi(\bigoplus_{i=1}^r M(m_i)),\qquad\text{ }0 \leq r \leq d\] has graded pieces of the form $M(W)$, that is, the filtration satisfies $\frac{V^r}{V^{r-1}} \cong M(W_r)$ where $W_r$ are some $\bk[\fS_{m_r}]$-modules. We call the filtration induced by $\Pi$, the {\bf $\sharp$-filtration} of $V$. We use $\Pi$ and $V$ interchangeably if there is an obvious surjection giving $V$ a $\sharp$-filtered $\FI$-module structure. 

We call $d$ the {\bf length} of the $\sharp$-filtration and the pair $\tilde{V}\coloneq (\bigoplus_{i=1}^d M(m_i), (m_i)_{1 \leq i \leq d})$ a {\bf cover} of $V$. The second coordinate just keeps track of the order which yields the desired $\sharp$-filtration. We dispense with the second co-ordinate whenever it is clear from the context and identify the cover with the first co-ordinate.
\end{definition}
The above definition implies that there are surjections $\pi^r: M(m_r) \twoheadrightarrow M(W_r)$ and hence each $W_r$ is a singly generated $\bk[\fS_{m_r}]$-module. But this is just a matter of convenience (see Lemma~\ref{lemma:zero}). Note that NOT every finitely generated $\FI$-module admits a $\sharp$-filtration. But if a finitely generated $\FI$-module $V$ admits a filtration \[0 = V^0 \subset V^1 \subset \ldots \subset V^d=V \] satisfying $\frac{V^r}{V^{r-1}} \cong M(W_r)$ where $W_r$ are some $\bk[\fS_{m_r}]$-modules, then $V$ admits a $\sharp$-filtered $\FI$-module structure. The definition of a $\sharp$-filtered $\FI$-module just contains extra data of a cover (this extra data is useful in producing bounds on the period of the cohomology groups).

\begin{definition}
\label{def:sequential}
Let $\Pi^1: \bigoplus_{i=1}^{d_1} M(m_i) \ra V$ and $\Pi^2: \bigoplus_{k=1}^{d_2} M(n_k) \ra W$ be $\sharp$-filtered $\FI$-modules. An $\FI$-module map $\phi : V \ra W$ is called {\bf sequential} if there are subsets $S \subset [d_1]$, $T \subset [d_2]$ and an isomorphism $f_{\phi} : S \ra T$ such that $n_k = m_{f_{\phi}^{-1}(k)}$ for each $k \in T$ and the map of covers $\tilde{\phi}: \bigoplus_{i=1}^{d_1} M(m_i) \ra \bigoplus_{k=1}^{d_2} M(n_k)$ given by
\[(\tilde{\phi}_n(L))_k =
\begin{cases}
L_{f_{\phi}^{-1}(k)}, & \text{if } k \in T  \\
0, & \text{otherwise } \\
\end{cases}
\] determines $\phi$. In other words, if $l \in V_n$ and $L = (L_i)_{1 \leq i \leq d_1} \in\bigoplus_{i=1}^{d_1} M(m_i)_n$ is any lift of $l$ (that is, $\Pi^1_n(L)=l$) then $\tilde{\phi}_{n}(L)$ is a lift of $\phi(l)$.
\end{definition}

The following theorem is proven in \S\ref{section:filteredFI}.

\begin{main}[{\bf A $\sharp$-filtered resolution of a finitely generated $\FI$-module}]
\label{thm:structure}
Let $V$ be a finitely generated $\FI$-module over a Noetherian ring $\bk$ and is generated in degree $D$, that is, $V$ admits a surjection \[\Pi: \tilde{V}\coloneq \bigoplus_{i=1}^d M(m_i) \ra V\] with $m_i \leq D$ for each $i$. Then, \begin{enumerate}

\item[(A)] for large enough $a$, $\cS_{+a} V$ is $\sharp$-filtered; 

\item[(B)] there exists a commutative diagram (Figure~\ref{Fig:structure}) \begin{figure}[h]
\[\begin{tikzcd}
 {0}\ar{r} & \tilde{V} \ar{r}{\tilde{\iota}} \ar{d}{\Pi} & \tilde{J^0}\ar{r}{\tilde{\phi}^0} \ar{d}{\Pi^0} &
 \tilde{J}^1 \ar{r}{\tilde{\phi}^1} \ar{d}{\Pi^1} & \tilde{J}^2 \ar{d}{\Pi^2} \ar{r}{\tilde{\phi}^2} & \ldots \ar{r}{\tilde{\phi}^{N-1}} & \tilde{J}^N \ar{d}{\Pi^N} \ar{r} & 0\\ 
 {0}\ar{r} & V \ar{r}{\iota} \ar{d} & J^0\ar{r}{\phi^0} \ar{d} &
  J^1 \ar{r}{\phi^1} \ar{d} & J^2 \ar{d} \ar{r}{\phi^2} & \ldots \ar{r}{\phi^{N-1}} & J^N \ar{d} \ar{r} & 0\\
   & 0 & 0 &
    0 & 0 & & 0
 \end{tikzcd}\]
 \caption{} \label{Fig:structure}
 \end{figure} with the columns and the first row exact and the second row exact in high enough degree (say $n \geq C$), that is, the sequence \[0 \ra V_n \ra J^0_n \ra J^1_n \ra \ldots \ra J^N_n \ra 0\] is exact if $n \geq C$. Here for each $0 \leq i \leq N$,   $J^i$  is a $\sharp$-filtered $\FI$-module with cover $\Pi^i:\tilde{J}^i \ra J^i$ and for each $0 \leq i \leq N-1$, the map $\phi^i$ is sequential (sequentialness is witnessed by $\tilde{\phi}^i$). Moreover, $N \leq D$ and $\tilde{J}^i$ is generated in degree at most $D-i$.
 \end{enumerate}
\end{main}

\begin{remark} If $W$ is a projective $\bk[\fS_m]$-module, then $M(W)$ is a projective $\FI$-module (see \cite[Remark~2.2.A]{CEF} or \cite{Wei}). This implies that if $\bk$ is a field of characteristic $0$ or of characteristic $ > D$ then each $J^i$ in Theorem~\ref{thm:structure} is a direct sum of $\FI$-modules of the form $M(W)$. Thus \cite[Theorem ~1.13]{CEF}, which shows that a finitely generated $\FI$-module is uniformly representation stable in the sense of $\cite{CF}$, follows from Theorem~\ref{thm:structure} because by Pieri's rule $\FI$-modules of the form $M(W)$ are uniformly representation stable.
\end{remark}

\begin{remark}
Over a field of characteristic $0$, Sam and Snowden showed in \cite{SS} that $\FI$-modules of the form $M(W)$ are injective in the category of finitely generated $\FI$-modules and that every finitely generated $\FI$-module admits a finite injective resolution. 

Example~\ref{example:1} shows that in positive characteristic, there are $\sharp$-filtered $\FI$-modules which are not  NOT direct sums of $M(W)$'s. Unlike the category of $\bk[\fS_m]$-modules, projective objects in $\FIMod$ may not be injective. We believe that the category of finitely generated $\FI$-modules over a field of positive characteristic does not have enough injectives.
\end{remark}

\begin{remark}
We believe that it is possible to obtain a good bound on the constant $C$ in Theorem~\ref{thm:structure} via upcoming works of Church and Ellenberg (\cite{ChE}) and that of Ramos (\cite{ramos}).
\end{remark}

Note that when $\bk$ is a field and $W$ is a $\bk[\fS_m]$-module, then $\dim_{\bk} M(W)_n = \binom{n}{n-m} \dim_{\bk} W $ is an integer valued polynomial in $n$. It is proven in \cite[Theorem~B]{CEFN} that for a finitely generated $\FI$-module $V$, $\dim_{\bk} V_n$ is eventually a polynomial in $n$. Theorem~\ref{thm:structure} immediately yields a strengthening of this result.

\begin{theorem}[{\bf Polynomiality}]
\label{thm:polynomial}
Suppose $\bk$ is a Noetherian ring. Let $V$ be a finitely-generated $\FI$-module over $\bk$.  Then there exist finitely many integer-valued polynomials $P_i(T) \in \Q[T]$ and nontrivial $\bk$-modules $W_i$ so that for all sufficiently large $n$, \[ [V_n] = \sum_{i} P_i(n) [W_i]\] in the Grothendieck group $K_0(\bk)$ of $\bk$.
\end{theorem}

\begin{definition}
\label{def:euler-degree}
For a finitely generated $\FI$-module $V$ over a Noetherian ring $\bk$, we define $\chi(V)$ to be the smallest nonnegative integer $c$ such that for each of the finitely many polynomials $P_i(T)$ in Theorem~\ref{thm:polynomial} the degree of $P_i(T) \leq c$. When $\bk$ is a field, $\chi(V)$ is equal to the degree of the polynomial $P(T)$ which satisfies $\dim_{\bk} V_n = P(n)$ for large enough $n$.

We define $D_V$ to be the least number $c$ such that $V$ is generated in degree $\leq c$.
\end{definition}

\begin{remark}\label{rem:dimension-generation}
Note that for a finitely generated $\FI$-module $V$, we have $D_V \geq \chi(V)$. For a $\sharp$-filtered $\FI$-module $J$ the equality holds: $D_J = \chi(J)$. Moreover, Theorem~\ref{thm:structure} and its proof (see \S \ref{section:filteredFI}) implies that $\chi(V) = \chi(J^0)$ and that $\tilde{J}^i$ are generated in degree $\leq \chi(V) - i$.
\end{remark}

\subsection{Statements of the main theorems for \texorpdfstring{$\sharp$}{sharp}-filtered \texorpdfstring{$\FI$}{FI}-modules} 
\label{subsection:mainfilteredintro}

The goal of \S \ref{subsection:mainfilteredintro} is to state Theorem~\ref{thm:filteredstabilityrange}, which is the special case of our main theorem where V is a $\sharp$-filtered FI-module.  In order to state the theorem, we need to set up some preliminary notations and definitions.

For any $n \geq 0$, let $\mathcal{B}_{t}(\fS_n)$ be the free abelian group on $\fS_n^{t+1}$ with $\fS_n$ acting diagonally. Then $\mathcal{B}_{\bullet}(\fS_n) \to \mathbb{Z} \to 0$ is the bar resolution of $\fS_n$ where the differential $\partial_{t} : \mathcal{B}_{t-1}(\fS_n) \to \mathcal{B}_{t}(\fS_n)$ is given by \[\partial_t(\sigma_0, \sigma_1, \ldots, \sigma_t) = \sum_{i=0}^t (\sigma_0, \sigma_1, \ldots, \hat{\sigma}_i, \ldots, \sigma_t)  .\] We allow $t$ to equal $-1$ and define $\mathcal{B}_{-1}(\fS_n)$ to be the trivial $\bk[\fS_n]$-module. We assume that $\bk$ is a field of characteristic $p$ and fix a $\sharp$-filtered $\FI$-module $\Pi: \tilde{V} \coloneq \bigoplus_{i=1}^{d} M(m_i) \ra V$ over $\bk$ of length $d$. Let \[0 = V^0 \subset V^1 \subset \ldots \subset V^d=V\] be the $\sharp$-filtration induced by $\Pi$. For each $1 \leq r \leq d$, the map $\Pi^r: \tilde{V}^r \coloneq \bigoplus_{i=1}^{r} M(m_i) \ra V^r$ obtained by restricting $\Pi$ to $\bigoplus_{i=1}^{r} M(m_i)$ is a cover of $V^r$ and gives $V^r$ a $\sharp$-filtered $\FI$-module structure. Let $D = \max_{i \in [d]} m_i$ be the degree of $\Pi$.
 
 \begin{remark}
Recall that the definition of a $\sharp$-filtered $\FI$-module $\Pi: \tilde{V} \coloneq \bigoplus_{i=1}^{d} M(m_i) \ra V$ contains the data of the $d$-tuple $(m_1, \ldots, m_d)$. Hence $\Pi$ gives rise to the $\sharp$-filtered $\FI$-modules $\Pi^r$ for $1 \leq r \leq d$ as above in a unique way. 
 \end{remark}
 
  With these notations, we have the following definition.

\begin{definition}
\label{def:twistedcyles}
 Let $z^r \in \Hom_{\fS_n}(\mathcal{B}_{t+1}(\fS_n), \tilde{V}^r_n)$. We call \[l^r \in \Hom_{\fS_{n}}(\mathcal{B}_{t}(\fS_n), V_n^r)\] a {\bf twisted} $z^r$-{\bf cycle} if \[l^r \circ \partial_{t+1} = \Pi_n^r \circ z^r.\] 
 
 We say that two twisted $z^r$-cycles $l^r_1$ and $l^r_2$ are {\bf equivalent} if $l^r_1-l^r_2$ is a boundary (classical coboundary) in the usual sense, that is, there exists a $b \in \Hom_{\fS_n}(\mathcal{B}_{t-1}(\fS_n), \tilde{V}^r_n)$ such that \[l^r_1-l^r_2 = \Pi_n^r \circ b \circ \partial_t.\] This defines an equivalence relation $\equiv$ on twisted $z^r$-cycles. We denote the set formed by these equivalence classes by $H^{t,z^r}(\fS_n,V_n^r)$ and the equivalence class of $l^r$ by $[l^r]$. Note that when $z^r = 0$, $H^t(\fS_n,V_n^r)\coloneq H^{t,z^r}(\fS_n,V_n^r)$ is the classical cohomology group and for any $z^r$, $H^{t,z^r}(\fS_n,V_n^r)$ is a $H^t(\fS_n,V_n^r)$-torsor.
\end{definition}

Since the data of $z^r$ is contained in the data of a twisted $z^r$-cycle, the above definition is NOT independent of the cover. The nice lift construction that we describe in \S \ref{subsection:nicelift} depends on this fact. 

%
%

In Definition~\ref{def:periodicity}, we describe periodic elements of $\Hom_{\fS_n}(\mathcal{B}_{t+1}(\fS_n), \tilde{V}^r_n)$ where a period is a $r$-length sequence of nonnegative which keeps track of periodicity in each of the $r$ component (recall $\tilde{V}^r_n = \bigoplus_{i=1}^{r} M(m_i)$). To be able to state the main theorem for $\sharp$-filtered $\FI$-modules we need certain maps on $d$ length sequences which depend on the tuple $\vec{m} = (m_1, m_2, \ldots, m_d)$ appearing in the cover $\tilde{V}^d$ of $V^d$. We first define the operator $\ddH$. \begin{equation}
\ddH (b_1,b_2) \coloneq
\begin{cases}
v_p ((b_1 - b_2)!) +1, & \text{if }b_1>b_2 \\
0, & \text{otherwise.} \label{def:deltaH}
\end{cases}
\end{equation}

\begin{definition}
\label{def:threeoperators}
Suppose $\Pi: \tilde{V} \coloneq \bigoplus_{i=1}^{d} M(m_i) \ra V$ be a $\sharp$-filtered $\FI$-module. Recall that the $d$-tuple $\vec{m} = (m_1, m_2, \ldots, m_d)$ is a part of the data of the cover $\tilde{V}$. Let $(H^{i,d})_{1 \leq i \leq d} \in \mathbb{Z}^d_{\geq 0}$ be a sequence of nonnegative integers. We define integers $H^{i,r}$ for $i \in [r]$, $r<d$ recursively by 
\begin{equation}
H^{i,r-1} \coloneq
\begin{cases}
\max(H^{i,r}, H^{r,r} + \ddH  (m_r,m_i)), & \text{if }m_r \geq m_i \\
H^{i,r}, & \text{otherwise} \label{def:H}
\end{cases}
\end{equation} This recursive definition clearly depends on $\vec{m}$ but we have suppressed the dependence here for clarity. We define the operators $\mathfrak{D}^r_{\Pi}: \mathbb{Z}^d_{\geq 0} \ra \mathbb{Z}^r_{\geq 0}$, $\mathfrak{D}_{\Pi} : \mathbb{Z}^d_{\geq 0} \ra \mathbb{Z}^d_{\geq 0}$ and $\mathfrak{I}_{\Pi} : \mathbb{Z}^d_{\geq 0} \ra \mathbb{Z}_{\geq 0}$ given by
\begin{eqnarray}
\mathfrak{D}^r_{\Pi}((H^{i,d})_{1\leq i \leq d}) &\coloneq& (H^{i,r})_{1\leq i \leq r}. \label{def:mathfrakDr}\\
\mathfrak{D}_{\Pi}((H^{i,d})_{1\leq i \leq d}) &\coloneq& (H^{i,i})_{1\leq i \leq d}, \qquad \text{  and} \label{def:mathfrakD}\\
\mathfrak{I}_{\Pi}((H^{i,d})_{1\leq i \leq d}) &\coloneq& \max_{1 \leq i \leq d} (H^{i,i} + \ddH(m_i,0)),\label{def:mathfrakI} 
\end{eqnarray} We may replace the subscript $\Pi$ by $V$ when a $\sharp$-filtered structure on $V$ is clear (or equivalently when $\vec{m}$ is clear from the context), or remove the subscript altogether when it is clear which $\sharp$-filtered $\FI$-module is under discussion.
\end{definition}

\begin{xample}
\label{example:1bound}
Suppose $\Pi \colon M(0) \oplus M(m) \to V$ be a $\sharp$-filtered $\FI$-module of length $d=2$ (see Example~\ref{example:1}). Then $\mathfrak{D}^2_{V}((0,0)) = (0,0)$, $\mathfrak{D}^1_{V}((0,0)) = \ddH(m, 0)$, $\mathfrak{D}_{V}((0,0)) = (\ddH(m, 0), 0)$ and $\mathfrak{I}_V ((0,0)) = \ddH(m, 0)$. 
\end{xample}

We also need the following map to state our theorem.

\begin{definition}[{\bf The map $R$}]
\label{def:Rmap}
Let $U$ be an $\FIsharp$-module. By Lemma~\ref{lemma:zero} in \S \ref{section:filteredFI}, $U$ admits a natural projection $\lambda_a : \cS_{+a} U \ra U$ which induces a map \[ \Hom_{\fS_{n-a}}(\mathcal{B}_{t}(\fS_{n-a}), \cS_{+a} U_{n-a}) \ra  \Hom_{\fS_{n-a}}(\mathcal{B}_{t}(\fS_{n-a}), U_{n-a}).\] 
Also we have the natural restriction map \[\Hom_{\fS_{n}}(\mathcal{B}_{t}(\fS_n), U_n) \ra \Hom_{\fS_{n-a}}(\mathcal{B}_{t}(\fS_{n-a}), \cS_{+a} U_{n-a}).\] Let \[R_{t, U}^{n,a} : \Hom_{\fS_{n}}(\mathcal{B}_{t}(\fS_n), U_n) \ra \Hom_{\fS_{n-a}}(\mathcal{B}_{t}(\fS_{n-a}), U_{n-a})\] be the composition of these two maps. We allow ourself to shed some of the subscripts and superscripts when there is no risk of confusion. 
\end{definition}

\begin{remark} For a general $\FI$-module, there is no analog of the map $\lambda_a$ (as in Definition~\ref{def:Rmap}). In fact, \cite[Theorem~4.1.5]{CEF} implies that if there is a natural map $\lambda_a \colon \cS_{+a} V \to V$ for each $a$ then $V$ must admit a $\FIsharp$-structure.
\end{remark}

\begin{remark} \label{remark:RcommuteswithSequential}
Let $\Pi^1: \bigoplus_{i=1}^{d_1} M(m_i) \ra V$ and $\Pi^2: \bigoplus_{k=1}^{d_2} M(n_k) \ra W$ be $\sharp$-filtered $\FI$-modules. And let $\phi : V \ra W$ be a sequential map (see Definition~\ref{def:sequential}). Then $R_t$ commutes with the corresponding map of covers $\tilde{\phi}_n$, that is, if $L \in \Hom_{\fS_{n}}(\mathcal{B}_{t}(\fS_n), \bigoplus_{i=1}^{d_1} M(m_i)_n)$ then we have \[R_t(\tilde{\phi}_{n} \circ L) = \tilde{\phi}_{n-a} \circ R_t(L).\qedhere \]
\end{remark}

With the notations of Definition~\ref{def:twistedcyles}, we show in Claim~\ref{claim:welldefinedness}, that the map \begin{equation}\mathfrak{R}_{t, V^d}^{n, a, z^d}: H^{t,z^d}(\fS_n,V_n^d) \ra H^{t,R_{t+1}(z^d)}(\fS_{n-a},V_{n-a}^d) \end{equation} given by $[l^d] \mapsto [\Pi^d_{n-a} \circ R_t (\tilde{l}^d)]$ (here $\tilde{l}^d$ is any "nice lift" of $l^d$ as defined in \S \ref{subsection:nicelift}) is well-defined if $a$ is divisible by a sufficiently large power of $p$. We now state the main theorem for $\sharp$-filtered $\FI$-modules.

\begin{main}[{\bf The main theorem for $\sharp$-filtered $\FI$-modules}] \label{thm:filteredstabilityrange} Let $\Pi^d: \tilde{V}^d \coloneq \bigoplus_{i=1}^{d} M(m_i) \ra V^d$ be a $\sharp$-filtered $\FI$-module of length $d$ as in Definition~\ref{def:twistedcyles} and let $D = \max_{i \in [d]} m_i$. Let $\SQ$ be the sequence of length $d$ consisting of zeros. If $p^{\mathfrak{I}_{V^d} \circ \mathfrak{D}_{V^d}(\SQ)} \mid a$ and $n-a \geq 2 (t+d-1) + D$, then \[\mathfrak{R}^{n,a,0}_{t,V^d} : H^t(\fS_n, V^d_n) \ra H^t(\fS_{n-a}, V^d_{n-a})\] is an isomorphism.
\end{main}

We also extend Theorem~\ref{thm:filteredstabilityrange} to $H^t(\fS_n, V^d_n)$-torsors.

\begin{main}[{\bf Periodicity of  torsors}] \label{thm:twistedfilteredstabilityrange}
Let $\SQ \in \mathbb{Z}^d_{\geq 0}$ and let  $z^d$ be $\SQ$-periodic (periodicity is defined later in Definition~\ref{def:periodicity}). If $p^{\mathfrak{I}_{V^d}\circ \mathfrak{D}_{V^d}(\SQ)} \mid a$ and $n-a \geq 2 (t+d-1) + D$, then \[\mathfrak{R}_{t, V^d}^{n, a, z^d}: H^{t,z^d}(\fS_n,V_n^d) \ra H^{t,R_{t+1}(z^d)}(\fS_{n-a},V_{n-a}^d)\] is an isomorphism (as sets) provided $H^{t,z^d}(\fS_n,V_n^d)$ is nonempty.
\end{main}

In Lemma~\ref{lemma:boundonM}, we give bounds on the period depending only on the degree $D$ and the periodicity of $z^d$. In particular, such a bound on the period in Theorem~\ref{thm:filteredstabilityrange} is $p^{2 D}$. In example~\ref{example:1} we show that the smallest period for the cohomology groups of a $\sharp$-filtered $\FI$-module of length $2$ could be an arbitrarily large power of $p$.

\subsection{Statements of the main theorem and its generalizations}
\label{subsection:themainthmintro}
We keep the assumption that $\bk$ is a field of positive characteristic.

Let $V$ be a finitely generated $\FI$-module over $\bk$ generated in degree $\leq D$. By Theorem~\ref{thm:structure}, there exists a resolution  
\[0 \ra V \ra J^0 \ra J^1 \ra \ldots \ra J^N \ra 0\] of $V$ with $\sharp$-filtered $\FI$-modules $J^i$, which is
exact in high enough degree; say $n \geq C$. Let $\phi^i : J^i \ra J^{i+1}$ and $\iota: V \ra J^0$ be the maps given by the exact sequence above. Note that by Theorem~\ref{thm:structure}, the maps $\phi^i$ are sequential. For each $n \geq 0$, let $E^{\bullet, \bullet}(n)$ be a double complex spectral sequence supported in columns $0 \leq x \leq N$, with the following data on page $0$: 
\begin{align}  \tag{$E$} E^{x,y}(n) &= \Hom_{\fS_{n}}(\mathcal{B}_{y}(\fS_n), J^x_n), \label{eqn:E}\\
_{\ra}d^{x,y}(n) &: E^{x,y}(n) \ra E^{x+1, y}(n),  \qquad \text{induced by } \phi^x \text{ and,}\nonumber \\
{_{\uparrow}d}^{x,y}(n) &: E^{x,y}(n) \ra E^{x, y+1}(n), \qquad \text{induced by } \partial_{y+1}. \nonumber
\end{align} In \S\ref{subsection:proofofmain}, we analyze the spectral sequence \eqref{eqn:E} and define maps \[\mathfrak{R}^{x,y}_{\infty}(n,a): {_{\uparrow}E}^{x,y}_{\infty}(n) \ra {_{\uparrow}E}^{x,y}_{\infty}(n-a)\] where ${_{\uparrow}E}^{x,y}_{\infty}$ is the $\bk$-vector space at the position $(x,y)$ of the page $\infty$ of the vertically oriented spectral sequence (where we take the homology with respect to vertical maps to get the first page) as above. This leads us to our main theorem for finitely generated $\FI$-modules.

\begin{main}[{\bf The main theorem for finitely generated $\FI$-modules}] \label{thm:maintheorem}
Let $V$ be a finitely generated $\FI$-module generated in degree $\leq D$. Then for $n \geq C$, $H^t(\fS_n, V_n)$ admits a filtration of length $N+1$ with $({_{\uparrow}E}^{x,y}_{\infty}(n))_{x+y=t, 0 \leq x \leq N}$ as graded pieces (here $N \leq D$ and $C$ are constants as in Theorem~\ref{thm:structure}). And there are constants $M_{\infty}^t$ and $\SD_{\infty}^t$ such that if $p^{M_{\infty}^t} \mid a$ and $n-a \geq \max\{\SD_{\infty}^t, C\}$ then the map $\mathfrak{R}^{x,y}_{\infty}(n,a): {_{\uparrow}E}^{x,y}_{\infty}(n) \ra {_{\uparrow}E}^{x,y}_{\infty}(n-a)$ is an isomorphism. In particular, $\dim{H^t(\fS_n, V_n)}$ is eventually periodic in $n$ with period $p^{M_{\infty}^t}$.
\end{main}
We provide an algorithm to calculate the stable range $\SD_{\infty}^t$ and the period $p^{M_{\infty}^t}$ (see Remark~\ref{remark:stabilityalgomain}). Lemma~\ref{lem:bound-main} provides the following estimates:
\begin{align*}
M_{\infty}^t & \leq \min\{(t+3)D, \max\{2 D, D(D+1)/2 \} \}, \\
\SD_{\infty}^t & \leq 2(t+ \max_x d_x - 1) +D.
\end{align*} where $d_x$ is the length of the $\sharp$-filtered $\FI$-module $J^x$. By Remark~\ref{rem:dimension-generation}, we can replace $D$ in the bounds above by $\chi(V)$ (recall that $\chi(V) \leq D$).

At the expense of increasing the period slightly, we construct isomorphism  $H^t(\fS_n, V_n) \ra H^t(\fS_{n-a}, V_{n-a})$ in Theorem~\ref{thm:vectormain} preserving the filtration as in the Theorem~\ref{thm:maintheorem}.

 In \S\ref{subsection:generalization}, we generalize our main theorem to a complex of finitely generated $\FI$-module. Consider an arbitrary complex of finitely generated $\FI$-modules  \[0 \ra V^0 \ra V^1 \ra \ldots \ra V^x \ra \ldots\] with  differential $\delta$. For each $n \geq 0$, define a double complex spectral sequences $E^{\bullet, \bullet}(n)$ with the following data on page $0$: 
\begin{align}  \tag{$\vec{E}$} E^{x,y}(n) &= \Hom_{\fS_{n}}(\mathcal{B}_{y}(\fS_n), V^x_n), \label{eqn:vecE}\\
_{\ra}d^{x,y}(n) &: E^{x,y}(n) \ra E^{x+1, y}(n),  \qquad \text{induced by } \delta^x \text{ and,}\nonumber \\
{_{\uparrow}d}^{x,y}(n) &: E^{x,y}(n) \ra E^{x, y+1}(n), \qquad \text{induced by } \partial_{y+1}. \nonumber
\end{align} We deduce the generalization to the main theorem by analyzing the spectral sequence \eqref{eqn:vecE}. 

\begin{main} [{\bf Periodicity of the cohomology of the $\FI$-complexes}]
\label{thm:generalizedmain}
Let $V^{\bullet}$ be a complex of finitely generated $\FI$-modules and $E^{\bullet, \bullet}$ be the corresponding spectral sequence as defined above. Let $r \in \mathbb{N} \cup \{\infty\}$. Then, if $p^{\vec{M}_{r}^{x,y}} \mid a$ and $n-a \geq \max(\vec{\SD}_{r}^{x,y}, C^{x,y})$ then the map $\vec{\mathfrak{R}}^{x,y}_{r}(n,a): {_{\uparrow}E}^{x,y}_{r}(n) \ra {_{\uparrow}E}^{x,y}_{r}(n-a)$ is an isomorphism.
\end{main}

Equations~\ref{equ:vectorrec1} and \ref{equ:vectorrec2} (that are stated in \S \ref{subsection:generalization}) provide a recursive way to calculate the period $p^{\vec{M}_{r}^{x,y}}$ and the stable range $\max(\vec{\SD}_{r}^{x,y}, C^{x,y})$. Remark~\ref{rem:bound-general} gives an estimate on these quantities. The full strength of this theorem is used for the application to configuration spaces. Also, see some of the interesting consequences (Corollary~\ref{corollary:kerperiodicity} and Corollary~\ref{corollary:connectinghomperiodicity}) of Theorem~\ref{thm:generalizedmain}.

\subsection{Statements of the results on unordered configuration spaces}
\label{subsection:applicationintro}
Let $\mathcal{M}$ be a manifold. 
As noted in \cite[\S 6]{CEF} or \cite[\S 4]{CEFN}, the configuration space of $\mathcal{M}$ is a \emph{co-$\FI$-space}, that is, a contravariant functor from $\FI$ to topological spaces: For any finite set $S$, let $\Conf_S(\mathcal{M})$ denote the space $\Inj(S,\mathcal{M})$ of injections $S\hookrightarrow \mathcal{M}$. An inclusion $f\colon S\hookrightarrow T$ induces a restriction map $f^* : \Conf_T(\mathcal{M})\to \Conf_S(\mathcal{M})$ giving  $\Conf(\mathcal{M})$ a \emph{co-$\FI$-space} structure.

When $S=[n]$, the space of injections $[n]\hookrightarrow \mathcal{M}$ can be identified with the classical configuration space $\Conf_n(\mathcal{M})$ of ordered $n$-tuples of distinct points in $\mathcal{M}$: \[\Conf_n(\mathcal{M})\coloneq \{(P_1, P_2, \ldots, P_n) \in \mathcal{M}^n\ \big|\ P_i \neq P_j\}\] Since cohomology is contravariant, the functor taking $S$ to $H^m(\Conf_S(\mathcal{M}),\bk)$ is an $\FI$-module $H^m(\Conf(\mathcal{M}),\bk)$ over $\bk$. Under certain mild conditions on $\mathcal{M}$, this $\FI$-module is finitely generated.

\begin{theorem}[{\cite[Theorem~E]{CEFN}}]
\label{thm:configurationsfg}
Let $\bk$ be a Noetherian ring, and let $\mathcal{M}$ be a connected orientable manifold of dimension $\geq 2$ with the homotopy type of a finite $\CW$ complex (e.g.\ $\mathcal{M}$ compact). Then, for any $m\geq 0$, the $\FI$-module $H^m(\Conf(\mathcal{M}),\bk)$ is finitely generated.
\end{theorem}

The classical configuration space of unordered $n$-tuples of distinct points in $\mathcal{M}$ is defined as:\[\conf_n(\mathcal{M})\coloneq \{(P_1, P_2, \ldots, P_n) \in M^n \ \big|\ P_i \neq P_j\}/ \fS_n.\] In the past, there have been several investigations on the stability of the cohomology groups $H^t(\conf_n(\mathcal{M}), \bk)$ but all of them required either as assumption on the manifold $\mathcal{M}$ (restricting to an open or a punctured manifold  or an odd-dimensional manifold), or an assumption on the characteristic of the field $\bk$ (restricting to $\mathbb{F}_2$ or a field of characteristic $0$), or an assumption on the boundary (restricting to manifolds with nonempty boundary); see -- \cite{BCT}, \cite{Nap}, \cite{C}, \cite{RW}, \cite{mcduff}, \cite{segal}, \cite{jeremy}. For a more complete history of the subject, see \cite{icm}.

In this paper, we assume that the manifold satisfy the mild assumptions from Theorem~\ref{thm:configurationsfg} and we allow the field to be of arbitrary positive characteristic. In particular, our result is new when the manifold is even-dimensional and the field is different from $\mathbb{F}_2$ and unique in the sense that we show periodicity and not stability (see Remark~\ref{remark:exampleBi}). Our method is similar in spirit to Church's method in \cite{C} (which we later updated in \cite{CEF} and \cite{CEFN}) but we have to deal with nonexactness of the functor $H^0(\fS_n, .)$ and that forces us to make use of the full strength of Theorem~\ref{thm:generalizedmain}. Now we state our result.

\begin{main}[{\bf Periodicity of cohomology of unordered configuration spaces}] \label{thm:confmain}
Let $\bk$ be a field of characteristic $p>0$ and let $\mathcal{M}$ satisfies the hypothesis of Theorem~\ref{thm:configurationsfg}. There exist constants $\vec{M}^t_{\infty}$, $\vec{\SD}^t_{\infty}$ and $C^t$ such that \[\dim_{\bk}{H^{t}(\conf_n(\mathcal{M}),\bk)} = \dim_{\bk}{H^{t}(\conf_{n-a}(\mathcal{M}), \bk)}\] whenever $p^{\vec{M}^t_{\infty}} \mid a$ and $n-a \geq \max(\vec{\SD}^t_{\infty}, C^t)$. (See \eqref{eqn:confperiod}, \eqref{eqn:confstability}and \eqref{eqn:confCt} for the definitions of the constants $\vec{M}^t_{\infty}$, $\vec{\SD}^t_{\infty}$ and $C^t$.)
\end{main}

Remark~\ref{rem:bound-config} in \S \ref{section:application} implies that $\vec{M}^t_{\infty} $ above can be taken to be  $(t+3)(2 t +2)$. 

\begin{remark}
\label{rem:palmer}
Recent work of Cantero and Palmer (\cite[Corollary~E]{palmer}) implies that if $p$ is an odd prime and the Euler characteristic $\chi(\cM)$ of $\cM$ is nonzero then, $H^t(\conf_n(\cM), \bk)$ is eventually periodic with period $p^{v_p(\chi(\cM))+1}$ where $v_p({-})$ is the $p$-adic valuation. Note that our results give bounds that are independent of $\chi(\cM)$ but depend only on $t$. 
\end{remark}


\begin{remark}
\label{remark:exampleBi}
The easiest example where period is not $1$ is the $2$-sphere, where we have \cite[Theorem~1.11]{Bi}: \[H_1(\conf_n(S^2), \mathbb{Z}) = \mathbb{Z} / (2n-2) \mathbb{Z}\] and by the universal coefficient theorem \[H^1(\conf_n(S^2), \mathbb{Z} / p \mathbb{Z}) = \begin{cases}
\mathbb{Z} / p \mathbb{Z}, & \text{if } p \mid 2n-2 \\
0, & \text{otherwise}
\end{cases}\] Thus when $p \neq 2$, the smallest period is $p$.
\end{remark}

\section{Theory of \texorpdfstring{$\sharp$}{sharp}-filtered \texorpdfstring{$\FI$}{FI}-modules over Noetherian rings}
\label{section:filteredFI}

The aim of this section is to prove Theorem~\ref{thm:structure} and see some of its immediate consequences. Except in Lemma~\ref{lemma:zero}, which holds over a general commutative ring, $\bk$ is assumed to be a Noetherian ring throughout \S \ref{section:filteredFI}.

\begin{definition}[{\bf The set $D_{m,n}$ and the coset representatives $\gamma_f$}]
\label{def:Dandcosetrepresentatives}

For $m \leq n$, let $D_{m,n}$ be the set of subsets of $[n]$ of size $m$. $\fS_n$ acts naturally on $D_{m,n}$. For each $f \in D_{m,n}$, we have a unique element $\gamma_f \in \fS_n$ such that $\gamma_f|_{[m]}$, $\gamma_f |_{[n] \setminus [m] }$ are order preserving and $\gamma_f([m])=f$. Then $\bk[\fS_n] \cong \bigoplus_{f \in D_{m,n}} \bk[\fS_{m} \times \fS_{n-m}]\gamma_{f}^{-1}$ as a $\bk[\fS_{m} \times \fS_{n-m}]$-module. 

We do not distinguish $D_{m,n}$ from the set of order preserving injections $f: [m] \hookrightarrow [n]$. We adopt the conventions $f(0) \coloneq 0$ and $f(m+1) \coloneq n+1$.
\end{definition}

The definition above is used extensively only in section \S \ref{section:periodicityofinv} but we give it here because of its use in the following lemma.

\begin{lemma} \label{lemma:zero}
Let $m\geq 0$ and $W$ be a $\bk [\fS_m]$-module over a commutative ring $\bk$. Then the following hold:

\begin{enumerate}
\item Let \[0 \ra U^1 \ra M(W) \ra U^2 \ra 0 \] be an exact sequence of $\FI$-modules with $U^1$ generated in degree $m$. Then $U^1$ and $U^2$ are $\FIsharp$-modules. More precisely, \[U^i = M(U^i_m), \qquad \text{for } i \in \{1,2\}. \]

\item \[\cS_{+1} M(W) = M(W) \oplus M(\Res^{\fS_n}_{\fS_{n-1}} W).\]

\item If $W$ is finitely generated as a $\bk[\fS_m]$-module, then for any $a$, $\cS_{+ a} M (W)$ is $\sharp$-filtered with $M(W)$ as a direct summand.
\end{enumerate}
\end{lemma}

\begin{proof} \begin{enumerate}
\item Proof follows from the exactness of $\Ind_H^G$.

\item It is enough to show that \[\Res_{\fS_n}^{\fS_{n+1}} \Ind_{\fS_m \times \fS_{n+1-m}}^{\fS_{n+1}} W \boxtimes \bk = (\Ind_{\fS_m \times \fS_{n-m}}^{\fS_{n}} W \boxtimes \bk) \oplus (\Ind_{\fS_{m-1} \times \fS_{n-m+1}}^{\fS_{n}} (\Res_{\fS_{m-1}}^{\fS_{m}} W) \boxtimes \bk).\] Any element of the module on left may be written uniquely as \[\sum_{f \in D_{m,n+1}} \gamma_f \otimes a_f, \qquad a_f \in W \text{ for each } f \in D_{m,n+1}.\] Note that $D_{m, n+1} = D_{m,n} \sqcup (D_{m,n+1} \setminus D_{m,n})$ and $D_{m,n+1} \setminus D_{m,n}$ is naturally isomorphic to $D_{m-1, n}$. This induces the decompostion \[\sum_{f \in D_{m,n+1}} \gamma_f \otimes a_f = \sum_{f \in D_{m,n}} \gamma_f \otimes a_f + \sum_{f \in D_{m-1,n}} \gamma_f \otimes a_f \] finishing the proof.

\item By the exactness of the shift functor and the natural isomorphism $\cS_{+(a+b)} V \cong \cS_{+a} \cS_{+b} V$, it is enough to show part $3$ for the case $a=1$. By part $2$, it is enough to show that $M(W)$ is filtered. Now by exactness of $\Ind_H^G$ we are reduced to showing that $W$ admits a filtration with graded pieces which are singly generated $\bk[\fS_m]$-modules; which follows because $W$ is finitely generated. \qedhere
\end{enumerate}
\end{proof}

\begin{remark}
The above lemma implies that the shift functor $\cS_{+}$ preserves sequentialness (see Definition~\ref{def:sequential}): if $\delta: V \ra W$ is sequential and $a \geq 0$, then $\cS_{+a} \delta$ is sequential. 
\end{remark}

\begin{proof}[{\bf Proof of Theorem~\ref{thm:structure}}]
We first show that $\cS_{+a} V$ is $\sharp$-filtered when $a$ is large enough. Our proof is by induction on $\max_i m_i$ and the number of times $\max_i m_i$ appears in the tuple $(m_1, \ldots, m_d)$. We will refer to this tuple as the degree structure of $\Pi$ (see Definition~\ref{def:finitegeneration}). The base case is trivial because if $\max_i m_i= 0$ and $0$ appears only once in the degree structure of $\Pi$, then for large enough $a$, $\cS_{+a} V$ is a torsion free $\FI$-module singly generated in degree $0$; and such an $\FI$-module must be either $0$ or $M(0)$. 

By Lemma~\ref{lemma:SaMd}, $\cS_{+a} \tilde{V} = {{\bigoplus}_{i=1}^{d} (M(m_{i})\bigoplus Q_{i,a}) }$ where $Q_{i,a} =  \bigoplus_{j=1}^{d_i} M(m_{i,j})$ and $m_{i,j} <m_i$ for each $j$. We pick an $i$ so that $m_i$ is maximum. The natural
projection $\cS_{+ a} \tilde{V} \twoheadrightarrow M (m_i)$ induces
the following natural commutative diagram (Figure~\ref{Fig:2}) with exact rows and columns. \begin{figure}[h]
\centering
\begin{tikzcd}
0 \ar{r} & \tilde{U}^a \ar{r} \ar{d}{\Pi^a} & \cS_{+a} \tilde{V} \ar{r} \ar{d}{\cS_{+a}\Pi} & M(m_i) \ar{r} \ar{d}{\varphi_a} & 0 \\
0 \ar{r} & U^a \ar{r} \ar{d} & \cS_{+a} V \ar{r} \ar{d} & A^a \ar{r} \ar{d} & 0 \\
& 0 & 0 & 0 
\end{tikzcd}
\caption{} \label{Fig:2}
\end{figure} Here $\Pi^a$ is the restriction of $\cS_{+a} \Pi$ to $U^a$. The $\bk$-modules $(\ker \varphi_a)_{m_i} \subset \bk[\fS_{m_i}]$ are increasing in $a$ and hence must stabilizes for $a$ large enough (see Proof of Theorem~A in \cite{CEFN} for a proof that these modules are increasing). Fixing such an $a$, $\ker \varphi_a$ is generated in degree $m_i$ and hence by part 1 of Lemma~\ref{lemma:zero}, $A^a$ is of the
form $M (W)$ for some quotient $W$ of $\bk[\fS_{m_i}]$. So we have an exact sequence: \[ 0 \ra U^a \ra \cS_{+a}V \ra M(W) \ra 0\] with the cardinality of $m_i$ in the degree structure of $\Pi^a$ one less than the cardinality of $m_i$ in the degree structure of $\Pi$ (or $\cS_{+a} \Pi$). Hence by induction, $\cS_{+ b} U^a$ is $\sharp$-filtered for large enough $b$. Shifting
the exact sequence above by $b$ yields:\[0 \ra \cS_{+b} U^a \ra \cS_{+(a+b)}V \ra \cS_{+b} M(W) \ra 0.\] By part $3$ of Lemma~\ref{lemma:zero}, $\cS_{+ b} M (W)$ is $\sharp$-filtered and hence $\cS_{+a} V$ is $\sharp$-filtered for $a$ large, completing the first part of the proof. For the second part we fix such an $a$ and proceed as follows.

Let $Q$ be the cokernel of the map $\iota \coloneq X_a(V)$ (Definition~\ref{def:torsionsubmodule}). Then we have a commutative diagram (Figure~\ref{Fig:3}) with exact columns, exact first row and the second row exact in large enough degree; say $C_1$. Recall, from Definition~\ref{def:torsionsubmodule}, that the failure of exactness in certain finitely many degrees is coming from the torsion submodule of $V$ (which vanishes in high enough degree). 
\begin{figure}[h]
\centering
\begin{tikzcd}
0 \ar{r} & \tilde{V} \ar{r}{\tilde{\iota}} \ar{d}{\Pi} & \cS_{+a} \tilde{V} \ar{r}{\tilde{\phi}'} \ar{d}{\cS_{+a}\Pi} & \tilde{Q} \ar{r} \ar{d}{\Xi} & 0 \\
0 \ar{r} & V \ar{r}{\iota} \ar{d} & \cS_{+a} V \ar{r}{\phi'} \ar{d} & Q \ar{r} \ar{d} & 0 \\
& 0 & 0 & 0 
\end{tikzcd}
\caption{} \label{Fig:3}
\end{figure}
It follows from Lemma~\ref{lemma:SaMd} that $\tilde{Q}$ is generated in degree $< D$. By induction on degree, the theorem is true for $\Xi \colon \tilde{Q} \ra Q$, that is, we have a commutative diagram as in Figure~\ref{Fig:4} which satisfies the assertions of theorem.

\begin{figure}[h]
\centering
\begin{tikzcd}
 {0}\ar{r} & \tilde{Q} \ar{r}{\tilde{\iota}'} \ar{d}{\Xi} & \tilde{K}^0 \ar{r}{\tilde{\psi}^0} \ar{d}{\Xi^0} &
 \tilde{K}^1 \ar{r}{\tilde{\psi}^1} \ar{d}{\Xi^1} & \tilde{K}^2 \ar{d}{\Xi^2} \ar{r}{\tilde{\psi}^2} & \ldots \ar{r}{\tilde{\psi}^{N'-1}} & \tilde{K}^{N'} \ar{d}{\Xi^{N'}} \ar{r} & 0\\ 
 {0}\ar{r} & Q \ar{r}{\iota'} \ar{d} & K^0\ar{r}{\psi^0} \ar{d} &
  K^1 \ar{r}{\psi^1} \ar{d} & K^2 \ar{d} \ar{r}{\psi^2} & \ldots \ar{r}{\psi^{N'-1}} & K^{N'} \ar{d} \ar{r} & 0\\
   & 0 & 0 &
    0 & 0 & & 0
 \end{tikzcd}
 \caption{} \label{Fig:4}
 \end{figure}

The theorem for $\Pi: \tilde{V} \ra V$ then follows by concatenating the two diagrams (Figure~\ref{Fig:3} and \ref{Fig:4}), setting $J^0 \coloneq \cS_{+a}V$, $J^i\coloneq K^{i-1}$, $\phi^i\coloneq \psi^{i-1}$ for $i>0$ and noting that $\phi^0 \coloneq \iota' \circ \phi'$ is sequential.
\end{proof}

\begin{remark}
\label{remark:structuregeneralization} The proof of Theorem~\ref{thm:structure} shows that a resolution with $\sharp$-filtered $\FI$-modules can be constructed for complexes of $\FI$-modules (in \S\ref{subsection:generalization}, we will need a version of Theorem~\ref{thm:structure} for complexes of finitely generated $\FI$-modules). Consider an arbitrary complex \[0 \ra V^0 \ra V^1 \ra \ldots \ra V^x \ra \ldots\] of finitely generated $\FI$-modules with differential $\delta$ and assume that $V^x = 0$ for $x > N$.

Note that we may construct free finitely generated $\FI$-modules $\tilde{V}^x$ and surjections $\Pi^x : \tilde{V}^x  \ra V^x$ such that the diagram in Figure~\ref{Fig:5}  commutes. Moreover, the maps $\tilde{\delta}^x$ may be assumed to be sequential: by induction on $x$ pick finite generating set $G^x$ of $V^x$ containing the image $\delta^{x-1} (G^{x-1})$. These $G^x$ define $\tilde{V}^x$ and make $\tilde{\delta}^x$ sequential. 
\begin{figure}[h]
 \centering
\begin{tikzcd}
  {0}\ar{r} & \tilde{V}^0\ar{r}{\tilde{\delta}^0} \ar{d} &
  \tilde{V}^1 \ar{r}{\tilde{\delta}^1} \ar{d} & \tilde{V}^2 \ar{d} \ar{r}{\tilde{\delta}^2} & \ldots\\ {0}\ar{r} & V^0\ar{r}{\delta^0} \ar{d} &
   V^1 \ar{r}{\delta^1} \ar{d} & V^2 \ar{d} \ar{r}{\delta^2} & \ldots\\
    & 0 &
     0 & 0 & 
     \end{tikzcd}
     \caption{} \label{Fig:5}
\end{figure}
 Let $\tilde{V}^x$ be generated in degree $D_x$. Unlike Theorem~\ref{thm:structure}, we may not assume that $D_x$ is the smallest so that $V^x$ is generated in degree $\leq D_x$. But, if $D$ is large enough so that each $V^x$ is generated in degree $\leq D$ then we may assume that $D_x \leq D$ for each $0 \leq x \leq N$.
 
 The shift functor $\cS_{+}$ preserves the sequentialness. Hence by applying the construction in Theorem~\ref{thm:structure} simultaneously for each $x$ (that is, choosing $a$ large enough so that for each $x$, $\cS_{+a}V^x$ is $\sharp$-filtered and repeating the process with $Q^x \coloneq \frac{\cS_{+a}V^x}{V}$), we obtain resolutions \[0 \ra V^x \ra J^{0,x} \ra J^{1,x} \ra \ldots \ra J^{N_x , x} \ra 0\] exact in high enough degree (say $n \geq C_x$) making the diagram in Figure~\ref{Fig:6} commutes and such that the maps $\delta^{y,x}$, $0 \leq y \leq N_x$, $0 \leq x \leq N$ are sequential.

\begin{figure}[h]
\centering
\begin{tikzcd}
{} & \ldots &
     \ldots & \ldots & \\
 {0}\ar{r} & J^{1,0} \ar{r}{\delta^{1,0}} \ar{u} &
  J^{1,1} \ar{r}{\delta^{1,1}} \ar{u} & J^{1,2} \ar{u} \ar{r}{\delta^{1,2}} & \ldots\\ 
 {0}\ar{r} & J^{0,0}\ar{r}{\delta^{0,0}} \ar{u} &
 J^{0,1} \ar{r}{\delta^{0,1}} \ar{u} & J^{0,2} \ar{u} \ar{r}{\delta^{0,2}} & \ldots\\ {0}\ar{r} & V^0\ar{r}{\delta^0} \ar{u}{\iota^0} &
  V^1 \ar{r}{\delta^1} \ar{u}{\iota^1} & V^2 \ar{u}{\iota^2} \ar{r}{\delta^2} & \ldots\\
   & 0 \ar{u} &
    0 \ar{u} & 0 \ar{u} & 
 \end{tikzcd}
 \caption{} \label{Fig:6}
\end{figure}
As in Theorem~\ref{thm:structure}, the maps $\phi^x : J^{y,x} \ra J^{y+1,x}$, $0 \leq y \leq N_x -1$ $0 \leq x \leq N$ given by the diagram above are sequential, $N_x \leq D_x$ and $J^{y,x}$ is generated in degree $\leq D_x - y$. We also have the corresponding diagram (Figure~\ref{Fig:7}) on covers with exact columns and sequential rows. \qedhere

\begin{figure}[h]
\centering
\begin{tikzcd}
 {} & \ldots &
     \ldots & \ldots & \\
 {0}\ar{r} & \tilde{J}^{1,0}\ar{r}{\tilde{\delta}^{1,0}} \ar{u} &
  \tilde{J}^{1,1} \ar{r}{\tilde{\delta}^{1,1}} \ar{u} & \tilde{J}^{1,2} \ar{u} \ar{r}{\tilde{\delta}^{1,2}} & \ldots\\ 
 {0}\ar{r} & \tilde{J}^{0,0}\ar{r}{\tilde{\delta}^{0,0}} \ar{u} &
 \tilde{J}^{0,1} \ar{r}{\tilde{\delta}^{0,1}} \ar{u} & \tilde{J}^{0,2} \ar{u} \ar{r}{\tilde{\delta}^{0,2}} & \ldots\\ {0}\ar{r} & \tilde{V}^0\ar{r}{\tilde{\delta}^0} \ar{u}{\tilde{\iota}^0} &
  \tilde{V}^1 \ar{r}{\tilde{\delta}^1} \ar{u}{\tilde{\iota}^1} & \tilde{V}^2 \ar{u}{\tilde{\iota}^2} \ar{r}{\tilde{\delta}^2} & \ldots\\
   & 0 \ar{u} &
    0 \ar{u} & 0 \ar{u} & 
 \end{tikzcd}
 \caption{} \label{Fig:7}
\end{figure}
\end{remark}

\begin{proposition}\label{prop:boundedness}
Assume that $\bk$ is a field with positive characteristic and $V$ be a finitely generated $\FI$-module. Then, \[\limsup_{n \ra \infty} \dim{H^t(\fS_n,V_n)} < \infty.\]
\end{proposition}
\begin{proof}
As in the proof of Theorem~\ref{thm:structure} the sequence, \[0 \to V \to \cS_{+a}V \to Q \to 0
\] is exact in high enough degree, $\cS_{+a} V$ is $\sharp$-filtered and $Q$ is generated in lower degree. Hence by induction on the degree, it is enough to show the result for $\sharp$-filtered $\FI$-modules. By induction on the length of the $\sharp$-filtration, it is enough to show the result for $V=M(W)$, where $W$ is a finitely generated $\bk[\fS_m]$-module. By the Shapiro's lemma and the K{\"u}nneth formula we have, \[
H^t(\fS_n, M(W)_n) = H^t(\fS_m \times \fS_{n-m}, W \boxtimes \bk) = \bigoplus_{a+b=t} H^a(\fS_m, W)\otimes H^b(\fS_{n-m}, \bk).\] Now it is proved in \cite{N} that $H^b(\fS_n , \bk) \cong H^b(\fS_{n-1}, \bk)$ if $n>2 b$. Hence $\dim{H^b(\fS_{n-m}, \bk)}$ is bounded in $n$, completing the proof. \end{proof}


\section{Periodicity of invariants} \label{section:periodicityofinv}

In this section we prove our main results on the periodicity of the cohomology groups $H^t(\fS_n ,V_n)$ for a finitely generated $\FI$-module $V$. By Theorem~\ref{thm:structure}, $V$ admits a resolution with $\sharp$-filtered $\FI$-module. So we can expect that a spectral sequence argument should reduce the problem to showing the periodicity when $V$ is a $\sharp$-filtered $\FI$-module. This is discussed in detail in \S \ref{subsection:themainthmintro}. With the notations of the aforementioned section, let $l \in \Hom_{\bk[\fS_n]}(\mathcal{B}_t(\fS_n), V^d_n)$ be a zero cycle (a classical cocycle) then it admits "a nice lift" $\tilde{l} \in \Hom_{\bk[\fS_n]}(\mathcal{B}_t(\fS_n), \tilde{V}^d_n)$ that is "periodic". \S \ref{subsection:nicelift} provides the nice lift construction, \S \ref{subsection:combinatorics} defines "periodic" and \S \ref{subsection:periodicityofnicelift} proves that "a nice lift" is "periodic". 

Our main technical result is the following: if $\tilde{l} \in \Hom_{\bk[\fS_n]}(\mathcal{B}_t(\fS_n), \tilde{V}^d_n)$ is "periodic" and $\Pi^d_n \circ \tilde{l} = 0$ then $\Pi^d_{n-a} \circ R_{t}^{n,a}(\tilde{l}) =0$ (see Definition~\ref{def:Rmap}) as long as $a$ is divisible by a high enough power of $p$. This is proved in Claim~\ref{claim:kernel} of \S \ref{subsection:periodicityofnicelift}.

 \S \ref{subsection:well-definedness} and \ref{subsection:mainfiltered} use the technical result above to show periodicity of the cohomology groups $H^t(\fS_n ,V_n^d)$ for a  $\sharp$-filtered $\FI$-module $V^d$. \S \ref{subsection:proofofmain} contains a spectral sequence argument that proves our main theorem for finitely generated $\FI$-modules and \S \ref{subsection:generalization} contains its generalization to complex of finitely generated $\FI$-modules.


Throughout $\bk$ is assumed to be a field of characteristic $p$. We start with some preliminary definitions and notations.

\subsection{Preliminaries}
\label{subsection:preliminaries}
 Let $\mathcal{B}(\fS_n)_{\bullet} \ra \mathbb{Z} \ra 0$ be the bar resolution of $\fS_n$ over $\bk$ and for $m \leq n$, $D_{m,n}$ and $\gamma_f$ be as in Definition~\ref{def:Dandcosetrepresentatives}.

\begin{definition}[{\bf The trace maps $\tr$ and $\Tr$}]
\label{def:trace}

We have a $\bk[\fS_m \times \fS_{n-m}]$-module isomorphism $\bk[\fS_n] \cong \bigoplus_{f \in D_{m,n}} \bk[\fS_{m} \times \fS_{n-m}]\gamma_{f}^{-1}$. The map \[\tr^m: \bk[\fS_n] \ra \bk[\fS_m \times \fS_{n-m}] \] is the $\bk[\fS_m \times \fS_{n-m}]$-module homomorphism defined by $\gamma_{f}^{-1} \mapsto 1$, $f \in D_{m,n}$. It induces the natural $\bk[\fS_m \times \fS_{n-m}]$-module homomorphism \[\Tr_{t}^{m}: \mathcal{B}_{t}(\fS_n) \ra \mathcal{B}_{t}(\fS_{m} \times \fS_{n-m})\] defined by $(\sigma_0, \sigma_1, \ldots, \sigma_t) \mapsto (\tr^{m}_n(\sigma_0), \tr^{m}_n(\sigma_1), \ldots, \tr^{m}_n(\sigma_t))$.
\end{definition}

We denote the natural inclusion $\mathcal{B}_{t}(\fS_{m} \times \fS_{n-m}) \ra \mathcal{B}_{t}(\fS_n)$ by $\iota^{m}_{t}$. Now note that we have a commutative ladder as in Figure~\ref{Fig:8} that extends the identity map $\id : \mathbb{Z} \ra \mathbb{Z}$. 
\begin{figure}[h]
\centering
\begin{tikzcd}{\ldots}\ar{r} & {\mathcal{B}_{1}(\fS_n)}\ar{r} \ar{d}{ \iota_1^{m} \circ \Tr_1^{m}} &
{\mathcal{B}_{0}(\fS_n) }\ar{r} \ar{d}{ \iota_0^{m} \circ \Tr_0^{m}} & \mathbb{Z} \ar{d}{\id} \ar{r} & 0\\ {\ldots}\ar{r} & {\mathcal{B}_{1}(\fS_n)}\ar{r} &
{\mathcal{B}_{0}(\fS_n) }\ar{r} & \mathbb{Z} \ar{r} & 0
\end{tikzcd}
\caption{} \label{Fig:8}
\end{figure}
This implies that there is a commutative ladder with $\id_t^{m} - \iota_t^{m} \circ \Tr_t^{m}$ as vertical maps, that extend the zero map $0 : \mathbb{Z} \ra \mathbb{Z}$. Thus there exist $\fS_m \times \fS_{n-m}$-equivariant homotopy maps $h_t^{m} \colon  \mathcal{B}_t(\fS_n) \to \mathcal{B}_{t+1}(\fS_n)$, $t \geq -1$ satisfying $\id_t^{m} - \iota_t^{m} \circ \Tr_t^{m} = h_{t-1}^{m} \circ \partial_t + \partial_{t+1} \circ h_t^{m} $. 

Given $\omega : \{0, 1, \ldots, t+1\} \ra \{0, 1, \ldots, t\}$ and $\epsilon: \{0, 1, \ldots, t+1\} \ra \{0,1\}$, we define an $\fS_m \times \fS_{n-m}$-equivariant map $h_{t,\omega, \epsilon}: \mathcal{B}_{t}(\fS_n) \ra \mathcal{B}_{t+1}(\fS_n)$ by \[(\sigma_0, \sigma_1, \ldots, \sigma_t) \mapsto ((\tr^m)^{\epsilon(0)}(\sigma_{\omega(0)}), (\tr^m)^{\epsilon(1)}(\sigma_{\omega(1)}), \ldots , (\tr^m)^{\epsilon(t+1)}(\sigma_{\omega(t+1)}) )\]  where $(\tr^m)^0$ is the identity map. The proof of the following lemma is classical and pointed out to us by Thomas Church. The homotopy map constructed in the proof is quite explicit  but we still emphasize the description in terms of $h_{t, \omega, \epsilon}$ because it is useful in understanding the proof of Claim~\ref{claim:one} later.

\begin{lemma} \label{lemma:homotopy}
There are $\omega_i$, $\epsilon_i$ and $a_i \in \bk$, $0 \leq i \leq t$ (all the quantities here are independent of $n$) such that $h_t = \sum_{i=0}^{t} a_i h_{t,\omega_i,\epsilon_i}$.
\end{lemma}
\begin{proof}
One may check that \[h_t(\sigma_0, \sigma_1, \ldots, \sigma_t)= \sum_{i=0}^{t} (-1)^i (\tr \sigma_0, \tr \sigma_1, \dots, \tr \sigma_i, \sigma_i, \sigma_{i+1}, \ldots, \sigma_t)\] is an explicit description of the homotopy map.
\end{proof}

\begin{remark} \label{remark:compatibility} When we need to be more precise we add a subscript and write $\tr^m_n$, $\Tr_{t,n}^{m}$, $h_{t,n}^{m}$ and $\iota^m_{t,n}$. These maps satisfy the following compatibility conditions:
\begin{align}
\tr_{n-a}^{m} &= \tr_{n}^{m} \mid _{\bk[\fS_{n-a}]} \nonumber \\
\Tr_{t,n-a}^{m} &= \Tr_{t,n}^{m} \mid _{\mathcal{B}_{t}(\fS_{n-a})} \nonumber \\
h_{t,n-a}^{m} &= h_{t,n}^{m} \mid_{\mathcal{B}_{t}(\fS_{n-a})} \nonumber \\
\iota_{t,n-a}^{m} &= \iota_{t,n}^{m} \mid_{\mathcal{B}_{t}(\fS_{n-a})} \nonumber \qedhere
\end{align}
\end{remark}

The following maps are essential to describe the nice lift construction.

\begin{definition}[{\bf The maps $\downarrow$ and $\uparrow$}]
For any $\bk[\fS_{m}]$-module $W$ and $\bk[\fS_n]$-module $P$, we denote the natural restriction isomorphism $\Hom_{\fS_{n}}(P, M(W)_n) \cong \Hom_{\fS_{m} \times \fS_{n-m}}(P, W \boxtimes \bk)$ with $\downarrow_{m}$ and its inverse (the adjunction) with $\uparrow_m$. A concrete description of these maps is as follows: Recall that $M(W)_n = \Ind_{\fS_m \times \fS_{n-m}}^{\fS_n} W \boxtimes \bk$. Thus any element $x\in M(W)_n$ may be written uniquely as $\sum_{f \in D_{m,n} }\gamma_f \otimes a_f$, $a_f \in W$ (here we identify $W \boxtimes \bk$ with $W$ for convenience). This implies that any $F \in \Hom_{\fS_{n}}(P, M(W)_n)$ has a unique description of the form $F = \sum_{f \in D_{m,n} }\gamma_f \otimes F_f$ for some maps $F_f : P \ra W$. We define the map $\downarrow_m$ by \[F \downarrow_m \coloneq F_{[m]}, \qquad \text{for any } F \in \Hom_{\fS_{n}}(P, M(W)_n).\] We also have a relation $F_f(p)=F_{[m]}(\gamma_f^{-1}p)$ for each $f \in D_{m,n}$ that let us recover $F$ from $ F \downarrow_m$. The map $\uparrow_m$ is therefore given by \[(H \uparrow_m) (p) = \sum_{f \in D_{m,n} }\gamma_f \otimes H(\gamma_f^{-1}p), \qquad \text{for any }H \in \Hom_{\fS_{m} \times \fS_{n-m}}(P, W).\] 
When we need to be more precise we add superscripts and write $\downarrow^n_m$ and $\uparrow^n_m$.
\end{definition}

\begin{definition}[{\bf The transfer map $T$}]
\label{def:transfer}
Let $V$ be an $\FI$-module and $m \leq n$. We define the natural $\fS_n$-equivariant map \[T^{V,m}_n: M(V_{m})_n \ra V_n\] by \[\sum_{f \in D_{m_r, n}} \gamma_f \otimes a_f \mapsto \sum_{f \in D_{m, n}} f_{\star}(a_f).\] This is the degree $n$ piece of the natural map $T^{V,m} \colon M(V_m) \to V_n$ of $\FI$-modules. In particular when $0 \leq m_i \leq m_r \leq n$ the map \[T^{M(m_i),m_r}_n: M(M(m_i)_{m_r})_n \ra M(m_i)_n\] is given by \[\sum_{f \in D_{m_r, n}} (\gamma_f \otimes \sum_{g \in D_{m_i, m_r}} \gamma_g \otimes a_{f,g}) \mapsto \sum_{f \in D_{m_r, n}} \sum_{g \in D_{m_i, m_r}} \gamma_{f \circ g} \otimes a_{f,g}. \qedhere \] 
\end{definition}

\subsection{The nice lift construction} \label{subsection:nicelift}

As in \S\ref{subsection:mainfilteredintro}, consider a $\sharp$-filtered $\FI$-module $\Pi: \tilde{V} \coloneq \bigoplus_{i=1}^{d} M(m_i) \ra V$ over $\bk$ of length $d$. Let \[0 = V^0 \subset V^1 \subset \ldots \subset V^d=V\] be the $\sharp$-filtration induced by $\Pi$. For each $1 \leq r \leq d$, the map $\Pi^r: \tilde{V}^r \coloneq \bigoplus_{i=1}^{r} M(m_i) \ra V^r$ obtained by restricting $\Pi$ to $\bigoplus_{i=1}^{r} M(m_i)$ is a cover of $V^r$ and gives $V^r$ a $\sharp$-filtered $\FI$-module structure. Assume that $\frac{V^r}{V^{r-1}} \cong M(W_r)$ for some singly generated $\bk[\fS_{m_r}]$-module $W_r$. We have natural projection maps $\psi^{r}: V^r \ra M(W_r)$, $\pr^{r}: \bigoplus_{i=1}^{d} M(m_i) \ra M(m_r)$, $\pr^{\leq r}: \bigoplus_{i=1}^{d} M(m_i) \ra \bigoplus_{i=1}^{r} M(m_i)$ and $\pi^{r}: M(m_r) \ra M(W_r)$. Henceforth $T^{i,r}_n$ stands for the transfer map $T^{M(m_i), m_r}_n$ (see Definition~\ref{def:transfer}). We have the following definition.

\begin{definition}
\label{def:nicelift}
Fix an $n \geq 0$. Let $z^d \in \Hom_{\fS_n}(\mathcal{B}_{t+1}(\fS_n), \bigoplus_{i=1}^{d} M(m_i)_n)$ and $l^d \in \Hom_{\fS_{n}}(\mathcal{B}_{t}(\fS_n), V_n^d)$ be a twisted $z^d$-cycle (see Definition~\ref{def:twistedcyles}). A {\bf nice lift} of a twisted $z^d$-cycle $l^d$ is the data of an element \[\tilde{l}^d \in \Hom_{\fS_{n}}(\mathcal{B}_{t}(\fS_n), \tilde{V}_n^d)\] satisfying $ \Pi^d_n \circ \tilde{l}^d \equiv l^d$ (the two elements are equal up to a coboundary) together with the data of the following quantities (see Figure~\ref{Fig:13}) and compatibility conditions:
\begin{enumerate}
\item twists $z^{r} \in \Hom_{\fS_n}(\mathcal{B}_{t+1}(\fS_n), \bigoplus_{i=1}^{r} M(m_i)_n)$ for each $r \in [d-1]$;
\item twisted $z^{r}$-cycles $l^{r} \in \Hom_{\fS_{n}}(\mathcal{B}_{t}(\fS_n), V_n^{r})$ for each $r \in [d-1]$;
\item elements $a^r \in \Hom_{\fS_{m_r} \times \fS_{n-m_r}}(\mathcal{B}_{t}(\fS_{m_r} \times \fS_{n-m_r}), \bk[\fS_{m_r}])$ for each $r \in [d]$ such that \begin{equation}
\pr_n^r \circ \tilde{l}^r\coloneq (a^r \circ \Tr_t^{m_r} + (\pr_{n}^r \circ z^r) \downarrow_{m_r} \circ h_t^{m_r}) \uparrow_{m_r} \label{eqn:lrzr}
\end{equation} where $\tilde{l}^r \coloneq \pr_n^{\leq r} \circ \tilde{l}^d$ (in the special case $t=-1$, $\tilde{l}^d=0$);

\item elements $x^r \in \Hom_{\fS_{m_r} \times \fS_{n-m_r}}(\mathcal{B}_{t+1}(\fS_n),\ker \Pi_{m_r}^r)$ for each $r \in[d]$ such that \[\pr_{m_r}^r \circ x^r = w^r\] (see Figure~\ref{Fig:11}) where $w^r \coloneq (\pr_n^r \circ \tilde{l}^r \circ \partial_{t+1} - \pr^r_n \circ z^r) \downarrow_{m_r}$. We require $x^r$ to satisfy the following compatibility condition: if $e_j\coloneq(\sigma_{0,j},\sigma_{1,j}, \sigma_{2,j}, \ldots, \sigma_{t+1,j}) \in \mathcal{B}_{t+1}(\fS_n)$, $j \in \{1,2\}$ then \begin{equation}
x^r(e_1) = x^r(e_2) \label{eqn:xrwr}
\end{equation} whenever $w^r(e_1) = w^r(e_2)$ and $\tr^{m_r}(\sigma_{0,1}) = \tr^{m_r}(\sigma_{0,2})$;

\item elements $y^r \in \Hom_{\fS_n}(\mathcal{B}_{t+1}(\fS_n), \tilde{V}^r_n)$ for each $r \in [d]$  obtained by composing $x^r \uparrow_{m_r}$ with $T^{\tilde{V}, m_r}_n = (T^{i,r}_n)_{1 \leq i \leq r}$. This implies $\Pi^r_n \circ y^r = 0$ (see Step 3 in \S \ref{subsubsection:nicelift} and Figure~\ref{Fig:12}). Since $T^{r,r}_n$ is the identity map we also have \begin{equation}
\pr_n^r \circ y^r = w^r \uparrow_{m_r} = \pr_n^r \circ \tilde{l}^r \circ \partial_{t+1} - \pr^r_n \circ z^r; \label{eqn:random1}
\end{equation}
\item elements $c^r \in \Hom_{\fS_n} (\mathcal{B}_{t}(\fS_n), \bigoplus_{i=1}^{r-1} M(m_i)_n)$ for each $r \in [d]$ such that
\begin{equation}
(c^r \circ \partial_{t+1} (e), \pr_n^r \circ \tilde{l}^r \circ \partial_{t+1} (e)) - z^r(e) = 0 \mod{\ker \Pi_n^r} \label{eqn:random2}
\end{equation} for any $e \in \mathcal{B}_{t+1}(\fS_n)$ (in the special case $t=-1$, $c^r=0$);
\item \label{con6} $l^{r-1} = \Pi_n^{r-1} \circ c^r$;
\item $z^{r-1} = {'c}^r +  {''c}^r$ where ${'c}^r \coloneq \pr_n^{\leq (r-1)} \circ {y}^r$ and ${''c}^r \coloneq  \pr_n^{\leq (r-1) }\circ {z}^r$. \qedhere
\end{enumerate}
\end{definition}

Every twisted cycle admits a nice lift. Before proving the existence of a nice lift, we record a relation between the parameters $x^r$ and $y^r$ in Definition~\ref{def:nicelift} which will be used to prove the "periodicity" of the nice lift construction in \S\ref{subsection:periodicityofnicelift}.

\begin{lemma}
\label{lemma:xryr}
Let $e \in \mathcal{B}_{t+1}(\fS_n)$ and $i\leq r$. We have,
\[(\pr_n^i \circ y^r \downarrow_{m_i})(e)= \sum_{\substack {g \in D_{m_r,n} \\ [m_i]\subseteq g }} (\pr_{m_r}^i \circ x^r \downarrow_{m_i}^{m_r})(\gamma_{g}^{-1} e)\] In particular, both the expressions vanish if $m_i > m_r$.
\end{lemma}
\begin{proof}
By definition of the map $\uparrow$ we have, 
\begin{eqnarray}
(\pr_{m_r}^i \circ x^r \uparrow_{m_r})(e) &=& \sum_{g \in D_{m_r, n}} \gamma_g \otimes \big(\pr^i_{m_r} \circ x^r (\gamma_g^{-1}e)\big) \nonumber \\
&=& \sum_{g \in D_{m_r, n}} \gamma_g \otimes \big( \sum_{f \in D_{m_i, m_r}} \gamma_f \otimes (\pr^i_{m_r} \circ x^r \downarrow^{m_r}_{m_i}) (\gamma_f^{-1}\gamma_g^{-1}e)\big) \nonumber
\end{eqnarray} and hence by applying the transfer map we obtain,
\begin{eqnarray}
(\pr_n^i \circ y^r)(e)
&=& (T^{i,r}_n \circ (\pr_{m_r}^i \circ x^r \uparrow_{m_r}))(e) \nonumber \\
&=& \sum_{g \in D_{m_r,n}} \sum_{f \in D_{m_i, m_r}} \gamma_{g \circ f} \otimes (\pr^i_{m_r} \circ x^r \downarrow^{m_r}_{m_i}) (\gamma_f^{-1}\gamma_g^{-1}e)
\end{eqnarray}
The proof now follows by noting that the order preserving maps $g$ and $f$ satisfy $g \circ f = [m_i]$ if and only if $f=[m_i]$ and $[m_i] \subseteq g$.
\end{proof}

Next we describe how to construct a nice lift. This construction can be read after the proof of the main theorem in \S \ref{subsection:proofofmain}

\subsubsection{Construction of a nice lift of a twisted cycle} 
\label{subsubsection:nicelift}
We keep the notations of Definition~\ref{def:nicelift}.

The construction is by downward induction on $r$. Assume that we have constructed a twist $z^r \in \Hom_{\fS_n}(\mathcal{B}_{t+1}(\fS_n), \bigoplus_{i=1}^{r} M(m_i)_n)$ and a $z^r$-cycle $l^r \in \Hom_{\fS_{n}}(\mathcal{B}_{t}(\fS_n), V_n^r)$. The following steps ({\bf Step 1}--{\bf Step 4}) construct the corresponding quantities for $r-1$. The commutative diagram in Figure~\ref{Fig:9} summarizes the situation.

\begin{figure}[h]
\centering
\begin{tikzcd}[row sep=scriptsize, column sep=tiny]
{} &  & \mathcal{B}_{t+1}(\fS_n) \ar{dd}{\partial_{t+1}} \ar{rrdd}{z^r} & & & & & & \\
& & & & & & M(m_r)_n \ar{dd}{\pi^r_n}  & & \\
&  & \mathcal{B}_{t}(\fS_n) \ar{rrdd}{l^r} & & \bigoplus_{i=1}^{r} M(m_i)_n \ar{dd}{\Pi^r_n} \ar{rru}{\pr^r_n} & & & & \\
& & \bigoplus_{i=1}^{r-1} M(m_i)_n \ar{dd} \ar[crossing over]{rru} & & & & M(W_r)_n  & &\\
 & & & & V^r_n \ar{rru}{\psi^r_n} & & & &\\
& & V^{r-1}_n \ar{rru} & & & & & &
\end{tikzcd}

\caption{} \label{Fig:9}
\end{figure}

\begin{enumerate}
\item[{\bf Step 1.}] The first step of the construction is to analyze the projection $\psi^r_n \circ l^r$ of $l^r$. We start by post-composing the relation \[\id_t^{m_r} - \iota_t^{m_r} \circ \Tr_t^{m_r} = h_{t-1}^{m_r} \circ \partial_t + \partial_{t+1} \circ h_t^{m_r}\] with $(\psi^r_n \circ l^r) \downarrow_{m_r}$ to obtain: \[(\psi^r_n \circ l^r) \downarrow_{m_r} = (\psi^r_n \circ l^r) \downarrow_{m_r} \circ \ \iota_t^{m_r} \circ \Tr_t^{m_r} + (\psi^r_n \circ l^r) \downarrow_{m_r} \circ \ h_{t-1}^{m_r} \circ \partial_t + (\psi^r_n \circ l^r) \downarrow_{m_r} \circ \ \partial_{t+1} \circ h_t^{m_r}.\] This breaks $(\psi^r_n \circ l^r) \downarrow_{m_r}$ into three parts; the first factors through the trace map hence is "nice", the second is a boundary that we can get rid of while staying within the equivalence class and the third has a description in terms of $z^r$ because $l^r \circ \partial_{t+1} = \Pi^r_n \circ z^r$. Since the second term $(\psi^r_n \circ l^r) \downarrow_{m_r} \circ \ h_{t-1}^{m_r} \circ \partial_t$ is a boundary, it follows that the extension $((\psi^r_n \circ l^r) \downarrow_{m_r} \circ \ h_{t-1}^{m_r} \circ \partial_t )\uparrow_{m_r}$ is a boundary (because $\uparrow_{m_r}$ commutes with taking $\partial_t$). By projectivity of $\mathcal{B}_{t-1}(\fS_n)$, $((\psi^r_n \circ l^r) \downarrow_{m_r} \circ h_{t-1}^{m_r} \circ \partial_t )\uparrow_{m_r}$ can be lifted to a boundary $b^r \circ \partial_t \in \Hom_{\fS_{n}}(\mathcal{B}_{t}(\fS_n), V_n^r)$ (see Figure~\ref{Fig:10}). 

\begin{figure}[h]
\centering
\begin{tikzcd}[row sep=large]
\mathcal{B}_{t}(\fS_n) \ar{r}{\partial_t} & \mathcal{B}_{t-1}(\fS_n) \ar{d}{((\psi^r_n \circ l^r) \downarrow_{m_r} \circ h_{t-1}^{m_r}) \uparrow_{m_r}} \ar[swap]{dl}{b^r}\\
V^r_n \ar[swap]{r}{\psi^r_n} & M(W_r)_n \ar{r} & 0
\end{tikzcd}
\caption{} \label{Fig:10}
\end{figure}

Let ${'l}^r = l^r - b^r \circ \partial_t$. Clearly $'l^r \equiv l^r$ and we have \[(\psi^r_n \circ ('l^r)) \downarrow_{m_r} = (\psi^r_n \circ l^r) \downarrow_{m_r} \circ \iota_t^{m_r} \circ \Tr_t^{m_r} + (\psi^r_n \circ l^r) \downarrow_{m_r} \circ \partial_{t+1} \circ h_t^{m_r}.\] Let $a^r \in \Hom_{\fS_{m_r} \times \fS_{n-m_r}}(\mathcal{B}_{t}(\fS_{m_r} \times \fS_{n-m_r}), \bk[\fS_{m_r}])$ be a lift of $(\psi^r \circ l^r) \downarrow_{m_r} \circ \iota_t^{m_r}$. Then we have $(\psi^r \circ ('l^r)) \downarrow_{m_r} = \pi_{m_r}^ r \circ a^r \circ \Tr_t^{m_r} + (\pi_{n}^r \circ \pr_{n}^r \circ z^r) \downarrow_{m_r} \circ h_t^{m_r}$. We now define \begin{equation}
\pr_n^r \circ \tilde{l}^r\coloneq (a^r \circ \Tr_t^{m_r} + (\pr_{n}^r \circ z^r) \downarrow_{m_r} \circ h_t^{m_r}) \uparrow_{m_r}. \nonumber
\end{equation} Note that in the special case $t=-1$, $\pr_n^r \circ \tilde{l}^r=0$. 

To complete the description of $\tilde{l}^r$ we still need to define $\pr_n^i \circ \tilde{l}^r$ for $1 \leq i \leq r-1$. We do it by induction on $r$ in the following steps.

\item[{\bf Step 2.}] We construct the parameters $w^r$ and $x^r$ in this step. We define\[w^r \coloneq (\pr_n^r \circ \tilde{l}^r \circ \partial_{t+1} - \pr^r_n \circ z^r) \downarrow_{m_r}.\] It is clear from Step $1$ that $w^r$ maps to $\ker \pi_{m_r}^r$. Hence by projectivity of $\mathcal{B}_{t+1}(\fS_n)$, we may find $\fS_{m_r} \times \fS_{n-m_r}$-equivariant map $x^r: \mathcal{B}_{t+1}(\fS_n) \ra \ker \Pi_{m_r}^r$ such that $\pr_{m_r}^r \circ x^r = w^r$ (see Figure~\ref{Fig:11}).

\begin{figure}[h]
\centering
\begin{tikzcd}
{} & \ker \Pi_{m_r}^r \ar[hookrightarrow]{r} \ar[twoheadrightarrow]{d} & \bigoplus_{i=1}^{r} M(m_i)_{m_r} \ar[twoheadrightarrow]{d}\\
\mathcal{B}_{t+1}(\fS_n) \ar{ru}{x^r} \ar[swap]{r}{w^r} & \ker \pi_{m_r}^r \ar[hookrightarrow]{r} & M(m_r)_{m_r}
\end{tikzcd}
\caption{} \label{Fig:11}
\end{figure}

But for our proofs to go through, we require $x^r$ to be a particularly nice choice of lift of $w^r$. To make it precise, note that $G \coloneq \{(1,\sigma_1, \sigma_2, \ldots, \sigma_{t+1}) : \sigma_i \in \fS_n \forall i \in [t+1] \} $ is a basis of $\mathcal{B}_{t+1}(\fS_n)$ as a free $\fS_n$-module. Hence $\gamma_f^{-1}s$, $f\in D_{m_r,n}$, $s\in G$ forms a basis of $\mathcal{B}_{t+1}(\fS_n)$ as a free $\fS_{m_r} \times \fS_{n-m_r}$-module. So by freeness of $\mathcal{B}_{t+1}(\fS_n)$, we may assume that $x^r(\gamma_{f_1}^{-1}s_1) = x^r(\gamma_{f_2}^{-1}s_2)$ whenever $w^r(\gamma_{f_1}^{-1}s_1) = w^r(\gamma_{f_2}^{-1}s_2)$, $s_1,s_2\in G$. This allows us to say that if $e_j\coloneq(\sigma_{0,j},\sigma_{1,j}, \sigma_{2,j}, \ldots, \sigma_{t+1,j}) \in \mathcal{B}_{t+1}(\fS_n)$, $j \in \{1,2\}$ then \begin{equation}
x^r(e_1) = x^r(e_2) \nonumber
\end{equation} whenever $w^r(e_1) = w^r(e_2)$ and $\tr^{m_r}(\sigma_{0,1}) = \tr^{m_r}(\sigma_{0,2})$.

\item[{\bf Step 3.}] By extending and then  using the transfer we construct the parameter $y^r$ in this step. 

\begin{figure}[h]
\centering
\begin{tikzpicture}[commutative diagrams/every diagram]
\node (P0) at (90:2.8cm) {$\Hom_{\fS_{m_r} \times \fS_{n-m_r}}(\mathcal{B}_{t+1}(\fS_n), \tilde{V}^r_{m_r})$};
\node (P1) at (90+72:3.5cm) {$\Hom_{\fS_n}(\mathcal{B}_{t+1}(\fS_n), M(\tilde{V}^r_{m_r})_{n})$} ;
\node (P2) at (90+2*72:2.5cm) {\makebox[5ex][r]{$\Hom_{\fS_n}(\mathcal{B}_{t+1}(\fS_n), \tilde{V}^r_{n})$}};
\node (P3) at (90+3*72:2.5cm) {\makebox[5ex][l]{$\Hom_{\fS_n}(\mathcal{B}_{t+1}(\fS_n), M(m_r)_{n})$}};
\node (P4) at (90+4*72:3.5cm) {$\Hom_{\fS_{m_r} \times \fS_{n-m_r}}(\mathcal{B}_{t+1}(\fS_n), M(m_r)_{m_r})$};
\path[commutative diagrams/.cd, every arrow, every label]
(P0) edge node[swap] {$\uparrow_{m_r}$} (P1)
(P1) edge node[swap] {$T^{\tilde{V}, m_r}_n = (T^{i,r}_n)_{1\leq i \leq r}$} (P2)
(P2) edge node {$\pr^r_{n}$} (P3)
(P4) edge node {$\uparrow_{m_r}$} (P3)
(P0) edge node {$\pr^r_{m_r}$} (P4);
\end{tikzpicture}
\caption{} \label{Fig:12}
\end{figure}
Note that $x^r \uparrow_{m_r}$ maps to $M(\tilde{V}^r_{m_r})_n = \bigoplus_{i=1}^r M(M(m_i)_{m_r})_n$. Composing $x^r \uparrow_{m_r}$ with $T^{\tilde{V}, m_r}_n = (T^{i,r}_n)_{1 \leq i \leq r}$ we obtain a map $y^r : \mathcal{B}_{t+1}(\fS_n) \ra   \tilde{V}^r_n = \bigoplus_{i=1}^r M(m_i)_n$ (see commutative diagram in Figure~\ref{Fig:12}). We claim that $y^r$ has its image contained in $\ker \Pi_n^r \subset \tilde{V}^r_n$: note that by definition of transfer $y^r(e) = \sum_{g \in D_{m_r,n}} g_{\star}(x^r(\gamma_g^{-1} e))$ for any $e$ in $\mathcal{B}_{t+1}(\fS_n)$ and the claim follows because $x^r$ maps to $\ker \Pi^r_{m_r}$ 

Since $T^{r,r}_n$ is the identity map we have \begin{equation}
\pr_n^r \circ y^r = w^r \uparrow_{m_r} = \pr_n^r \circ \tilde{l}^r \circ \partial_{t+1} - \pr^r_n \circ z^r \nonumber
\end{equation}

\item[{\bf Step 4.}] We define parameters ${''l}^r$, $c^r$, ${'c}^r$, ${''c}^r$, $z^{r-1}$ and $ l^{r-1}$ in this step completing the construction by induction (see Figure~\ref{Fig:13}). By projectivity of $\mathcal{B}_{t}(\fS_n)$, there is a lift $''l^r : \mathcal{B}_{t}(\fS_n) \ra \bigoplus_{i=1}^r M(m_i)_n$ of $'l^r$ such that $\pr_n^r \circ (''l^r) = \pr_n^r \circ \tilde{l}^r$. Thus we have ${''l}^r = (c^r, \pr_n^r \circ \tilde{l}^r)$ where $c^r \in \Hom_{\fS_n} (\mathcal{B}_{t}(\fS_n), \bigoplus_{i=1}^{r-1} M(m_i)_n)$. Since ${'l}^r$ is a $z^r$ cycle, for any $e \in \mathcal{B}_{t+1}(\fS_n)$ we have,\[ (c^r \circ \partial_{t+1} (e), \pr_n^r \circ \tilde{l}^r \circ \partial_{t+1} (e)) - z^r(e) = 0 \mod{\ker \Pi_n^r}.\] Also, there exist $'c^r$,  ${''c}^r \in \Hom_{\fS_n} (\mathcal{B}_{t+1}(\fS_n), \bigoplus_{i=1}^{r-1} M(m_i)_n)$ defined by \[z^r = ({''c}^r, \pr_n^r \circ z^r)\] and \[y^r = ('c^r, \pr_n^r \circ \tilde{l}^r \circ \partial_{t+1} - \pr^r_n \circ z^r).\] It follows from the three equations above  that 
$$c^r \circ \partial_{t+1}(e) - {'c}^r(e) -{''c}^r(e) = 0 \mod{\ker \Pi_n^{r-1}}$$ Hence $l^{r-1} \coloneq \Pi_n^{r-1} \circ c^r$ is a $z^{r-1}$-cycle where $z^{r-1} \coloneq {'c}^r +{''c}^r$. By induction on $r$ there is a nice lift $\tilde{l}^{r-1}$ of $l^{r-1}$. We define $\tilde{l}^r= (\tilde{l}^{r-1}, \pr_n^r \circ \tilde{l}^r)$. This completes the construction of a nice lift $\tilde{l}^r$ of $l^r$.

\begin{figure}[h]
\centering
\begin{tikzcd}[row sep=normal, column sep=normal]
 \mathcal{B}_{t+1}(\fS_n) \ar{dd}{\partial_{t+1}} \ar{rrdd}{({''c}^r, \pr_n^r \circ z^r)}[swap]{z^r} \ar[controls={+(-2,0)and +(-1,0)}]{ddd}[swap]{{''c}^r}[swap, near end]{z^{r-1}}[swap, near start]{{'c}^r} \ar{rr}{{y}^r}[swap]{({'c}^r, w^r \uparrow_{m_r})} & & \ker \Pi^r_n \ar[hookrightarrow]{dd} & & & & \\
& & & & M(m_r)_n \ar{dd}{\pi^r_n}  & & \\
\mathcal{B}_{t}(\fS_n) \ar{rrdd}{{'l}^r} \ar{rr}{{''l}^r}[swap]{(c^r, \pr_n^r \circ \tilde{l}^r)} \ar{d}{c^r} & & \bigoplus_{i=1}^{r} M(m_i)_n \ar{dd}{\Pi^r_n} \ar{rru}{\pr^r_n} & & & & \\
\bigoplus_{i=1}^{r-1} M(m_i)_n \ar{dd} \ar[crossing over]{rru} & & & & M(W_r)_n  & &\\
& & V^r_n \ar{rru}{\psi^r_n} & & & &\\
V^{r-1}_n \ar{rru} & & & & & &
\end{tikzcd}

\caption{: This diagram does not commute.} \label{Fig:13}
\end{figure}
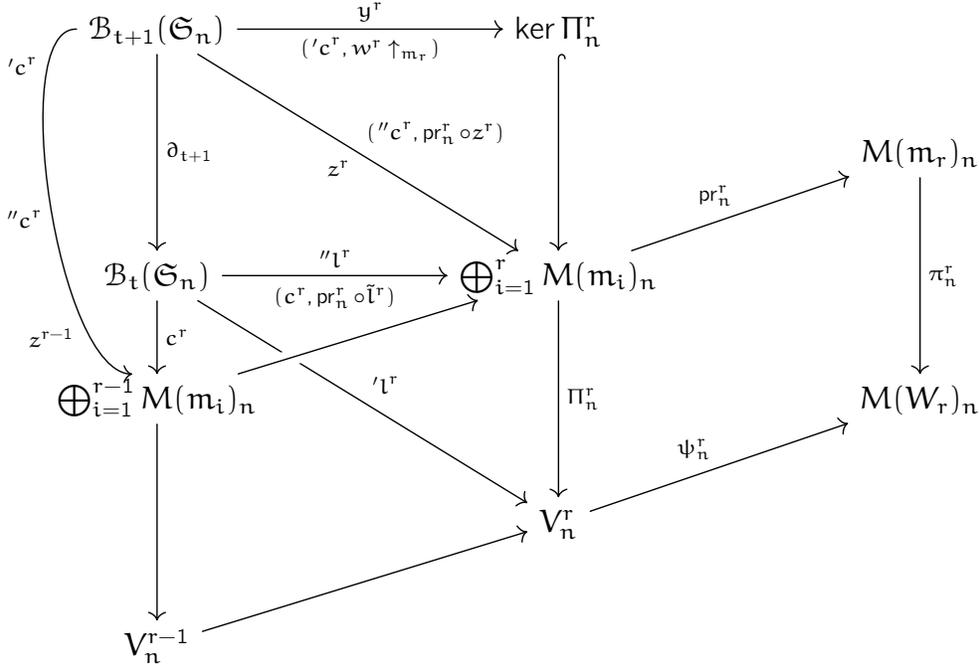
\end{enumerate}

\subsection{Combinatorial lemmas and the definition of periodicity}
\label{subsection:combinatorics}

The aim of \S\ref{subsection:combinatorics} is to provide the definition of  "periodicity" and the combinatorial motivation behind it. We start with some elementary definitions and lemmas.

\begin{definition}[{\bf The map $\beta$}]
We define the map $\beta^{m}: \fS_n \ra D_{m,n}$ by $\beta^{m}(\sigma)=g$ if $\sigma$ can be written as $ \sigma = h \gamma_g^{-1}$, $h \in \fS_{m} \times \fS_{n-m}$ and $g \in D_{m,n}$. In other words, \[\beta^{m} (\sigma) = g \iff \sigma = h \gamma_g^{-1} \iff g = \sigma^{-1} [m] \iff \tr^m(\sigma)^{-1} \sigma =\gamma^{-1}_g.\qedhere \]
\end{definition}

Note that $\beta^m$ satisfies compatibility relation analogous to the one mentioned in Remark~\ref{remark:compatibility} and has the following properties:
\begin{eqnarray}
\tr^m(ab)&=& \tr^m(a) \tr^m(\gamma_{\beta^m(a)}^{-1} b) \label{eqn:traceidentity}\\
\beta^m(ab) = & b^{-1} (\beta^m(a)) & = \beta^m(\gamma_{\beta^m(a)}^{-1} b) \label{eqn:betaproductidentity}\\
\tr^m(\sigma_1) = \tr^m(\sigma_2) &\Longrightarrow& \sigma_2^{-1} \sigma_1: \beta^m(\sigma_1)\ra \beta^m(\sigma_2) \text{ is order preserving.} \label{eqn:orderpreserving}
\end{eqnarray}

\begin{lemma}
\label{lemma:trace0}
Let $0 \leq n_1 \leq n_2 \leq \ldots \leq n_L$ be a finite chain. Assume that $n \geq n_L$ and consider subsets $f_{q,j} \in D_{n_q,n}$, $1 \leq q \leq L$, $j \in \{1,2\}$. Let $g_{q,j} = \beta^{n_q}(\eta_{q,j})$ where, \[\eta_{q,j} = \prod_{\xi = 1}^{q}\gamma_{f_{q-\xi+1,j}}^{-1} = \gamma_{f_{q,j}}^{-1} \ldots \gamma_{f_{1,j}}^{-1} \in \fS_n.\] Then,

\begin{enumerate}
\item For each $2 \leq q \leq L$, \[g_{q-1,j} \subseteq g_{q,j} \iff [n_{q-1}] \subseteq f_{q,j}\] additionally, we have \[\eta_{q-1, j} (g_{q,j} \setminus g_{q-1,j}) = f_{q,j} \setminus [n_{q-1}].\] 

\item Assume that we have filtrations $g_{1,j} \subseteq g_{2,j} \subseteq \ldots \subseteq  g_{L,j}$ for $j \in \{1,2\}$. Augment these by defining $g_{0,j} \coloneq \emptyset$ and $g_{L+1, j} \coloneq [n]$. Then, \[\tau\coloneq \eta_{L,2}^{-1}\eta_{L,1}: g_{L+1,1} \ra g_{L+1,2}\] is filtration preserving, that is, $\tau(g_{q,1}) = \tau(g_{q,2})$ for each $ 0 \leq q \leq L+1$. The restrictions to the successive differences, $\tau: g_{q,1} \setminus g_{q-1,1} \ra g_{q,2} \setminus g_{q-1,2}$ are order preserving for each $1 \leq q \leq L+1$. And $\tr^{n_L}(\eta_{L,j}) = (\alpha_j, \delta_j) \in \fS_{n_L} \times \fS_{n-n_L}$ satisfy \[\delta_j = \id_{\fS_{n-n_L}}, \qquad \text{for } j \in \{1,2\}.\]  Moreover, the following are equivalent.
\begin{enumerate}
\item[A.] $\tr^{n_L} (\eta_{L,1})= \tr^{n_L}( \eta_{L,2})$.
\item[B.] The unique order preserving bijection $\tau': g_{L,1} \ra g_{L,2}$ preserves the filtration, that is, $\tau'(g_{q,1}) = g_{q,2}$ for each $1 \leq q \leq L$.
\item[C.] $\tr^{n_q} (\eta_{q,1})= \tr^{n_q}( \eta_{q,2})$ for each $1 \leq q \leq L$.
\end{enumerate}
\end{enumerate}
\end{lemma}

\begin{remark}
We can imagine $\eta_{L,j}$ in the following way: assume that balls labeled $1$ to $n$ are arranged in the increasing order in a row (which we think of as the identity permutation) and that the balls with labels in $g_{q,j} \setminus g_{q-1, j}$ are of color $C_q$ where $1 \leq q \leq (L+1)$. This way each ball gets a unique color. Lift all the balls with color $C_1$ and put them in the front (the leftmost position in the row) without changing the relative order of balls of color $C_1$. Now pick all the balls with color $C_2$ and put them just after all the balls of color $C_1$ keeping the relative order of balls of color $C_2$. Similarly repeat the steps for $q = 3,\ldots, L$. Now this defines a permutation which is $\eta_{L,j}$.

If we repeat the steps of lifting and placing just for the colors $C_1$ to $C_q$, we get the permutation $\eta_{q,j}$. Also note that $f_{q+1,j} \setminus [n_{q}]$ is the set of positions of balls of color $C_{q+1}$ in the row we get after following the steps of lifting and placing for colors $C_1$ to $C_q$.
\end{remark}

\begin{proof}[Proof of Lemma~\ref{lemma:trace0}]
{\bf Part 1.} By definition of $\beta$, we have $g_{q,j} = \eta_{q,j}^{-1} [n_q]$. This implies,
\begin{eqnarray}
g_{q-1,j} \subseteq g_{q,j} &\iff& \eta_{q-1,j}^{-1} [n_{q-1}] \subseteq \eta_{q,j}^{-1} [n_q] \nonumber\\
&\iff& [n_{q-1}] \subseteq \eta_{q-1,j}\ \eta_{q,j}^{-1} [n_q] \nonumber\\
&\iff& [n_{q-q}] \subseteq \gamma_{f_{q,j}} [n_q] \nonumber \\
&\iff& [n_{q-1}] \subseteq f_{q,j}  \nonumber
\end{eqnarray} The second assertion is then immediate.

{\bf Part 2.}  Note that $\eta_{q,j}^{-1}([n_q]) = g_{q,j}$. And by part 1, for each $k>0$, we have $[n_q] \subseteq f_{q+k,j}$ which implies $\gamma_{f_{q+k,j}}^{\pm 1}: [n_q] \ra [n_q]$ is the identity map. Thus,
\begin{eqnarray}
\tau(g_{q,1}) &=& \big(\eta_{q,2}^{-1} \gamma_{f_{q+1,2}}\ldots \gamma_{f_{L,2}}\big)\big(\gamma_{f_{L,1}}^{-1} \ldots \gamma_{f_{q+1,1}}^{-1} \eta_{q,1}\big)(g_{q,1}) \nonumber \\
&=& \big(\eta_{q,2}^{-1} \gamma_{f_{q+1,2}}\ldots \gamma_{f_{L,2}}\big)\big(\gamma_{f_{L,1}}^{-1} \ldots \gamma_{f_{q+1,1}}^{-1}\big)([n_q]) \nonumber \\
&=& \eta_{q,2}^{-1}([n_q]) \nonumber \\
&=& g_{q,2}. \nonumber
\end{eqnarray}
Hence $\tau$ preserves filtration. For the second assertion we first show that $\delta_j = id$. Note that it is equivalent to showing that $\eta_{L,j} : [n] \setminus g_{L,j} \ra [n] \setminus [n_L]$ is order preserving. By equation~\ref{eqn:betaproductidentity}, $\beta^{n_L} (\gamma_{f_{L,j}}^{-1} \ldots \gamma_{f_{2,j}}^{-1}) = \gamma_{f_{1,j}}^{-1} g_{L,j}$. Hence by induction on the length of the product \[\eta_{L,j} \gamma_{f_{1,j}}= \gamma_{f_{L,j}}^{-1} \ldots \gamma_{f_{2,j}}^{-1} : [n] \setminus \gamma_{f_{1,j}}^{-1} g_{L,j} \ra [n] \setminus [n_L]\] is order preserving. The assertion now follows by noting that $\gamma_{f_{1,j}}^{-1}: [n] \setminus g_{L,j} \ra [n] \setminus \gamma_{f_{1,j}}^{-1} g_{L,j}$ is order preserving (which is true because $f_{1,j} = g_{1,j} \subseteq g_{L,j}$).

Replacing $L$ by $q-1$, we see that $\eta_{q-1, j} : [n] \setminus g_{q-1,j} \ra [n] \setminus [n_{q-1}]$ is order preserving. Using part 1, we deduce that $\eta_{q-1, j} : g_{q,j} \setminus g_{q-1,j} \ra f_{q,j} \setminus [n_{q-1}]$ is order preserving. Also, $\gamma_{f_{q,j}} : f_{q,j} \setminus [n_{q-1}] \ra [n_q] \setminus [n_{q-1}]$ is order preserving and $\gamma_{f_{q+k,j}}^{\pm 1}$ is identity on $[n_q] \setminus [n_{q-1}]$ for $k>0$. Now  note that $\tau = \big(\eta_{q-1,2}^{-1} \gamma_{f_{q,2}}\ldots \gamma_{f_{L,2}}\big)\big(\gamma_{f_{L,1}}^{-1} \ldots \gamma_{f_{q,1}}^{-1} \eta_{q-1,1}\big)$ which, when restricted to $g_{q,j} \setminus g_{q-1,j}$, is a composition of order preserving functions. Thus $\tau$ is order preserving on $g_{q,j} \setminus g_{q-1,j}$.

$A \Longrightarrow B$: By equation~\ref{eqn:orderpreserving}, $\tau\coloneq \eta_{L,2}^{-1}\eta_{L,1}: g_{L,1} \ra g_{L,2}$ is order preserving and we have already shown that $\tau$ preserves filtration.

$B \Longrightarrow A$: $\tau'$ must agree with $\tau$ because $\tau$ preserves filtration and is order preserving on each successive difference. This implies that $\tau$ is order preserving. Now note that \[\tau = \gamma_{g_{L,2}} \alpha_2^{-1} \alpha_1 \gamma_{g_{L,1}}^{-1} : g_{L,1} \ra g_{L,2}.\] It follows that $\alpha_2^{-1} \alpha_1: [n_L] \ra [n_L]$ is order preserving. Thus $\alpha_2^{-1} \alpha_1 =id$ which shows that $\tr^{n_L} (\eta_{L,1})= \tr^{n_L}( \eta_{L,2})$.

It is now immediate that $A \iff B \iff C$.
\end{proof}

\begin{remark}
\label{remark:twodecompositions}
Let $n_L \leq m \leq n$. Keeping the notations and hypothesis of Lemma~\ref{lemma:trace0} above, for each $j \in \{1,2\}$ we have two natural decompositions of the set \[D_{m,n}' \coloneq \{f \in D_{m,n} : f \supseteq [n_L] \}.\] The first decomposition is given by \[D_{m,n}' = \prod_{\kappa \in \fS_m \times \fS_{n-m}}\{f \in D_{m,n}' : \tr^m(\gamma_f^{-1} \eta_{L,j}) = \kappa \} .\] Let $\Gamma^j$ be the set of sequences $(m_{\alpha})_{1 \leq \alpha \leq n_L+1}$ satisfying \[0 \leq m_{\alpha} \leq g_{L,j}(\alpha) - g_{L,j}(\alpha -1) -1, \qquad \text{and} \qquad \sum m_{\alpha} = m- n_L.\] Then the second decomposition is given by 
\begin{eqnarray}
D_{m,n}' &=& \prod_{(m_{\alpha}) \in \Gamma^j } \{f \in D_{m,n}' : |\beta^m(\gamma_f^{-1} \eta_{L,j}) \cap (g_{L,j}(\alpha-1), g_{L,j}(\alpha))| = m_{\alpha},\ 1 \leq \alpha \leq n_L+1  \} \nonumber \\
&\cong& \prod_{(m_{\alpha}) \in \Gamma^j } \prod_{1 \leq \alpha \leq n_L+1} D_{m_{\alpha}, n_{\alpha, j}} \nonumber
\end{eqnarray} where $n_{\alpha,j} = |(g_{L,j}(\alpha-1), g_{L,j}(\alpha))|$. The equivalence $A \iff B$ in Lemma~\ref{lemma:trace0} implies that the two decompositions are essentially the same, that is, there is an injection $\kappa_j : \Gamma^j \ra \fS_{m} \times \fS_{n-m}$ such that the set \[\{f \in D_{m,n}' : |\beta^m(\gamma_f^{-1} \eta_{L,j}) \cap (g_{L,j}(\alpha-1), g_{L,j}(\alpha))|=m_{\alpha},\ 1 \leq \alpha \leq n_L+1  \} \] is same as the set \[  \{f \in D_{m,n}' : \tr^m(\gamma_f^{-1} \eta_{L,j}) = \kappa_j((m_{\alpha})) \}\] and $\{g \in D_{m,n}^j : \tr^m(\gamma_g^{-1} \eta_{L,j}) = \kappa \}$ is nonempty if and only if $\kappa$ is in the image of $\kappa_j$. Note that the trace condition $\tr^{n_L} (\eta_{L,1})= \tr^{n_L}( \eta_{L,2})$ implies that $\kappa_1$ agrees with $\kappa_2$ on $\Gamma^1 \cap \Gamma^2$ and induces an isomorphism \[\kappa_1 = \kappa_2 : \Gamma^1 \cap \Gamma^2 \ra \im{\Gamma^1} \cap \im{\Gamma^2}.\qedhere \]
\end{remark}

\begin{definition}[{\bf The equivalence relation $\equiv^u_a$ on $D_{m,n}$}]
Let $u \geq 0$ and let $f_1 , f_2 \in D_{m,n}$, we define the equivalence relation $\equiv^u_a$ by $f_1 \equiv^u_a f_2$  if and only if $f_1 \cap [n-a]= f_2 \cap [n-a]$ and $f_1 (j) \equiv f_2(j) \mod{u}$ for $j \in [m]$. Note that $\equiv^{u_1}_a$ is a refinement of $\equiv^{u_2}_a$ if $u_2 \mid u_1$.
\end{definition}

The following two lemmas are essential in proving the "periodicity" of the nice lift construction. See the proofs of Claim~\ref{claim:one} and Claim~\ref{claim:two} for a justification of the names assigned to them.

\begin{lemma}[{\bf First collision lemma}] \label{lemma:collisiontwo} Assume $n \geq a \geq 0$ and $u >0$.  Let $\delta \in D_{m,n}/\equiv^u_a$ be any equivalence class such that $\delta \nsubseteq D_{m,n-a}$. Then $|\delta| = \binom{\lfloor a/u \rfloor -s + b}{b}$ for some $0 < b \leq m$ and $0 \leq s \leq b$. Moreover, if $u p^{v_p(b!)+1} \mid a $ then $s>0$ and, \[|\delta| \equiv 0 \mod{p}.\] Here $v_p({-})$ is the $p$-adic valuation.
\end{lemma}

\begin{proof}
We have $a= \lfloor a/u \rfloor u + t$ for some $0 \leq t \leq u-1$. By hypothesis, there exists a  $b>0$ such that for each $f \in \delta$ we have $|f \cap ([n] \setminus [n-a])| =b$. Note that $\delta$ has a natural lexicographic order on it. Assume that $f$ is the least element under this order and let $o = \max{f}$, that is, $o= \min_{g \in \delta} \max{g}$. Note that $o > n-a$ and by the division algorithm \[(s-1)u +t < o -(n-a) \leq s u +t\] for some $s \geq 0$ (clearly we may assume $s>0$ when $t=0$). Minimality of $f$ implies that $s \leq b$. Now note that $|\delta|$ equals the number of nonnegative integer solutions $(k_1,k_2, \ldots, k_b)$ satisfying the inequality \[(o - (n-a))+ \sum_{q=1}^{b}k_q u \leq a.\] It follows that $|\delta| = \binom{\lfloor a/u \rfloor -s + b}{b}$ and the second assertion now follows by noting that $b > b-s \geq 0$ and applying the classical result that $\binom{x}{b} \mod{p}$ is periodic in $x$ with period $p^{v_p(b!)+1}$.
\end{proof}

\begin{lemma}[{\bf Second collision lemma}] \label{lemma:collision} 
Let $0 \leq n_1 \leq n_2 \leq \ldots \leq n_L$ be a finite chain and $n \geq m \geq n_L$. Let $f_{q,j}$, $g_{q,j}$, and $\eta_{q,j}$ for $1 \leq q \leq L$, $j \in \{1,2\}$ be as defined in Lemma~\ref{lemma:trace0}. Assume that we have filtrations $g_{1,j} \subseteq g_{2,j} \subseteq \ldots \subseteq  g_{L,j}$ for $j \in \{1,2\}$ satisfying $\tr^{n_L} (\eta_{L,1})= \tr^{n_L}( \eta_{L,2})$ and $g_{L,1} \equiv^w_a g_{L,2}$. Let $\delta \in D_{m,n}/ \equiv^u_a$ be any equivalence class and $\kappa$ be any element of $\fS_{m} \times \fS_{n-m}$. Define the sets \[\Gamma_j(\delta, \kappa) = \{ f \in D_{m,n}: \beta^{m}(\gamma_{f}^{-1} \eta_{L,j}) \in \delta, \tr^{m}(\gamma_{f}^{-1} \eta_{L,j})=\kappa,  [n_L] \subseteq f \}, \qquad j \in \{1,2\}.\] If $u p^{\ddH(m, n_L)} \mid w$ (see equation~\ref{def:deltaH}) then, \[|\Gamma_1(\delta, \kappa)| \equiv |\Gamma_2(\delta, \kappa)| \mod{p}.\]
\end{lemma}

\begin{proof}
Assume first that for each $j$, the set $\Gamma_j(\delta, \kappa)$ is nonempty. As in Remark~\ref{remark:twodecompositions}, the trace conditions $\tr^{m}(\gamma_{g}^{-1} \eta_{L,j})=\kappa$, $j \in \{1,2\}$ specify a sequence of numbers $\{m_{\alpha}\}_{1 \leq \alpha \leq n_L+1}$, independent of $j$, such that $g$ satisfies $\tr^{m}(\gamma_{g}^{-1} \eta_{L,j})=\kappa$ if and only if $|\beta^m(\gamma_f^{-1} \eta_{L,j}) \cap (g_{L,j}(\alpha-1), g_{L,j}(\alpha))|=m_{\alpha}$,  $1 \leq \alpha \leq n_L+1$. Let $n_{\alpha, j}$ be as in Remark~\ref{remark:twodecompositions} and  define $a_{\alpha, j} \coloneq \max \{0, \min \{a- (n - g_{L,j}(\alpha) +1), n_{\alpha,j}\}\}$. Then as in Remark~\ref{remark:twodecompositions}, we have the natural decomposition \[\Gamma_j(\delta, \kappa) = \prod_{1 \leq \alpha \leq n_L+1} \delta_{\alpha, j}, \qquad \text{where} \qquad \delta_{\alpha, j} \in D_{m_{\alpha}, n_{\alpha,j}} / \equiv^u_{a_{\alpha,j}}.\]
Since $g_{L,1} \equiv^w_a g_{L,2}$ there exists $1 \leq \alpha_0 \leq n_L+1$ such that $\delta_{\alpha, j} \nsubseteq D_{m_{\alpha}, n_{\alpha,j}- a_{\alpha,j}}$ or $n_{\alpha, j} = 0$ for  each $\alpha > \alpha_0$ and $j \in \{1,2\}$. If $n_{\alpha, j}=0$ or $m_{\alpha} =0$, we define $b_{\alpha, j} \coloneq 0$ and $s_{\alpha , j} \coloneq 0$. Hence by Lemma~\ref{lemma:collisiontwo}, we have 
\begin{equation}
|\Gamma_j(\delta, \kappa)| = \prod_{\alpha_0 < \alpha \leq n_L+1} \binom{\lfloor a_{\alpha,j}/ u \rfloor - s_{\alpha, j} + b_{\alpha, j}}{b_{\alpha, j}} \label{eqn:product}
\end{equation} for some $b_{\alpha, j} \leq m_{\alpha}$ and $s_{\alpha, j}$ as in the proof of Lemma~\ref{lemma:collisiontwo}. Again by $g_{L,1} \equiv^w_a g_{L,2}$, it follows that $b_{\alpha, j}$ and $s_{\alpha, j}$ are independent of $j$. Also $a_{\alpha, 1} \equiv a_{\alpha,2} \mod{w}$ and $s_{\alpha, 1} \equiv s_{\alpha,2} \mod{w/u}$. This implies \[\lfloor a_{\alpha,j}/ u \rfloor - s_{\alpha, j} + b_{\alpha, j} \equiv \lfloor a_{\alpha,j}/ u \rfloor - s_{\alpha, j} + b_{\alpha, j} \mod{p^{\ddH (m,n_L)}}. \] To avoid dealing with the case when exactly one of $\Gamma_j(\delta, \kappa)$ (say $\Gamma_2(\delta, \kappa)$) is empty separately, we extend our observation that $b_{\alpha, j}$ and $s_{\alpha, j}$, if exist, are independent of $j$ and define $b_{\alpha, 2}\coloneq b_{\alpha, 1}$ and $s_{\alpha, 2}\coloneq s_{\alpha, 1}$ for each $\alpha$. It is easy to see that with these conventions the equation~\ref{eqn:product} is always valid when at least one of $\Gamma_j(\delta, \kappa)$ is nonempty.

The result now follows from the classical fact that $\binom{x}{b_{\alpha,j}} \mod{p}$ is periodic in $x$ with period $p^{v_p(b_{\alpha, j}!)+1}$ and noting that $b_{\alpha, j} \leq m_{\alpha} \leq m-n_L$.
\end{proof}

The following lemma together with the collision lemmas provide the motive behind the forthcoming definition of periodicity (Definition~\ref{def:periodicity}).

\begin{lemma} \label{lemma:trace}
Let $\sigma\in \fS_{n-a} \subseteq \fS_n$ and $n_1 \leq n_2 \leq \ldots \leq n_L$ be a finite chain. Let $f_{q,j}$, $g_{q,j}$, and $\eta_{q,j}$ for $1 \leq q \leq L$, $j \in \{1,2\}$ be as defined in Lemma~\ref{lemma:trace0}. Assume that we have filtrations $g_{1,j} \subseteq g_{2,j} \subseteq \ldots \subseteq  g_{L,j}$ for $j \in \{1,2\}$ satisfying $\tr^{n_L} (\eta_{L,1})= \tr^{n_L}( \eta_{L,2})$ and $g_{L,1} \equiv^1_a g_{L,2}$. Consider arbitrary $\epsilon_q \in \{0,1\}$ for $1 \leq q \leq L$. Define $\sigma_{q,j}$ recursively by \[\sigma_{q,j}=(\tr^{n_q})^{\epsilon_q}(\gamma_{f_{q,j}}^{-1}\sigma_{q-1,j})\] where $\sigma_{0,j} = \sigma$ and $(\tr^{n_q})^0$ denotes the identity map. Then we have, \[\tr^{n_L}(\sigma_{L,1}) = \tr^{n_L}(\sigma_{L,2}) \in \fS_{n_L} \times \fS_{n-n_L- a + o_L} \subseteq \fS_{n-a+o_L} \] where $o_q\coloneq |g_{q,j} \cap ([n] \setminus [n-a])|$.
\end{lemma}

\begin{proof}
By Lemma~\ref{lemma:trace0}, the trace condition $\tr^{n_L} (\eta_{L,1})= \tr^{n_L}( \eta_{L,2})$ says precisely that there is an order preserving bijection $\tau: g_{L,1} \ra g_{L,2}$ that preserves the filtration. Hence the condition $g_{L,1} \equiv^1_a g_{L,2}$ implies that $g_{q,1} \equiv^1_a g_{q,2}$ for each $1 \leq q \leq L$. By Lemma~\ref{lemma:trace0} again, we have \[\tr^{n_L} (\eta_{L,1})= \tr^{n_L}( \eta_{L,2}) \in \fS_{n_L} \subseteq \fS_{n_L} \times \fS_{n-n_L- a + o_L} \subseteq \fS_{n-a+o_L}.\] The assertion of the lemma is easily verified in the base case when $L=1$, because we have $f_{1,1}=g_{1,1} \equiv^1_a g_{1,2} =f_{1,2}$ and $\sigma_{0,1} =\sigma_{0,1} = \sigma \in \fS_{n-a}$.

For the general case $L>1$ assume first that $\epsilon_{L-1} = 0$. Then we have 
\begin{eqnarray}
\tr^{n_L}(\sigma_{L,j}) &=& \tr^{n_L}(\gamma_{f_{L,j}}^{-1} \gamma_{f_{L-1,j}}^{-1} \sigma_{L-2,j}) \nonumber \\
&=& \tr^{n_L}(\gamma_{f_{L,j}}^{-1} \gamma_{f_{L-1,j}}^{-1}) \tr^{n_L} (\gamma_{\beta^{n_L}(\gamma_{f_{L,j}}^{-1} \gamma_{f_{L-1,j}}^{-1})}^{-1} \sigma_{L-2,j}) \label{eqn:trace}
\end{eqnarray} where equation~\ref{eqn:trace} follows from the identity in equation~\ref{eqn:traceidentity}. Now note that we have $f_{L-1,j} \subseteq \beta^{n_L}(\gamma_{f_{L,j}}^{-1} \gamma_{f_{L-1,j}}^{-1})$ and there is a filtration and order preserving bijection $\tau' : \beta^{n_L}(\gamma_{f_{L,1}}^{-1} \gamma_{f_{L-1,1}}^{-1}) \ra \beta^{n_L}(\gamma_{f_{L,2}}^{-1} \gamma_{f_{L-1,2}}^{-1})$ hence \[\tr^{n_L}(\gamma_{f_{L,1}}^{-1} \gamma_{f_{L-1,1}}^{-1}) =\tr^{n_L}(\gamma_{f_{L,2}}^{-1} \gamma_{f_{L-1,2}}^{-1}) \in \fS_{n_L} \subseteq \fS_{n_L} \times \fS_{n-n_L- a + o_L} \subseteq \fS_{n-a+o_L}.\] By induction on $L$ and noting that  \[f'_{q,j} = \begin{cases}
f_{q,j} \in D_{n_q, n} & \text{ if } 1 \leq q \leq L-2 \\
\beta^{n_L}(\gamma_{f_{L,j}}^{-1} \gamma_{f_{L-1,j}}^{-1}) & \text{ if } q= L-1
\end{cases}\] or equivalently (by equation~\ref{eqn:betaproductidentity}), \[g'_{q,j} = \begin{cases}
g_{q,j} \in D_{n_q, n} & \text{ if } 1 \leq q \leq L-2 \\
g_{L,j} & \text{ if } q= L-1
\end{cases}\] satisfy the hypothesis of the lemma with $o'_{L-1} = o_L$ we conclude that 
\[\tr^{n_L} (\gamma_{\beta^{n_L}(\gamma_{f_{L,1}}^{-1} \gamma_{f_{L-1,1}}^{-1})}^{-1} \sigma_{L-2,1})= \tr^{n_L} (\gamma_{\beta^{n_L}(\gamma_{f_{L,2}}^{-1} \gamma_{f_{L-1,2}}^{-1})}^{-1} \sigma_{L-2,2})\in \fS_{n_L} \times \fS_{n-n_L- a + o_L} \subseteq \fS_{n-a+o_L}\] completing the proof of the lemma in the case when $\epsilon_{L-1}=0$.

Now if $\epsilon_{L-1}=1$, then by induction we have $$\sigma_{L-1,1} = \sigma_{L-1,2} \in \fS_{n-a+o_{L-1}}.$$
Note that by Lemma~\ref{lemma:trace0} part 1, it follows that $f_{L,1} = \gamma^{-1}_{g_{L-1, 1}} g_{L,1} \equiv_{a}^1 \gamma^{-1}_{g_{L-1, 2}} g_{L,2} = f_{L,2}$ and in particular, $f_{L,1} \equiv_{a-o_{L-1}}^1 f_{L,2}$. We also have
$ |f_{L,j} \cap ([n] \setminus [n-a+o_{L-1}])| = o_L -o_{L-1}$. Hence from the base case, we have 
$$\tr^{n_L}(\gamma_{f_{L,1}}^{-1} \sigma_{L-1,1}) =\tr^{n_L}(\gamma_{f_{L,2}}^{-1} \sigma_{L-1,2}) \in \fS_{n_L} \times \fS_{n-n_L- a + o_L} \subseteq \fS_{n-a+o_L} $$
completing the proof of the lemma.
\end{proof}

\begin{definition}[{\bf $(n_1,\ldots, n_L)$-structure}] Let $n_1 \leq n_2 \leq \ldots \leq n_L$ be a finite chain of nonnegative integers. An {\bf $(n_1,\ldots, n_L)$-structure}  $\mathfrak{s}$ on $n$ is a collection of  subsets $f_{q} \in D_{n_q,n}$, $1 \leq q \leq L$. 

Let $H$ be a nonnegative integer. Let $\mathfrak{s}_j = (f_{q,j})_{1 \leq q \leq L}$, $j \in \{1,2\}$ be two $(n_1,\ldots, n_L)$-structures on $n$. We say that $\mathfrak{s}_1 \equiv^H_a \mathfrak{s}_2$ if all of the following three conditions hold: \begin{itemize}
   \item  $\tr^{n_L} (\eta_{L,1})= \tr^{n_L}( \eta_{L,2})$ where $\eta_{q,j} = \prod_{\xi = 1}^{q}\gamma_{f_{q-\xi+1,j}}^{-1}$,
  
   \item  $g_{q,1}  \equiv^{u_q}_a g_{q,2}$ with $u_q = p^{H + \sum_{\xi=1}^{L-q} \ddH (n_{L-\xi+1}, n_{L- \xi})} $ for $1 \leq q \leq L$ where $g_{q,j} = \beta^{n_q}(\eta_{q,j})$, and
   
   \item $[n_{q-1}] \subseteq f_{q,j}$  (or equivalently $g_{q-1,j} \subseteq g_{q,j}$) for $2 \leq q \leq L$.
   \end{itemize}
   
Let $\mathfrak{s} = (f_q)_{1 \leq q \leq L}$ be a $(n_1,\ldots, n_L)$-structure. For an element $e =(\sigma_0, \sigma_1, \ldots, \sigma_{t+1}) \in \mathcal{B}_{t+1}(\fS_{n})$ and an arbitrary function $\epsilon : [L] \times \{0, 1, \ldots, t+1\} \ra \{0,1\}$ we define $e_{q, \epsilon, \mathfrak{s}} \coloneq  (\sigma_{0,q}, \sigma_{1,q}, \ldots, \sigma_{t+1,q})$ where $\sigma_{k,q}$ is given recursively by \[\sigma_{k,q}=(\tr^{n_q})^{\epsilon(q, k)}(\gamma_{f_{q}}^{-1}\sigma_{k,q-1})\] and $\sigma_{k,0} = \sigma_{k}$. 
\end{definition}

\begin{definition}[{\bf Periodicity}] \label{def:periodicity} Let $z \in \Hom_{\fS_n}(\mathcal{B}_{t+1}(\fS_n), M(m)_n)$ and $H$ be a nonnegative integer. Then $z$ is $H$-{\bf periodic} if for every chain $n_1 \leq n_2 \leq \ldots \leq n_L = m$ of nonnegative integers and every function $\epsilon \colon [L] \times \{0, 1, \ldots, t+1\} \ra \{0,1\}$, we have \[( z \downarrow_{m})(e_{L,\epsilon, \mathfrak{s}_1}) = ( z \downarrow_{m})(e_{L,\epsilon, \mathfrak{s}_2})\]  whenever $e =(\sigma_0, \sigma_1, \ldots, \sigma_{t+1}) \in \mathcal{B}_{t+1}(\fS_{n-a}) \subseteq \mathcal{B}_{t+1}(\fS_{n})$ and $\mathfrak{s}_1 \equiv^H_a \mathfrak{s}_2$ are two equivalent $(n_1,\ldots, n_L)$-structures.
   
Let  $z \in \Hom_{\fS_n}(\mathcal{B}_{t+1}(\fS_n), \bigoplus_{i=1}^{d} M(m_i)_n)$ and $\pr^i : \bigoplus_{i=1}^{d} M(m_i) \ra M(m_i)$ be the natural projection. Let $\SQ=(H^{i,d})_{1 \leq i \leq d}$ be a sequence of nonnegative integers. Then $z$ is $\SQ$-{\bf periodic} if $\pr^i_n \circ z$ is $H^{i,d}$-periodic for each $1 \leq i \leq d$.
  \end{definition}

\begin{remark} 
\label{remark:periodgcd}
Note that if $z_j \in \Hom_{\fS_n}(\mathcal{B}_{t+1}(\fS_n), \bigoplus_{i=1}^{d} M(m_i)_n)$, $j \in \{0,1\}$ is $(H^{i}_j)_{1 \leq i \leq d}$-periodic then any $\bk$ linear combination $x z_0 +y z_1$ is also periodic with period \[\gcd((H^{i}_0)_{1 \leq i \leq d}, (H^{i}_1)_{1 \leq i \leq d}) \coloneq (\max(H^{i}_0, H^{i}_1))_{1 \leq i \leq d}. \qedhere \]
\end{remark}

\subsection{Periodicity of the nice lift construction}
\label{subsection:periodicityofnicelift}

This section contains all the technical results we need. We keep the notations from \S \ref{subsection:nicelift}.

\begin{claim}[{\bf Inheritability of periodicity}] \label{claim:one}
With the notations of the nice lift construction in \S\ref{subsection:nicelift}, let $z^d$ be $\SQ$-periodic where $\SQ = (H^{i,d})_{1 \leq i \leq d}$ is a sequence of nonnegative integers. Then for any $r \leq d$, $z^r$ is $\mathfrak{D}^r_{V^d}(\SQ)$-periodic and $\tilde{l}^r$ is $\mathfrak{D}_{V^r} \circ \mathfrak{D}^r_{V^d} (\SQ)$-periodic. (See Definitions~\ref{def:mathfrakD} and \ref{def:mathfrakDr}.)
\end{claim}

\begin{proof}
We first prove the periodicity of $z^r$ by downward induction on $r$. The base case $r=d$ is true by hypothesis. We now prove the assertion for $r-1$: for $r \leq d$, $i \leq r-1$, we show that $\pr_n^i \circ z^{r-1}$ is $H^{i,r-1}$-periodic. For that let $\mathfrak{s}_j = (f_{q,j})_{1 \leq q \leq L}$, $j \in \{1,2\}$ be two $(n_1, \ldots, n_L)$-structures with $n_L = m_i$ and let $e =(\sigma_0, \sigma_1, \ldots, \sigma_{t+1}) \in \mathcal{B}_{t+1}(\fS_{n-a}) \subseteq \mathcal{B}_{t+1}(\fS_{n})$. We let $e_{L,j} \coloneq e_{L, \epsilon, \mathfrak{s}_j}$, $j \in \{1,2\}$ be as defined in Definition~\ref{def:periodicity} with $H \coloneq H^{i, r-1}$ and $\epsilon$ arbitrary. Assume that $\mathfrak{s}_1 \equiv^H_a \mathfrak{s}_2$. By the definition of a nice lift (Definition~\ref{def:nicelift}), we have
\begin{eqnarray}
 (\pr_n^i \circ z^{r-1} \downarrow_{m_i})(e_{L,j}) & = & (\pr_n^i \circ ({'c}^r) \downarrow_{m_i} ) (e_{L,j}) + (\pr_n^i \circ ({''c}^r) \downarrow_{m_i})(e_{L,j}) \nonumber \\
&=& (\pr_n^i \circ z^r \downarrow_{m_i} ) (e_{L,j}) + (\pr_n^i \circ \ y^r \downarrow_{m_i})(e_{L,j}) \nonumber \\
&=& (\pr_n^i \circ z^r \downarrow_{m_i} ) (e_{L,j}) + (T^{i,r}_n \circ (\pr_{m_r}^i \circ x^r \uparrow_{m_r}) \downarrow_{m_i})(e_{L,j}) \nonumber \\
&=& (\pr_n^i \circ z^r \downarrow_{m_i} ) (e_{L,j}) + \sum_{\substack {g \in D_{m_r,n} \\ [m_i]\subseteq g }} (\pr_{m_r}^i \circ x^r \downarrow_{m_i}^{m_r})(\gamma_{g}^{-1} e_{L,j}) \label{eqn:tobesplit}
\end{eqnarray}
where equation~\ref{eqn:tobesplit} follows from Lemma~\ref{lemma:xryr}.
We use a shorthand $\equiv^{i,r}$ for $\equiv^{p^{H^{i,r}}}_a$ and split the summation in (\ref{eqn:tobesplit}) as a double sum to obtain
\begin{align}
 (\pr_n^i \circ z^{r-1} \downarrow_{m_i})&(e_{L,j})  = \nonumber \\& (\pr_n^i \circ z^r \downarrow_{m_i} ) (e_{L,j}) + \! \! \! \! \! \sum_{\substack{\delta \in D_{m_r,n}/ \equiv^{r,r} \\
 \kappa \in  \fS_{m_r} \times \fS_{n-m_r}}} \sum_{\substack {\beta^{m_r}(\gamma_{g}^{-1} \eta_{L,j}) \in \delta  \\ \tr^{m_r}(\gamma_{g}^{-1} \eta_{L,j})=\kappa \\ [m_i]\subseteq g}} \! \! \! \! \! \! (\pr_{m_r}^i \circ x^r \downarrow_{m_i}^{m_r})(\gamma_{g}^{-1} e_{L,j})  \label{eqn:splitted}
\end{align}

Note that $H^{i,r} \leq H^{i,r-1}$ by equation~\ref{def:H}. Hence by downward induction on $r$ we have, $(\pr_n^i \circ z^r \downarrow_{m_i} ) (e_{L,1}) =(\pr_n^i \circ z^r \downarrow_{m_i} ) (e_{L,2})$. This completes the inductive step in the case when $m_i > m_r$ because then $\pr_{m_r}^i \circ x^r \downarrow_{m_i}^{m_r}=0$. So we assume that $m_i \leq m_r$ and our remaining task is to show that the second term (the double sum) in equation~\ref{eqn:splitted} is independent of $j$. We accomplish this task by showing that every term in the second summation of the double summation in equation~\ref{eqn:splitted} is the same irrespective of $j$ and the number of terms for $j=1$ and $j=2$ differ by a multiple of $p$. In a sense, a collision is happening here. We need the following sub-lemmas.

\begin{sublemma}
\label{sublemma:collision}
Let $\Gamma_j(\delta, \kappa) = \{ g \in D_{m_r,n}: \beta^{m_r}(\gamma_{g}^{-1} \eta_{L,j}) \in \delta, \tr^{m_r}(\gamma_{g}^{-1} \eta_{L,j})=\kappa,  [m_i] \subseteq g \}$. Then, \[|\Gamma_1(\delta, \kappa)| \equiv |\Gamma_2(\delta, \kappa)| \mod{p}.\]
\end{sublemma}

\begin{proof}
Note that $m_r \geq n_L$ and $H^{r,r} + \ddH(m_r, n_L) \leq H^{i,r-1}$. The proof now follows from Lemma~\ref{lemma:collision}. \renewcommand{\qedsymbol}{\Coffeecup}
\end{proof}

\begin{sublemma} \label{sublemma:tracetwo}
There are $\kappa'_k \in  \fS_{m_r} \times \fS_{n-m_r}$, $1 \leq k \leq t+1$ such that $\forall g\in \Gamma_j(\delta, \kappa), j \in \{1,2\}$ we have
$$\tr^{m_r}(\gamma_{g}^{-1} \sigma_{k,L,j})=\kappa'_k$$
\end{sublemma}

\begin{proof}
Proof immediately follows from Lemma~\ref{lemma:trace} noting that $\equiv^u_a$ is a refinement of $\equiv^1_a$ if $1 \mid u$. \renewcommand{\qedsymbol}{\Coffeecup}
\end{proof}
 
By Sub-lemma~\ref{sublemma:collision} and \ref{sublemma:tracetwo}, and the equation~\ref{eqn:xrwr} relating $x^r$ and $w^r$, it is enough to show that there is $C \in \bk[\fS_{m_r}]$ such that $w^r(\gamma_g^{-1}e_{L,j}) = C$ for all $g \in \Gamma_j(\delta, \kappa)$, $j \in \{1,2\}$. We also know from Definition~\ref{def:nicelift} (or Step 2 of the nice lift construction) that $w^r$ is a linear combination of the following three functions 
\[ w^r = a^r \circ \Tr_t^{m_r}\circ \ \partial_{t+1} + (\pr_n^r \circ z^r) \downarrow_{m_r} \circ \  h_t^{m_r} \circ \partial_{t+1} - (\pr_n^r \circ z^r) \downarrow_{m_r}.\] By sub-Lemma~\ref{sublemma:tracetwo} again and the commutativity of $\Tr$ and $\partial$, there is $C_1$ such that $a^r \circ \Tr_t^{m_r}\circ \ \partial_{t+1} (\gamma_g^{-1}e_{L,j}) = C_1$ for all $g \in \Gamma_j(\delta, \kappa)$, $j \in \{1,2\}$. Note that $H^{i,r-1} \geq H^{r,r} + \ddH (m_r , m_i)$. Hence by Lemma~\ref{lemma:homotopy} and downward induction on $r$, there are $C_2$, $C_3$ such that 
$(\pr_n^r \circ z^r) \downarrow_{m_r} \circ h_t^{m_r} \circ \partial_{t+1} (\gamma_g^{-1}e_{L,j}) = C_2$ and $(\pr_n^r \circ z^r) \downarrow_{m_r} (\gamma_g^{-1}e_{L,j}) = C_3$ for all $g \in \Gamma_j(\delta, \kappa)$, $j \in \{1,2\}$. Thus we may take $C = C_1 + C_2 - C_3$, completing the proof of the first assertion in claim.

For the second assertion, it is enough to prove the result for $r=d$ (because by the nice lift construction $\pr^i_n\circ \ \tilde{l}^r = \pr^i_n\circ \ \tilde{l}^d$ for any $i \leq r \leq d$). The proof is now immediate by the periodicity of $z^r$, Lemma~\ref{lemma:homotopy} and the equation~\ref{eqn:lrzr} relating $\tilde{l}^r$ to $z^r$.
\end{proof}

The following corollary follows from the proof of the claim above by taking $L=1$, $f_{1,j} =[m_i]$ and $\epsilon_1$ to be identically zero.

\begin{corollary} \label{cor:needed}
Let $e \in \mathcal{B}_{t+1}(\fS_{n-a})$ and $\delta \in D_{m_r,n}/\equiv^{r,r}$. Then for each $1 \leq i \leq r$ there is a constant $C$ such that $\pr_{m_r}^i \circ x^r \downarrow_{m_i}^{m_r} (\gamma_g^{-1}e) = C$ for all $[m_i] \subseteq g \in \delta$.
\end{corollary}

The following is the most crucial result. It provides another motive behind the definition of periodicity.

\begin{claim}[{\bf Periodicity of $\sharp$-filtered kernels}] \label{claim:kernel}
Let $z^d \in \Hom_{\fS_n}(\mathcal{B}_{t+1}(\fS_n), \bigoplus_{i=1}^{d} M(m_i)_n)$ be $\SQ$-periodic. If  $\Pi^d_n \circ z^d = 0$ and $p^{\mathfrak{I}_{V^d}(\SQ)} \mid a$, then $\Pi^d_{n-a} \circ R_{t+1}(z^d) = 0$ (see Definition~\ref{def:Rmap} for the definition of the $R$ map).
\end{claim}

\begin{proof}
First we deal with the case when $t=-1$. Our proof is by induction on $d$. The case $d=1$ is trivial because then $V^d$ is of the form $M(W)$ and $R_{t+1}$ is a well-defined map on $M(W)$ (see Definition~\ref{def:Rmap}). For $d>1$ note that $l^d = 0 \in \Hom_{\fS_n}(\mathcal{B}_{t}(\fS_n), V^d_n) = 0$ is a twisted $z^d$-cycle (because $\Pi_n^d \circ z^d = 0$). By the nice lift construction, a nice lift of $l^d$ is $\tilde{l}^d \coloneq 0$ (it has to be zero because for $t=-1$, $\mathcal{B}_t(\fS_n)$ is defined to be the trivial space) and there are parameters $x^r$, $y^r$ and $z^r$ satisfying the conditions in Definition~\ref{def:nicelift} (definition of a nice lift). It follows from that definition that $\Pi^d_n \circ y^d =0$, $z^d = ({''c}^d, - \pr^d_n y^d)$ and $z^{d-1} = z^d + y^d$. Hence $\Pi^{d-1}_n \circ z^{d-1} =0$. By Claim~\ref{claim:one}, $z^{d-1}$ is $\mathfrak{D}^{d-1}_{V^d}(\SQ)$-periodic. Thus by induction on $d$, $\Pi^{d-1}_{n-a} \circ R_{t+1}(z^{d-1}) =0$ (note that $p^{\mathfrak{I}_{V^{d-1}}(\mathfrak{D}^{d-1}_{V^d}(\SQ))} \mid p^{\mathfrak{I}_{V^d}(\SQ)}$). Hence, to finish the proof, it is enough to show that $\Pi^{d}_{n-a} \circ R_{t+1} (y^d) = 0$ whenever $p^{\mathfrak{I}_{V^d}(\SQ)} \mid a$. For that let $e \in \mathcal{B}_{t+1}(\fS_{n-a})$. By Lemma~\ref{lemma:xryr}, for each $1 \leq i \leq d$, we have the following:
\begin{eqnarray}
\pr_n^i \circ y^d \downarrow_{m_i}^n (e) &=& \sum_{g \in D_{m_d, n}, \; [m_i] \subseteq g} \pr_{m_d}^i \circ x^d \downarrow_{m_i}^{m_d}(\gamma_g^{-1}e) \nonumber \\
&=& \sum_{\delta \in D_{m_d, n}/\equiv^{d,d}} \sum_{\substack{g \in \delta \\ [m_i] \subseteq g}} \pr_{m_d}^i \circ x^d \downarrow_{m_i}^{m_d}(\gamma_g^{-1}e) \nonumber \\
&=& \sum_{\delta \in D_{m_d, n-a}/\equiv^{d,d}} \sum_{\substack{g \in \delta \\ [m_i] \subseteq g}} \pr_{m_d}^i \circ x^d \downarrow_{m_i}^{m_d}(\gamma_g^{-1}e) \label{eqn:collision}\\
&=& \sum_{g \in D_{m_d, n-a}, \; [m_i] \subseteq g} \pr_{m_d}^i \circ x^d \downarrow_{m_i}^{m_d}(\gamma_g^{-1}e) \nonumber \\
&=& (T^{i,d}_{n-a} \circ R_{t+1} ( \pr_{m_d}^i \circ x^d \uparrow_{m_d}^n)) \downarrow_{m_i}^{n-a} (e) \nonumber
\end{eqnarray} Here equation~\ref{eqn:collision} is another instance of collision and follows from Corollary~\ref{cor:needed} to Claim~\ref{claim:one} and the first collision lemma  (this is Lemma~\ref{lemma:collisiontwo}; to be able to apply this lemma we need $p^{\SQ_d + \ddH(m_d, m_i)} \mid a$ which is true by hypothesis on $a$). This implies, \begin{align*}
R_{t+1}(y^d)(e) & =  R_{t+1}(T^{\tilde{V}^d, m_d}_n \circ (x^d \uparrow_{m_d}^n)) \nonumber \\
& = T^{\tilde{V}^d, m_d}_{n-a} \circ R_{t+1}(x^d \uparrow_{m_d}^n) \in \ker \Pi^d_{n-a}
\end{align*} completing the proof in the case when $t=-1$.

For the general case, note that the domain of $R_{t+1}(z^d)$ is $\mathcal{B}_{t+1}(\fS_{n-a})$. Let $e \in \mathcal{B}_{t+1}(\fS_{n-a})$ and define a function $\breve{z}^d \in \Hom_{\fS_n}(\mathcal{B}_{0}(\fS_n), \bigoplus_{i=1}^{d} M(m_i)_n)$ by $\breve{z}^d(\sigma) = z^d(\sigma e)$. Hypothesis on $z^d$ implies that $\Pi^d_n \circ \breve{z}^d =0$ and that $\breve{z}^d$ is $\{H^{i,d}\}$-periodic. By the $t=-1$ case above, we have $\Pi^d_{n-a} \circ R_0(\breve{z}^d) = 0$. Hence $\Pi^d_{n-a} \circ R_{t+1}(z^d)(e) = \Pi^d_{n-a} \circ R_0(\breve{z}^d)(1_{\fS_{n-a}}) =0$, completing the proof.
\end{proof}

\subsection{Well-definedness and commutativity results}
\label{subsection:well-definedness}

The aim of \S \ref{subsection:well-definedness} is to use the technical results in \S \ref{subsection:periodicityofnicelift} to prove some well-definedness and commutativity results.

\begin{claim} \label{claim:two}
 Let $\SQ$ be a $d$-length sequence of nonnegative integers. With the notation of the nice lift construction, let $z^d$ be $\SQ$-periodic. Let $\tilde{l}^d \in \Hom_{\fS_{n}}(\mathcal{B}_{t}(\fS_n), \bigoplus_{i=1}^{d} M(m_i)_n)$ be a nice lift of a $z^d$-cycle $l^d$. If $p^{\mathfrak{I}(\SQ)} \mid a$, then $\Pi^d_{n-a} \circ R_t (\tilde{l}^d)$ is a $R_{t+1}(z^d)$ cycle in $\Hom_{\fS_{n-a}}(\mathcal{B}_{t}(\fS_n), V^d_{n-a})$ and $R_t(\tilde{l}^d)$ is one of its nice lifts.
\end{claim}

\begin{proof} Proof is very similar to the one for Claim~\ref{claim:kernel}. We still provide details.

Note that during the construction of $\tilde{l}^d$ we define $\tilde{l}^r$ for $ 1 \leq r \leq d$. We prove by upward induction on $r$ that $\Pi^r_{n-a}  \circ R_t (\tilde{l}^r)$ is a $R_{t+1}  (z^r)$-cycle and $R_t (\tilde{l}^r)$ is one of its nice lifts. For that let $\Res^r:\Hom_{\fS_{m_r} \times \fS_{n-m_r}}(\mathcal{B}_{t}(\fS_n), M(m_r)_{m_r}) \ra \Hom_{\fS_{m_r} \times \fS_{n-a-m_r}}(\mathcal{B}_{t}(\fS_{n-a}), M(m_r)_{m_r})$ be the natural restriction map and $\varsigma_t: \mathcal{B}_{t}(\fS_{m_r} \times \fS_{n-a-m_r}) \ra \mathcal{B}_{t}(\fS_{m_r} \times \fS_{n-m_r})$ be the natural inclusion. Then we have
\begin{eqnarray} \pr_{n-a}^r  \circ R_t (\tilde{l}^r) &=& \Res^r(\pr_n^r \circ \  \tilde{l}^r \downarrow^n_{m_r}) \uparrow_{m_r}^{n-a} \nonumber \\
&=& \Res^r(a^r \circ \Tr_{t,n}^{m_r} + (\pr_n^r \circ z^r) \downarrow_{m_r}^n \circ h_{t,n}^{m_r} ) \uparrow_{m_r}^{n-a} \nonumber \\
&=& (a^r \circ \varsigma_t \circ \Tr_{t,n-a}^{m_r} + \Res^r(\pr_n^r \circ z^r \downarrow_{m_r}^n) \circ h_{t,n-a}^{m_r}) \uparrow_{m_r}^{n-a} \nonumber \\
&=& (a^r \circ \varsigma_t \circ \Tr_{t,n-a}^{m_r} + (\pr_{n-a}^r \circ R_t ( z^r)) \downarrow_{m_r}^{n-a} \circ h_{t,n-a}^{m_r}) \uparrow_{m_r}^{n-a} \label{eqn:restriction} 
\end{eqnarray} 

For the case $r=1$, $V^r$ is an $\FIsharp$-module and hence it is clear that $\Pi^1_{n-a}  \circ R_t (\tilde{l}^1)$ is a $R_{t+1} (z^1)$ cycle and equation \ref{eqn:restriction} shows that $R_t (\tilde{l}^1)$ is a nice lift.

For $r>1$, equation \ref{eqn:restriction} implies that a nice lift of $\Pi^r_{n-a}  \circ R_t (\tilde{l}^r)$ is given by $(\tilde{o}, \pr_{n-a}^r \circ R_t (\tilde{l}^r))$ where $\tilde{o}$ is a nice lift of the $'z^{r-1}$ cycle $\Pi^{r-1}_{n-a} \circ R_t (\tilde{l}^{r-1})$ where, from the nice lift construction, we have $\pr_{n-a}^i \circ ('z^{r-1}) =\pr_{n-a}^i \circ R_{t+1} (''c^r) + T^{i,r}_{n-a} \circ  R_{t+1} (\pr_{m_r}^i \circ x^r \uparrow_{m_r}^n)$. Hence to complete the proof of the claim it is enough to show that $'z^{r-1} = R_{t+1} ( z^{r-1})$ because then,  by induction on $r$, $\tilde{o}$ may be replaced by $R_t ( \tilde{l}^{r-1})$. And that is equivalent to showing the following commutativity relation for $1 \leq i \leq r$,\[T^{i,r}_{n-a} \circ R_{t+1} ( \pr_{m_r}^i \circ x^r \uparrow_{m_r}^n) =  R_{t+1} ( T_n^{i,r} \circ (\pr_{m_r}^i \circ x^r \uparrow_{m_r}^n)).\] For that let $e \in \mathcal{B}_{t+1}(\fS_{n-a})$. As in the proof of Claim~\ref{claim:kernel} we have
\begin{eqnarray}
\pr_n^i \circ y^r \downarrow_{m_i}^n (e) &=& \sum_{g \in D_{m_r, n}, \; [m_i] \subseteq g} \pr_{m_r}^i \circ x^r \downarrow_{m_i}^{m_r}(\gamma_g^{-1}e) \nonumber \\
&=& \sum_{\delta \in D_{m_r, n}/\equiv^{r,r}} \sum_{\substack{g \in \delta \\ [m_i] \subseteq g}} \pr_{m_r}^i \circ x^r \downarrow_{m_i}^{m_r}(\gamma_g^{-1}e) \nonumber \\
&=& \sum_{\delta \in D_{m_r, n-a}/\equiv^{r,r}} \sum_{\substack{g \in \delta \\ [m_i] \subseteq g}} \pr_{m_r}^i \circ x^r \downarrow_{m_i}^{m_r}(\gamma_g^{-1}e) \nonumber\\
&=& \sum_{g \in D_{m_r, n-a}, \; [m_i] \subseteq g} \pr_{m_r}^i \circ x^r \downarrow_{m_i}^{m_r}(\gamma_g^{-1}e) \nonumber \\
&=& (T^{i,r}_{n-a} \circ R_{t+1} ( \pr_{m_r}^i \circ x^r \uparrow_{m_r}^n)) \downarrow_{m_i}^{n-a} (e) \nonumber
\end{eqnarray} This implies,
\begin{eqnarray}
(T^{i,r}_{n-a} \circ R_{t+1} ( \pr_{m_r}^i \circ x^r \uparrow_{m_r}^n)) \downarrow_{m_i}^{n-a} &=& \Res^i(\pr_n^i \circ y^r \downarrow_{m_i}^n)\nonumber \\ &=& (R_{t+1} (\pr_n^i \circ y^r)) \downarrow_{m_i}^{n-a} \nonumber \\ &=& ( R_{t+1} ( T_n^{i,r} \circ (\pr_{m_r}^i \circ x^r \uparrow_{m_r}^n))) \downarrow_{m_i}^{n-a} \nonumber
\end{eqnarray} completing the proof of the claim.
\end{proof}

Note that $\mathfrak{I} \circ \mathfrak{D}(\SQ) \geq \mathfrak{I}(\SQ)$ for any $\SQ \in \mathbb{Z}^d_{\geq 0}$. We claim the following.

\begin{claim}[{\bf Well-definedness of $\mathfrak{R}$}] \label{claim:welldefinedness} Let $z^d$ be $\SQ$-periodic and assume that $p^{\mathfrak{I} \circ \mathfrak{D}(\SQ)} \mid a$. Then the map \[\mathfrak{R}_{t, V^d}^{n, a, z^d}: H^{t,z^d}(\fS_n,V_n^d) \ra H^{t,R_{t+1}(z^d)}(\fS_{n-a},V_{n-a}^d)\] defined by $[l^d] \mapsto [\Pi^d_{n-a} \circ R_t (\tilde{l}^d)]$ is well-defined.
\end{claim}

\begin{proof}
Let $l_j \in [l^d]$, $j \in \{1,2\}$. Let $\tilde{l}_j$ be any of their nice lifts. We have $l_1 -l_2 = b \circ \partial_{t}$ for some $b \in \Hom_{\fS_{n}}(\mathcal{B}_{t-1}(\fS_n), V_n^d)$. Also, $\Pi^d_{n} \circ \tilde{l}_j = l_j + b_j \circ \partial_t$ for some $b_j \in \Hom_{\fS_{n}}(\mathcal{B}_{t-1}(\fS_n), V_n^d)$. Combining these equations yields $\Pi^d_{n} \circ (\tilde{l}_1 - \tilde{l}_2) = (b+b_1-b_2) \circ \partial_t$. This shows that $b_3\coloneq b+b_1-b_2$ is a twisted $(\tilde{l}_1 - \tilde{l}_2)$-cycle. Let $\tilde{b}_3$ be one of its nice lifts. By Claim~\ref{claim:one}, $(\tilde{l}_1 - \tilde{l}_2)$ is $\mathfrak{D}(\SQ)$-periodic. Hence by Claim~\ref{claim:two}, $\Pi^d_{n-a} \circ R_{t-1} (\tilde{b}_3)$ is a $R_{t}(\tilde{l}_1 - \tilde{l}_2)$ cycle if $p^{\mathfrak{I} \circ \mathfrak{D}(\SQ)} \mid a$. This translates into $\Pi^d_{n-a} \circ R_{t}(\tilde{l}_1 - \tilde{l}_2)=\Pi^d_{n-a} \circ R_{t-1} (\tilde{b}_3) \circ \partial_t$, completing the proof because $\Pi^d_{n-a} \circ R_{t-1} (\tilde{b}_3) \circ \partial_t$ is a boundary. 
\end{proof}

\begin{definition}
\label{def:phi-star}
For a sequential map $\phi : V \ra W$ corresponding to $f_{\phi}: S \ra T$ as in Definition~\ref{def:sequential}, we define a map $\phi_{\star}: \mathbb{Z}^{d_1}_{\geq 0} \ra \mathbb{Z}^{d_2}_{\geq 0}$ by $\phi_{\star} ((H^i)_{1 \leq i \leq d_1}) = (G^k)_{1 \leq k \leq d_2}$ where $G^k$ is given by \[G^k =
\begin{cases}
H^{f_{\phi}^{-1}(k)}, & \text{if } k \in T  \\
0, & \text{otherwise } 
\end{cases}
 \qedhere \]
\end{definition}

\begin{remark} With the notations of Definition~\ref{def:sequential}, Definition~\ref{def:twistedcyles} and Definition~\ref{def:phi-star}, let $z^{d_1} \in \Hom_{\fS_{n}}(\mathcal{B}_{t+1}(\fS_n), \tilde{V}^{d_1}_n)$ be $\SQ$-periodic then, 
\begin{enumerate}
\item $z^{d_1} \circ \partial_{t+2}$ is $\SQ$-periodic,
\item  $\tilde{\phi}_n \circ z^{d_1}$ is $\phi_{\star}(\SQ)$-periodic,  and
\item $R_{t+1}(\tilde{\phi}_{n} \circ z^{d_1}) = \tilde{\phi}_{n-a} \circ R_{t+1}(z^{d_1})$. \qedhere
\end{enumerate} 
\end{remark}

Note that $H^{t,0}(\fS_n,V_n^d)$ is same as the usual cohomology group $H^{t}(\fS_n,V_n^d)$. We have the following lemma.

\begin{lemma} \label{lemma:sequentialcommutativity}

Let $\phi : V \ra  W$ be a sequential map as in Definition~\ref{def:sequential} and let $\SQ$ be the sequence of length $d_1$ consisting only of zeros. If $p^{\max\{\mathfrak{I}_W\circ \mathfrak{D}_W \circ \phi_{\star} \circ \mathfrak{D}_V(\SQ),\;  \mathfrak{I}_V\circ \mathfrak{D}_V(\SQ) \}} \mid a$, then the diagram in Figure~\ref{Fig:14} commutes.
\begin{figure}[h]
\centering
\begin{tikzcd} 
H^{t}(\fS_n,V^{d_1}_n) \ar{rr}{\bar{\phi}_{n,t}} \ar{d}{\mathfrak{R}^{n,a,0}_{t,V^{d_1}}} & & H^{t}(\fS_n,W^{d_2}_n) \ar{d}{\mathfrak{R}^{n,a,0}_{t,W^{d_2}}} \\ H^{t}(\fS_{n-a},V^{d_1}_{n-a}) \ar{rr}{\bar{\phi}_{n-a,t}}  & & H^{t}(\fS_{n-a},W^{d_2}_{n-a})
\end{tikzcd}
\caption{} \label{Fig:14}
\end{figure}
\end{lemma}
\begin{proof}
Let $l \in \Hom_{\fS_{n}}(\mathcal{B}_{t}(\fS_n), V^{d_1}_n)$ (we keep the superscript $d_1$, which just denote the length of the filtration of $V$, to keep our notations consistent with Definition~\ref{def:twistedcyles}) be a $0$-cycle and let $\tilde{l}$ be a nice lift of $l$. Then $\phi_n \circ l$ is a twisted $\tilde{\phi}_n \circ \tilde{l} \circ \partial_{t+1}$-cycle. Let $\tilde{l}_1$ be a nice lift of $\phi_n \circ l$. By Claim~\ref{claim:one}, $\tilde{l}$ is $\mathfrak{D}(\SQ)$-periodic. This implies that $\tilde{\phi}_n \circ \tilde{l} \circ \partial_{t+1}$ is $\phi_{\star} \circ \mathfrak{D}(\SQ)$-periodic. If $p^{\mathfrak{I} \circ \phi_{\star} \circ \mathfrak{D}(\SQ)} \mid a$ and $p^{\mathfrak{I} \circ \mathfrak{D}(\SQ)} \mid a$, then by Claim~\ref{claim:two} we have that $R_t(\tilde{l}_1)$ is a nice lift of a $\tilde{\phi}_{n-a} \circ R_{t}( \tilde{l}) \circ \partial_{t+1}$ cycle and also that $R_t(\tilde{l})$ is a nice lift of a $0$ cycle. This implies $\Pi^{2}_{n-a} \circ \tilde{\phi}_{n-a} \circ R_{t}( \tilde{l}) \circ \partial_{t+1} =0$. Hence $R_t(\tilde{l}_1)$ is a lift (not necessarily nice) of a zero cycle. Since $\Pi^{2}_n \circ \tilde{\phi}_{n} \circ \tilde{l} \circ \partial_{t+1} =0$, $\phi_n \circ l$ is also a $0$ cycle.

Let $\tilde{l}_2$ be a nice lift of $\phi_n \circ l$ as a $0$ cycle. This yields the equation, $\Pi^{2}_n \circ (\tilde{l}_1 - \tilde{l}_2) = b \circ \partial_{t}$ for some $b$. Here $b$ is a $\tilde{l}_1 - \tilde{l}_2$ cycle and $\tilde{l_1} - \tilde{l}_2$ is $\mathfrak{D} \circ \phi_{\star} \circ  \mathfrak{D}(\SQ)$-periodic. Thus if $p^{\mathfrak{I}\circ \mathfrak{D} \circ \phi_{\star} \circ \mathfrak{D}(\SQ)} \mid a$ then, as in Claim~\ref{claim:welldefinedness}, we have $\Pi^{2}_{n-a} \circ R_{t}(\tilde{l}_1 - \tilde{l}_2)=\Pi^{2}_{n-a} \circ R_{t-1} (\tilde{b}) \circ \partial_{t}$. This completes the proof because $\Pi^{2}_{n-a} \circ R_{t-1} (\tilde{b}) \circ \partial_{t}$ is a boundary, \begin{align}
\bar{\phi}_{ n-a,t} \circ \mathfrak{R}_{t, V^{d_1}}^{n,a,0} ([l]) &= [ \Pi^{2}_{n-a} \circ R_t(\tilde{\phi}_n \circ \tilde{l})] = [\Pi^{2}_{n-a} \circ R_t(\tilde{l}_1)], \qquad \text{and} \nonumber \\
\mathfrak{R}_{t, W^{d_2}}^{n,a,0} \circ \bar{\phi}_{ n,t} ([l])&= [\Pi^{2}_{n-a} \circ R_t(\tilde{l}_2)]. \nonumber \qedhere
\end{align}
\end{proof}

Let $l_0 \in H^{t,z^d}(\fS_n,V_n^d)$. Then it is clear that the map $\mathfrak{U}^{n, z^d,  l_0}_{t, V^d}:H^{t,z^d}(\fS_n,V_n^d) \ra H^{t,0}(\fS_n,V_n^d)$ defined by $l \mapsto l-l_0$ is bijective.

\begin{lemma}
\label{lemma:untwistingdiagram}
Let $z^d$ be $\SQ$-periodic and let $p^{\mathfrak{I} \circ \mathfrak{D}(\SQ)} \mid a$, then the diagram in Figure~\ref{Fig:15} commutes provided $H^{t,z^d}(\fS_n,V_n^d)$ is nonempty. (The obvious subscripts and superscripts are omitted here.)

\begin{figure}[h]
\centering
\begin{tikzcd}
H^{t,z^d}(\fS_n,V_n^d) \ar{rr}{\mathfrak{U}} \ar{d}{\mathfrak{R}} & & H^{t}(\fS_n,V_n^d) \ar{d}{\mathfrak{R}} \\ H^{t,R_{t+1}(z^d)}(\fS_{n-a},V_{n-a}^d) \ar{rr}{\mathfrak{U}} & & H^{t}(\fS_{n-a},V_{n-a}^d)
\end{tikzcd}
\caption{} \label{Fig:15}
\end{figure}
\end{lemma}
\begin{proof}
The proof is similar to the proof of well-definedness. We need to show that, under the hypothesis on $a$,  $\mathfrak{R}(l) - \mathfrak{R}(l_0)$ and $\mathfrak{R}(l-l_0)$ differ by a boundary. Let $\tilde{l}$, $\tilde{l}_0$ and $\tilde{l}_1$ be nice lifts of $l$, $l_0$ and $l-l_0$ respectively. Then we have  $\Pi^{d}_n \circ (\tilde{l} - \tilde{l}_0 -\tilde{l}_1) = b \circ \partial_{t}$ for some $b$. Then $b$ is a $(\tilde{l} - \tilde{l}_0 -\tilde{l}_1)$ cycle. Also $(\tilde{l} - \tilde{l}_0 -\tilde{l}_1)$ is $\mathfrak{D}(\SQ)$-periodic. Hence if $p^{\mathfrak{I} \circ \mathfrak{D}(\SQ)} \mid a$ then by Claim~\ref{claim:two}, we have $\Pi^{d}_{n-a} \circ R_t(\tilde{l} - \tilde{l}_0 -\tilde{l}_1) = \Pi^{d}_{n-a} \circ R_{t-1}(\tilde{b}) \circ \partial_{t}$, completing the proof.
\end{proof}

\subsection{Proof of the main theorem for \texorpdfstring{$\sharp$}{sharp}-filtered \texorpdfstring{$\FI$}{FI}-modules}
\label{subsection:mainfiltered}

The aim of \S\ref{subsection:mainfiltered} is to show that the map $\mathfrak{R}$ as in Claim~\ref{claim:welldefinedness} is an isomorphism whenever $n-a$ is large and $a$ is divisible by a sufficiently large power of $p$ (see Theorem~\ref{thm:filteredstabilityrange}). We keep the notations from \S\ref{subsection:mainfilteredintro}.

Consider the exact sequence \[
H^t(\fS_n, V^{d-1}_{n}) \xrightarrow{\bar{\phi}^d_{n,t}} H^t(\fS_n, V^d_n) \xrightarrow{\bar{\psi}^d_{n,t}} H^t(\fS_n, M(W_d)_n) \xrightarrow{\delta^d_{n,t}}  H^{t+1}(\fS_n, V^{d-1}_{n})
\] induced by the exact sequence \[0 \ra  V^{d-1} \xrightarrow{\phi^d}  V^d \xrightarrow{\psi^d}  M(W_d) \ra  0.\] We have the following lemma:

\begin{lemma}\label{lemma:image}
Let $\SQ$ be the sequence of length $d$ consisting of zeros. If $p^{\mathfrak{I} \circ \mathfrak{D}(\SQ)} \mid a$ (needed for the maps to be well-defined), then the diagram in Figure~\ref{Fig:16} commutes. Moreover, if $n-a \geq 2 t + m_d$, then the right vertical arrow is an isomorphism.
\begin{figure}[h]
\centering
\begin{tikzcd}[column sep = normal, row sep =normal] 
H^t(\fS_n, V^{d-1}_{n}) \ar{rr}{\bar{\phi}^d_{n,t}} \ar{d}{\mathfrak{R}^{n,a,0}_{t,V^{d-1}}} & & H^{t}(\fS_n,V_n^d) \ar{rr}{\bar{\psi}^d_{n,t}} \ar{d}{\mathfrak{R}^{n,a,0}_{t,V^d}} & & H^{t}(\fS_n,M(W_d)_n) \ar{d}{\mathfrak{R}^{n,a,0}_{t,M(W_d)}} \\ 
H^t(\fS_{n-a}, V^{d-1}_{n-a}) \ar{rr}{\bar{\phi}^d_{n-a, t}} & & H^{t}(\fS_{n-a},V_{n-a}^d) \ar{rr}{\bar{\psi}^d_{n-a, t}}  & & H^{t}(\fS_{n-a},M(W_d)_{n-a})
\end{tikzcd}
\caption{} \label{Fig:16}
\end{figure} 
\end{lemma}
\begin{proof}
The first statement about the diagram is immediate because the maps $\phi^d$ and $\psi^d$ are sequential. Also, note here that if $l \in \Hom_{\fS_{n}}(\mathcal{B}_{t}(\fS_n), V_n^{d-1})$ is a $0$-cycle and $\tilde{l}$ is one of its nice lifts then, $\tilde{\phi}^d_{n} \circ \tilde{l}$ is already a nice lift of the $0$-cycle $\phi^d_n \circ l$. A similar statement holds for $\psi^d$. Hence we have a simpler condition on $a$ then in Lemma~\ref{lemma:sequentialcommutativity}. For the second statement, note that the right vertical map is induced by the map \[R_{t,M(W_d)}^{n,a}: \Hom_{\fS_{n}}(\mathcal{B}_{t}(\fS_n), M(W_d)_n) \ra \Hom_{\fS_{n-a}}(\mathcal{B}_{t}(\fS_{n-a}), M(W_d)_{n-a})\] coming from the $\FIsharp$ structure on $M(W_d)$. \begin{figure}[h]
\centering
\begingroup
\setlength{\belowcaptionskip}{1000pt}
    \fontsize{9.5pt}{12pt}\selectfont
\begin{tikzcd}[column sep =normal, row sep=large] 
\Hom_{\fS_{n}}(\mathcal{B}_{t}(\fS_n), M(W_d)_n) \ar{rr}{R_{t,M(W_d)}^{n,a}} \ar{d}{\downarrow_{m_d}^n} &   &
\Hom_{\fS_{n-a}}(\mathcal{B}_{t}(\fS_{n-a}), M(W_d)_{n-a}) \ar{d}{\downarrow_{m_d}^{n-a}} \\\Hom_{\fS_{m_d} \times \fS_{n-m_d}}(\mathcal{B}_{t}(\fS_n), W_d) \ar{r}{\res_1} & \Hom_{\fS_{m_d} \times \fS_{n-m_d-a}}(\mathcal{B}_{t}(\fS_n), W_d) \ar{r}{\res_2} & \Hom_{\fS_{m_d} \times \fS_{n-m_d-a}}(\mathcal{B}_{t}(\fS_{n-a}), W_d)
\end{tikzcd}
\endgroup
\caption{} \label{Fig:17}
\end{figure}

Consider the commutative diagram in Figure~\ref{Fig:17} (where $\res_1$ and $\res_2$ are natural restriction maps). Vertical arrows are isomorphisms and $\res_2$ induces an isomorphism on the cohomology groups. Hence it is enough to show that $\res_1$ induces an isomorphism. By K{\"u}nneth formula, $\res_1$ induces the map \begin{align*}(id_i \otimes \mathfrak{R}_{j,M(0)}^{n-m_d,a,0})_{i+j = t} : \bigoplus_{i+j = t} H^i(\fS_{m_d}, W_d)&\otimes H^j(\fS_{n-m_d},\bk) \ra \\ &\bigoplus_{i+j = t} H^i(\fS_{m_d}, W_d)\otimes H^j(\fS_{n-m_d-a},\bk)\end{align*} up to isomorphisms $\bigoplus_{i+j = t} H^i(\fS_{m_d}, W_d)\otimes H^j(\fS_{n-m_d},\bk) \cong H^t(\fS_{m_d} \times \fS_{n-m_d},W \boxtimes \bk)$ and $\bigoplus_{i+j = t} H^i(\fS_{m_d}, W_d)\otimes H^j(\fS_{n-m_d-a},\bk) \cong H^t(\fS_{m_d} \times \fS_{n-m_d-a},W \boxtimes \bk)$. Thus it is enough to show that the map $\mathfrak{R}_{j,M(0)}^{n-m_d,a,0}$ induces isomorphism for $n-a \geq 2 j + m_d$, but that is the content of Nakaoka's stability theorem (see equation~\ref{eqn:nakaoka}).
\end{proof}

\begin{lemma}\label{lemma:delta}
Let $\SQ$ be the sequence of length $d-1$ consisting of zeros. If $p^{\mathfrak{I}_{V^d} \circ \phi^d_{\star} \circ  \mathfrak{D}_{V^{d-1}}(\SQ)} \mid a$, then the diagram in Figure~\ref{Fig:18} commutes.

\begin{figure}[h]
\centering
\begin{tikzcd}
H^{t}(\fS_n,M(W_d)_n) \ar{rr}{\delta^d_{n,t}} \ar{d}{\mathfrak{R}^{n,a,0}_{t,M(W_d)}} & & H^{t+1}(\fS_n,V_n^{d-1}) \ar{d}{\mathfrak{R}^{n,a,0}_{t+1,V^{d-1}}} \\ H^{t}(\fS_{n-a},M(W_d)_{n-a}) \ar{rr}{\delta^d_{n-a,t}}  & & H^{t+1}(\fS_{n-a},V^{d-1}_{n-a})
\end{tikzcd}
\caption{} \label{Fig:18}
\end{figure}
\end{lemma}

\begin{remark}
Note here that for the right vertical arrow to be well-defined we need that $p^{\mathfrak{I}_{V^{d-1}}\circ \mathfrak{D}_{V^{d-1}}(\SQ)} \mid a$. But it is easy to see that \[p^{\mathfrak{I}_{V^d} \circ \mathfrak{D}_{V^d} \circ \phi^d_{\star}(\SQ)} \geq p^{\mathfrak{I}_{V^d} \circ \phi^d_{\star} \circ  \mathfrak{D}_{V^{d-1}}(\SQ)} \geq p^{\mathfrak{I}_{V^{d-1}}\circ \mathfrak{D}_{V^{d-1}}(\SQ)}.\] In the following proof we will not write the obvious subscripts.
\end{remark}

\begin{proof}
Note that it is enough to show that the natural map $\delta_n' : \ker{\bar{\phi}^d_{n,t+1}}  \ra \coker{\bar{\psi}^d_{n,t}}$ that induces $\delta^d_{n,t}$, commutes with $\mathfrak{R}$. For that, let $l \in \Hom_{\fS_{n}}(\mathcal{B}_{t+1}(\fS_n), V_n^{d-1})$ be a $0$-cycle such that $\phi^d_n \circ l$ is a boundary. Since $\delta$ is a well-defined map on cohomology, we may assume, by picking an appropriate element in the equivalence class of $l$, that $l$ admits a lift $\tilde{l}$ that is nice. So we have, $\Pi^{d}_n \circ \tilde{\phi}^d_{n} \circ \tilde{l} = b \circ \partial_{t}$ for some $b$. Then by Claim~\ref{claim:one} and sequentialness of $\phi^d$, $\tilde{\phi}^d_{n} \circ \tilde{l}$ is $\phi^d_{\star} \circ \mathfrak{D}(\SQ)$-periodic. By the nice lift construction, we may find a nice lift $\tilde{b}$ of the $\tilde{\phi}^d_{n} \circ \tilde{l}$ cycle $b$. Now note that $\delta_n'([l]) = [\pi^d \circ \tilde{\psi}^d_{n} \circ \tilde{b}]$. Hence it follows from Claim~\ref{claim:two}, that the diagram in Figure~\ref{Fig:18} commutes if $p^{\mathfrak{I} \circ \phi^d_{\star} \circ  \mathfrak{D}(\SQ)} \mid a$.
\end{proof}

\begin{proof}[{\bf Proof of Theorem~\ref{thm:filteredstabilityrange}}]
Proof is by induction on $d$ which is the length of the $\sharp$-filtration. When $d=1$ then the proof follows from the second assertion of Lemma~\ref{lemma:image}. In the general case, by Lemma~\ref{lemma:image} and \ref{lemma:delta}, the diagram in Figure~\ref{Fig:19} commutes.
\begin{figure}[h]
\centering

\resizebox{\textwidth}{!}{\begingroup
\setlength{\belowcaptionskip}{1000pt}
   \fontsize{8.0pt}{12pt}\selectfont \begin{tikzcd}[row sep =huge, column sep =scriptsize, ampersand replacement = \&] H^{t-1}(\fS_n,M(W_d)_n) \ar{d}{\mathfrak{R}^{n,a,0}_{t-1,M(W_d)}} \ar{rr}{\delta^d_{n,t-1}} \& \& H^t(\fS_n, V^{d-1}_{n}) \ar{rr}{\bar{\phi}^d_{n,t}} \ar{d}{\mathfrak{R}^{n,a,0}_{t,V^{d-1}}} \& \& H^{t}(\fS_n,V_n^d) \ar{rr}{\bar{\psi}^d_{n,t}} \ar{d}{\mathfrak{R}^{n,a,0}_{t,V^d}} \& \& H^{t}(\fS_n,M(W_d)_n) \ar{d}{\mathfrak{R}^{n,a,0}_{t,M(W_d)}} \ar{rr}{\delta^d_{n,t}}  \& \& H^{t+1}(\fS_n,V_n^{d-1}) \ar{d}{\mathfrak{R}^{n,a,0}_{t+1,V^{d-1}}} \\ H^{t-1}(\fS_{n-a},M(W_d)_{n-a}) \ar{rr}{\delta^d_{n-a,t-1}}  \& \&
H^t(\fS_{n-a}, V^{d-1}_{n-a}) \ar{rr}{\bar{\phi}^d_{n-a, t}} \& \& H^{t}(\fS_{n-a},V_{n-a}^d) \ar{rr}{\bar{\psi}^d_{n-a, t}}  \& \& H^{t}(\fS_{n-a},M(W_d)_{n-a}) \ar{rr}{\delta^d_{n-a,t}}  \& \& H^{t+1}(\fS_{n-a},V^{d-1}_{n-a})
\end{tikzcd}
\endgroup}
\caption{} \label{Fig:19}
\end{figure}
Note that the rows are exact. First and the fourth vertical arrows are isomorphism by the base case. Let $\SQ'$ be the sequence of length $d-1$ consisting of zeros. Since $p^{\mathfrak{I}_{V^d} \circ \mathfrak{D}_{V^d}(\SQ)} \geq p^{\mathfrak{I}_{V^{d-1}}\circ \mathfrak{D}_{V^{d-1}}(\SQ')}$, the induction hypothesis holds. Hence the second and the fifth vertical arrows are isomorphisms. The proof now follows from the five lemma.
\end{proof}

\begin{proof}[{\bf Proof of Theorem~\ref{thm:twistedfilteredstabilityrange}}]
Let $\SQ'$ be the sequence of length $d$ consisting of zeros. The proof is immediate from the fact that $p^{\mathfrak{I}\circ \mathfrak{D}(\SQ)} \geq p^{\mathfrak{I}\circ \mathfrak{D}(\SQ')}$, Theorem~\ref{thm:filteredstabilityrange} and the commutativity of the diagram in Lemma~\ref{lemma:untwistingdiagram}.
\end{proof}

We have the following bound on the period.

\begin{lemma}
\label{lemma:boundonM}
Let $\SQ = (H^{i,d})_{1 \leq i \leq d} \in \mathbb{Z}^d_{\geq 0}$ and define  $D_1\coloneq\max_{1 \leq i \leq d}{H^{i,d}}$. Then, \begin{align*} \mathfrak{I}(\SQ) & \leq D_1 +  D, \\
\mathfrak{D}(\SQ)_i & \leq D_1 +  D - m_i \qquad \text{and } \\
\mathfrak{I}\circ \mathfrak{D}(\SQ) & \leq D_1 + 2 D.
\end{align*} In particular, the cohomology groups $H^t(\fS_n, V^d_n)$ are eventually periodic in $n$ with period dividing $p^{2 D}$.
\end{lemma} 
\begin{proof}
We have $\mathfrak{I}\circ \mathfrak{D}(\SQ)=\max_{1 \leq i \leq d} (\mathfrak{D} \circ \mathfrak{D}(\SQ)_i + \ddH (m_i,0))$. It is elementary that if $b$ is positive integer, then $v_p(b!) = \sum_{j=1}^{\infty} \lfloor{\frac{b}{p^j}}\rfloor \leq b-1$. This implies $\ddH (a,b) \leq \max(a-b,0)$. Moreover, the recursive definition of $H^{i,r-1}$ implies by induction on $r$ that $H^{i,r-1} \leq D_1 + D - m_i$. Thus $\mathfrak{D}(\SQ)_i =  H^{i,i} \leq D_1 + D-m_i$. This proves the second assertion and the first follows immediately. Iterating the argument we get, $\mathfrak{D}\circ \mathfrak{D}(\SQ)_i \leq D_1 + 2 D - m_i$. Hence $\mathfrak{I}\circ \mathfrak{D}(\SQ) \leq D_1 + 2 D$, completing the proof.
\end{proof}

We record the following lemma to be used in the next section.

\begin{lemma}
\label{lem:boundphi}
Let $\Pi^1: \bigoplus_{i=1}^{d_1} M(m_i) \ra V$ and $\Pi^2: \bigoplus_{k=1}^{d_2} M(n_k) \ra W$ be $\sharp$-filtered $\FI$-modules. Define $D_1 \coloneq \max_{1 \leq i \leq d}{m_i}$ and $D_W \coloneq \max_{1 \leq k \leq d_2}{n_k}$. Let $\phi : V \ra W$ be a sequential map (see Definition~\ref{def:sequential} and Definition~\ref{def:phi-star}) and let $\SQ \in \mathbb{Z}^{d_1}_{\geq 0}$. Then $\mathfrak{D}(\phi_{\star}(\SQ))_k \leq D_1 + D_{W} - n_k$ for each $1 \leq k \leq d_2$.
\end{lemma}
\begin{proof}
Proof follows immediately from Lemma~\ref{lemma:boundonM}.
\end{proof}

\begin{xample}
\label{example:1}
Working over $\mathbb{F}_2$, for any $d \geq 3$ let $V$ be the quotient of $M(0) \oplus M(d)$ by
the sub $\FI$-module generated by the element $(1, \sum_{f : [d] \rightarrow [d]}
f) \in (M(0) \oplus M(d))_d$. Note that $V$ admits a natural exact sequence \begin{equation}0
\rightarrow M (0) \rightarrow V \rightarrow M (W) \rightarrow 0 \label{eqn:doesnotsplit} \end{equation} where $W$ is
the quotient of $M (d)$ by the sub $\FI$-module $K$ generated by $\sum_{f : [d]
\rightarrow [d]} f \in M (d)_d$.

It is clear that the image of $M(0)_n$ lie inside $H^0(\fS_n, V_n) = V_n^{\fS_n} \subseteq V_n$. We show that $\dim{M (W)_n^{\fS_n}} = \dim{W^{\fS_d}} = 1$ and hence
$\dim{V_n^{\fS_n}}$ is either $1$ or $2$ depending on $n$. The first equality is clear by Shapiro's lemma. For the
second, let $0 \neq x \in W^{\fS_d}$ and let $y \in M (d)_d$ be a lift of
$x$. By construction, if $\sigma \in \fS_d$ then $\sigma y$ is either $y$ or $y
+ \sum_{f : [d] \rightarrow [d]} f$. This define a surjective homomorphism
$\phi : \fS_d \rightarrow \{0, 1\}$. By simplicity of $A_d$ (or when $d=4$, the fact that the only index two subgroup of $\fS_4$ is $A_4$), we have $\ker
(\phi) = A_d$. This shows that $x = \sum_{f \in A_d} f \mod{K}_d$ and hence
the second equality.

Now we calculate the dimension of $V_n^{\fS_n}$. Note that $\dim{V_n^{\fS_n}}=2$ if and only if there exists $0 \neq x \in V_n^{\fS_n}$ admitting a lift, say $y = (y_1, y_2) \in (M (0) \oplus M (d))_n$, such that $y_2 \neq 0 \mod{K_n}$. In that case, by the previous paragraph we may take $y_2 = \sum_{g \in D_{d,n}} \sum_{f \in A_d} g \circ f$ and this implies \[\sigma x = x \qquad \Longleftrightarrow \qquad  \Gamma_n(\sigma) \coloneq|\{g \in D_{d,n}: A_d \not\owns \gamma_{\sigma(g)}^{-1}\sigma \gamma_g: [d] \ra [d] \}| \equiv 0 \mod{2} \] for any $\sigma \in \fS_n$. Note that the transposition $\sigma_0 = (1, 2)$ and the cycle $\sigma_1 = (1, 2, 3, \ldots ., n)$ generate $\fS_n$ and \[\Gamma_n(\sigma) = \begin{cases}
\binom{n-2}{d-2}, &\text{if } \sigma = \sigma_0  \\
\binom{n-1}{d-1}, &\text{if } \sigma = \sigma_1 \text{ and } 2 \mid d \\
0 , &\text{if } \sigma = \sigma_1 \text{ and } 2 \nmid d.
\end{cases}\] We conclude that \[\dim{V_n^{\fS_n}} =  \begin{cases}
2 \Leftrightarrow \binom{n-2}{d-2} \equiv \binom{n-1}{d-1} \equiv 0
\mod{2}, &\text{ if } 2 \mid d \text{ and, } \\ 
2 \Leftrightarrow \binom{n - 2}{d - 2} \equiv 0 \mod{2}, &\text{ if } 2 \nmid d.
\end{cases}\]

The residue $\binom{n}{x} \mod{2}$ is periodic in $n$ with the smallest period $\geq x$. This implies that the smallest period of $\dim{H^0(\fS_n, V_n)}$ could be an arbitrarily large power of $2$. In particular, when $d = 5$,
$\dim V_n^{\fS_n}$ is eventually periodic in $n$ with smallest period $2^2$. Our argument also shows that \eqref{eqn:doesnotsplit} does not split because otherwise the smallest period would be the $\lcm$ of periods for $H^0(\fS_n, M(0)_n)$ and $H^t(\fS_n, M(W)_n)$ which are both $1$.
\end{xample}

\subsection{Proof of the main theorem for finitely generated \texorpdfstring{$\FI$}{FI}-modules}
\label{subsection:proofofmain}
The aim of this section is to prove the main theorem for finitely generated $\FI$-modules (Theorem~\ref{thm:maintheorem}) and develop machinery for the generalization to $\FI$-complexes which is the content of \S\ref{subsection:generalization}. We keep the notations from \S\ref{subsection:themainthmintro}.

\begin{proof}[{\bf Proof of Theorem~\ref{thm:maintheorem}}]
Assume $n-a \geq C$ and consider the spectral sequences $E^{\bullet, \bullet}(n)$ defined in \eqref{eqn:E}. Note that the rightward oriented first page is given by \[_{\ra}E^{x,y}_1(n) = \begin{cases}
0, & \text{if } x>0  \\
\Hom_{\fS_{n}}(\mathcal{B}_{y}(\fS_n), V_n), & \text{if } x=0 \\
\end{cases}\] and it follows that the rightward oriented second page satisfy \[_{\ra}E^{x,y}_2(n) =   {_{\ra}E}^{x,y}_{\infty}(n)= \begin{cases}
0, & \text{if } x>0  \\
H^y(\fS_n,  V_n), & \text{if } x=0 \\
\end{cases}\] Hence the $y$-th cohomology group of the associated total complex is isomorphic to $H^y(\fS_n,  V_n)$. Thus, if we change orientation, then the upward oriented spectral sequence $_{\uparrow}E^{\bullet, \bullet}_r(n)$ abuts to $H^{\bullet}(\fS_n, V_n)$ and $_{\uparrow}E^{x, t-x}_{\infty}(n)$, $0 \leq x \leq N$ are graded pieces of the induced filtration on $H^{t}(\fS_n, V_n)$. We show by induction on the page number $r$ that if $a$ is divisible by a large enough power of $p$ and $n-a$ is sufficiently large then there is an isomorphism $_{\uparrow}E^{x, y}_{r}(n) \cong {_{\uparrow}E}^{x, y}_{r}(n-a)$.

We start with the base case. Observe that ${_{\uparrow}E}^{x,y}_1(n) = H^y(\fS_n, J^x_n)$ and, by the results of the previous section (Theorem~\ref{thm:filteredstabilityrange} and Lemma~\ref{lemma:boundonM}), we have a map $\mathfrak{R}^{x,y}_1(n,a) : {_{\uparrow}E}^{x,y}_1(n) \ra {_{\uparrow}E}^{x,y}_1(n-a)$ defined by $\mathfrak{R}^{x,y}_1(n,a)\coloneq \mathfrak{R}^{n,a,0}_{y, J^x}$ that is an isomorphism  as long as $p^{2 D_x} \mid a$ and $n-a \geq 2 (y + d_x -1) +  D_x$. Here $d_x$ and $D_x$ are the lengths of the $\sharp$-filtration and the degree of generation (resp.) of $J^x$. Note that by Theorem~\ref{thm:structure}, $D_x \leq D_0 - x$. Moreover, $D_0$ is less than or equal to the degree of generation $D$ of $V$. We set $M^{x,y}_1 = 2D_x$ and $\SD^{x,y}_1 = 2 (y + d_x -1) +  D_x$.

Assume, by induction, that $\mathfrak{R}^{x,y}_1(n,a)$ induces a map $\mathfrak{R}^{x,y}_r(n,a) : {_{\uparrow}E}^{x,y}_r(n) \ra {_{\uparrow}E}^{x,y}_r(n-a)$ that is an isomorphism as long as $p^{M^{x,y}_r} \mid a$ and $n-a \geq \SD^{x,y}_r$. Consider the diagram in Figure~\ref{Fig:20}.
\begin{figure}[h]
\centering
\begin{tikzcd}[row sep = large, column sep=38pt]
{_{\uparrow}E}^{x-r,y+r-1}_r(n) \ar{rr}{{_{\uparrow}d}^{x-r,y+r-1}_r(n)} \ar{d}{\mathfrak{R}^{x-r,y+r-1}_{r}(n,a)} & & {_{\uparrow}E}^{x,y}_r(n) \ar{rr}{{_{\uparrow}d}^{x,y}_r(n)} \ar{d}{\mathfrak{R}^{x,y}_{r}(n,a)} & & {_{\uparrow}E}^{x+r,y-r+1}_r(n) \ar{d}{\mathfrak{R}^{x+r,y-r+1}_{r}(n,a)} \\ 
{_{\uparrow}E}^{x-r,y+r-1}_r(n-a) \ar{rr}{{_{\uparrow}d}^{x-r,y+r-1}_r(n-a)} & & {_{\uparrow}E}^{x,y}_r(n-a) \ar{rr}{{_{\uparrow}d}^{x,y}_r(n-a)}  & & {_{\uparrow}E}^{x+r,y-r+1}_r(n-a)
\end{tikzcd}
\caption{} \label{Fig:20}
\end{figure} We find a condition on $n$ and $a$ such that this diagram commutes and the vertical arrows are isomorphisms. 

Following Vakil's notes \cite[\S1.7]{RV}, we define an $(x,y)$ strip associated to a spectral sequence $E^{\bullet, \bullet}$ to be an element of $\oplus_{i\geq 0} E^{x+i,y-i}$ (This is same as the Vakil's definition but with orientation changed to upward instead of rightward). Note that, in our case, the differential $d(n) = {_{\ra}d}(n) + {_{\uparrow}d}(n)$ sends a $(x,y)$ strip $s$ to a $(x,y+1)$ strip $ds$. We shall now use Vakil's description of the differential ${_{\uparrow}d}_r$ in terms of $(x,y)$ strips throughout.

Let $\bar{l} \in {_{\uparrow}E}^{x,y}_r(n)$. Then there is a corresponding $(x,y)$ strip $s = (l_i) \in \oplus_{i\geq 0} E^{x+i,y-i}(n)$ such that $d s$ is a $(x+r, y-r+1)$ strip. This implies that $l_0$ is a $0$ cycle and $\phi^{x+i}_n \circ l_i = l_{i+1} \circ \partial_{y-i}$, $0 \leq i \leq r-2$. Moreover, the image of $\bar{l}$ under ${_{\uparrow}d}^{x,y}_r(n)$ is determined by $\phi^{x+r-1}_n \circ l_{r-1}$. The well-definedness of ${_{\uparrow}d}_r^{x,y}(n)$ allows us to change each $l_i$ up to a boundary. Hence, by induction on $i$, we may assume that the $\tilde{\phi}^{x+i} \circ \tilde{l}_i$ cycle $l_{i+1}$ admits a lift $\tilde{l}_{i+1}$ that is nice. Let $\SQ^{d_x}$  be the sequence of length $d_x$ consisting only of zeros. It follows from Claim~\ref{claim:one} that $\tilde{l}_0$ is $\SQ^{x,y}_{r,0} \coloneq \mathfrak{D}(\SQ^{d_x})$-periodic and that $\tilde{l}_{i+1}$ is $\SQ^{x,y}_{r,i+1} \coloneq \mathfrak{D}(\phi^{x+i}_{\star}(\SQ^{x,y}_{r,i}))$-periodic. Define $N^{x,y}_r \coloneq \max\{\mathfrak{I}(\SQ^{d_x}), \max_{0 \leq i \leq r-1} \mathfrak{I}(\phi^{x+i}_{\star}(\SQ^{x,y}_{r,i}))\}$. Then, by Claim~\ref{claim:two} and the definition of the map $\mathfrak{R}_1$ (recall that $R$ commutes with a sequential map, see Remark~\ref{remark:RcommuteswithSequential}), the right square in the diagram above commutes if $p^{N^{x,y}_r} \mid a$. We now define
\begin{eqnarray}
M^{x,y}_{r+1} &\coloneq& \max \{M^{x,y}_r, M^{x-r,y+r-1}_r, M^{x+r,y-r+1}_r, N^{x,y}_r, N^{x-r,y+r-1}_r \}, \label{eqn:recursionM}\\
\SD^{x,y}_{r+1} &\coloneq& \max\{ \SD^{x,y}_r, \SD^{x-r,y+r-1}_r, \SD^{x+r,y-r+1}_r \} \label{eqn:recursionSD}
\end{eqnarray}

\noindent and note that the diagram above commutes and vertical arrows are isomorphisms as long as $p^{M^{x,y}_{r+1}} \mid a$ and $n-a \geq \SD^{x,y}_{r+1}$. Hence, under these assumptions on $n$ and $a$, $\mathfrak{R}^{x,y}_r(n,a)$ induces an isomorphism $\mathfrak{R}^{x,y}_{r+1}(n,a): {_{\uparrow}E}^{x,y}_{r+1}(n) \ra {_{\uparrow}E}^{x,y}_{r+1}(n-a)$, completing the inductive step of the proof. Now taking $r \ra \infty$ completes the proof of the theorem.
\end{proof}

\begin{remark} \label{remark:stabilityalgomain}
 The proof above also provides an algorithm to calculate the period $p^{M_{\infty}^t}$ and the stable range $\SD_{\infty}^t$ because we have \eqref{eqn:recursionM} and \eqref{eqn:recursionSD} together with the following equations.
\begin{align}
M_{\infty}^t &= \max_{x+y = t, 1 \leq r < \infty}{M^{x,y}_r}, \nonumber \\
\SD_{\infty}^t &= \max_{x+y = t, 1 \leq r < \infty}{\SD^{x,y}_r}. \nonumber \qedhere
\end{align}
\end{remark}

The following lemma gives bounds on the period and the stability range in Theorem~\ref{thm:maintheorem}.

\begin{lemma} \label{lem:bound-main} We have
\begin{align*}
M_{\infty}^t & \leq \min\{(t+3)D, \max\{2 D, D(D+1)/2 \} \}, \\
\SD_{\infty}^t & \leq 2(t+ \max_x d_x - 1) +D.
\end{align*} (by Remark~\ref{rem:dimension-generation}, we can replace $D$  above by $\chi(V)$)
\end{lemma}
\begin{proof}
Recall that $D_x \leq D-x$. This implies $M_1^{x,y} \leq 2 D$. Next we provide bounds on $M_r^{x,y}$ and $N_r^{x,y}$ independent of $x$ and $y$. Note that by Lemma~\ref{lemma:boundonM} and Lemma~\ref{lem:boundphi}, we have \begin{align*}
\mathfrak{I}(\SQ^{d_x}) & \leq D_x \qquad \text{and}\\
\mathfrak{I}(\phi^{x+i}_{\star}(\SQ^{x,y}_{r,i})) & \leq \sum_{y=x}^{x+i+1} D_y
\end{align*} which show that \[N_r^{x,y} \leq \begin{cases}
 (r+1)D & \text{if } r< D \\
D(D+1)/2 & \text{if } r \geq D
\end{cases}\] This implies, by induction on $r$, that \[M_{r+1}^{x,y} \leq \begin{cases}
 (r+1)D & \text{if } r< D \\
\max\{2 D, D(D+1)/2\} & \text{if } r \geq D
\end{cases}\] We know that if $x+y = t$ the spectral sequence stabilizes when $r = t+2$. This shows that $M^t_{\infty} \leq \min\{(t+3)D, \max\{2 D, D(D+1)/2 \} \}$. Proof of the second assertion is similar.
\end{proof}

Theorem~\ref{thm:maintheorem} establishes the periodicity of dimensions of the groups $H^t(\fS_n, V_n)$ but does not give explicit maps $H^t(\fS_n, V_n) \ra H^t(\fS_{n-a}, V_{n-a})$. Instead it provides us with a filtration on $H^t(\fS_n, V_n)$ and $H^t(\fS_{n-a}, V_{n-a})$ such that the graded pieces are isomorphic. In the remaining of this section we remedy this and construct explicit maps at the expense of increasing the period slightly.

Let \[\TE^{\bullet}(n) \coloneq \bigoplus_{x+y=\bullet} \Hom_{\fS_{n}}(\mathcal{B}_{y}(\fS_n), J^x_n)\] with the differential $d_T(n)$ denote the total complex associated to $E^{\bullet, \bullet}(n)$. For each $0 \leq x \leq N$, let $\Pi^x: \tilde{J}^x = \bigoplus_{i=1}^{d_x} M(m_{i,x}) \twoheadrightarrow J^x$ be a cover as in Theorem~\ref{thm:structure}. We define \[\tilde{\TE}^{\bullet}(n)\coloneq \bigoplus_{x+y=\bullet} \Hom_{\fS_{n}}(\mathcal{B}_{y}(\fS_n), \tilde{J}^x_n)\] to be the cover of $\TE^{\bullet}$ with the obvious surjective map, \[\vec{\Pi}_n^t :\tilde{\TE}^{t}(n) \ra \TE^{t}(n)\] given by $\vec{\Pi}_n^t((l_x)_{0 \leq x \leq t})= (\Pi_n^x \circ l_x)_{0 \leq x \leq t}$.

\begin{definition}
Let $\vec{z}= (z_x)_{0 \leq x \leq t+1} \in \tilde{\TE}^{t+1}(n)$. We call $\vec{l} = (l_x)_{0 \leq x \leq t} \in \TE^t(n)$ a {\bf twisted $\vec{z}$-cycle} if we have \[d_T^t(n)(\vec{l}) = \vec{\Pi}^{t+1}_n(\vec{z}).\] Two twisted $\vec{z}$ cycles $\vec{l}_1$ and $\vec{l}_2$ are equivalent if they differ by a boundary, that is, $\vec{l}_1 - \vec{l}_2 = d_T^{t-1}(n)(\vec{b})$ for some $\vec{b} \in \TE^{t-1}(n)$. We denote the set (which is in fact a group when $\vec{z}=0$) formed by these equivalence classes by $H^{t, \vec{z}}(\TE^{\bullet}(n))$ and the equivalence class of $\vec{l}$ by $[\vec{l}]$. Let $\SQ^{x} =(H^{i,d_x})_{1 \leq i \leq d_x}$, $0 \leq x \leq N$ be sequences of nonnegative integers. We say that $\vec{z}$ is {\bf $(\SQ^x)_{0 \leq x \leq t+1}$-periodic} if $z_x$ is $\SQ^x$-periodic for each $x$.
\end{definition}

\begin{remark} \label{remark:iotaembedding} Note that the map \[\iota_{\star} : \Hom_{\fS_{n}}(\mathcal{B}_{t}(\fS_n), V_n) \ra \TE^t(n) = \bigoplus_{x+y=t} \Hom_{\fS_{n}}(\mathcal{B}_{y}(\fS_n), J^x_n)\] given by $l \mapsto (l_x)_{0 \leq x \leq t}$ where

$$
l_x =
\begin{cases}
\iota_n \circ l, & \text{if }x=0 \\
0, & \text{otherwise}
\end{cases}
$$

\noindent induces an isomorphism $\bar{\iota}_{\star} \colon H^t(\fS_n, V_n) \ra H^{t}(\TE^{\bullet}(n))=H^{t, \vec{0}}(\TE^{\bullet}(n))$ for $n \geq C$.
\end{remark}

Define the map \[\vec{R}_{t+1 , \tilde{\TE^{\bullet}}}^{n , a} : \tilde{\TE}^{t+1}(n) \ra \tilde{\TE}^{t+1}(n-a)\] by $(z_x)_{0 \leq x \leq t+1} \mapsto (R_{t+1-x}(z_x))_{0 \leq x \leq t+1}$. We will drop some of the subscripts and superscripts when there is no discrepancy. Let $\vec{l} =(l_x)_{0 \leq x \leq t} \in \TE^{t}(n)$ be a $\vec{z}$ cycle and assume that $\vec{z}$ is $(\SQ^x)_{0 \leq x \leq t+1}$-periodic. This amounts to the following equations: \begin{eqnarray}
 l_0 \circ \partial_{t+1} &=& \Pi^0_n \circ z_0 \nonumber \\
 l_{x+1} \circ \partial_{t-x} &=& \phi^x_n \circ l_x + \Pi^{x+1}_n \circ z_{x+1}, \text{ for }0 \leq x \leq t-1 \nonumber \\
 0 &=& \phi^t_n \circ l_t + \Pi^{t+1}_n \circ z_{t+1} \nonumber
\end{eqnarray}

By changing each $l_x$ upto a boundary we may assume, by induction on $x$, that the $\tilde{\phi}^{x} \circ \tilde{l}_x + z_{x+1}$ cycle $l_{x+1}$ admits a lift $\tilde{l}_{x+1}$ that is nice. It follows that $\tilde{l}_0$ is $\SQ^{t,0} \coloneq \mathfrak{D}(\SQ^0)$-periodic and that $\tilde{l}_{x+1}$ is $\SQ^{t, x+1}\coloneq \mathfrak{D}(\gcd(\phi^x_{\star}(\SQ^{t,x}), \SQ^{x+1}))$-periodic. We call $\vec{\tilde{l}}\coloneq(\tilde{l}_x)_{0 \leq x \leq t}$ a {\bf nice lift} of $\vec{l}$. By Claim~\ref{claim:two}, $\vec{\Pi}_{n-a}^t(\vec{R}_{t}((\tilde{l}_x)_{0 \leq x \leq t}))$ is a $\vec{R}_{t+1}(\vec{z})$ cycle if $p^{\vec{\mathfrak{I}}^t_{\TE^{\bullet}}((\SQ^x)_{0 \leq x \leq t+1})} \mid a$. Here the maps $\vec{\mathfrak{I}}^t_{\TE^{\bullet}}$ and $\vec{\mathfrak{D}}^t_{\TE^{\bullet}}$ are analogous to the maps $\mathfrak{I}$ and $\mathfrak{D}$ (resp.)  and are defined by \begin{align} \vec{\mathfrak{I}}^t_{\TE^{\bullet}}((\SQ^x)_{0 \leq x \leq t+1})&\coloneq \max\{\mathfrak{I}(\SQ^0), \max_{0 \leq x \leq t} \mathfrak{I}(\gcd(\phi^x_{\star}(\SQ^{t,x}), \SQ^{x+1}))\} \text{ and}, \nonumber\\ \vec{\mathfrak{D}}^t_{\TE^{\bullet}}((\SQ^x)_{0 \leq x \leq t+1}) &\coloneq (\SQ^{t,x})_{0 \leq x \leq t}. \nonumber
\end{align} Note that $\vec{\tilde{l}}$ is $\vec{\mathfrak{D}}^t_{\TE^{\bullet}}((\SQ^x)_{0 \leq x \leq t+1})$-periodic. We can now prove the following analog of Claim~\ref{claim:welldefinedness}.

\begin{claim}
Let $\vec{z}$ be $(\SQ^x)_{0 \leq x \leq t+1}$-periodic and assume that $p^{\vec{\mathfrak{I}}^{t-1}_{\TE^{\bullet}} \circ \vec{\mathfrak{D}}^t_{\TE^{\bullet}}((\SQ^x)_{0 \leq x \leq t+1})} \mid a$ and $p^{\vec{\mathfrak{I}}^t_{\TE^{\bullet}}((\SQ^x)_{0 \leq x \leq t+1})} \mid a$. Then the map \[\vec{\mathfrak{R}}^{n,a,\vec{z}}_{t, \TE^{\bullet}} : H^{t, \vec{z}}(\TE^{\bullet}(n)) \ra H^{t, \vec{R}_{t+1}(\vec{z})}(\TE^{\bullet}(n-a))\] defined by $[\vec{l}] \mapsto [\vec{\Pi}_{n-a}^t(\vec{R}_{t}((\tilde{l}_x)_{0 \leq x \leq t}))]$ is well-defined.
\end{claim}

\begin{proof}
Let $\vec{l}_j \in [\vec{l}]$, $j \in \{1,2\}$. Let $\vec{\tilde{l}}_j$ be any of their nice lifts. We have $\vec{l}_1 -\vec{l}_2 = d_T^{t-1}(n)(\vec{b})$ for some $\vec{b} \in \TE^{t-1}(n)$. Also, $\vec{\Pi}^t_{n}( \vec{\tilde{l}}_j) = \vec{l}_j + d_T^{t-1}(n)(\vec{b}_j)$ for some $\vec{b}_j \in \TE^{t-1}(n)$. Combining these equations yields $\vec{\Pi}^t_{n} (\vec{\tilde{l}}_1 - \vec{\tilde{l}}_2) = d_T^{t-1}(n)(\vec{b}+\vec{b}_1-\vec{b}_2)$. This shows that $\vec{b}_3\coloneq\vec{b}+\vec{b}_1-\vec{b}_2$ is a twisted $(\vec{\tilde{l}}_1 - \vec{\tilde{l}}_2)$-cycle. Let $\vec{\tilde{b}}_3$ be one of its nice lifts. By what we have observed in the paragraph above the lemma, $(\vec{\tilde{l}}_1 - \vec{\tilde{l}}_2)$ is $\vec{\mathfrak{D}}^t_{\TE^{\bullet}}((\SQ^x)_{0 \leq x \leq t+1})$-periodic. Hence $\vec{\Pi}^{t-1}_{n-a}( \vec{R}_{t-1, n, a} (\vec{\tilde{b}}_3))$ is a $\vec{R}_{t,n,a}(\vec{\tilde{l}}_1 - \vec{\tilde{l}}_2)$ cycle if $p^{\vec{\mathfrak{I}}^{t-1}_{\TE^{\bullet}} \circ \vec{\mathfrak{D}}^t_{\TE^{\bullet}}((\SQ^x)_{0 \leq x \leq t+1})} \mid a$. This translates to $\vec{\Pi}^t_{n-a}( \vec{R}_{t,n,a}(\vec{\tilde{l}}_1 - \vec{\tilde{l}}_2))=d_T^{t-1}(n)(\vec{\Pi}^{t-1}_{n-a}( \vec{R}_{t-1, n, a} (\vec{\tilde{b}}_3)))$, completing the proof because $d_T^{t-1}(n)(\vec{\Pi}^{t-1}_{n-a}( \vec{R}_{t-1, n, a} (\vec{\tilde{b}}_3)))$ is a boundary.
\end{proof}

Let $\vec{z}=0$ and for each $x$ let $\SQ^x$ be the sequence consisting of zeros. Define the quantities (the subscript $1$ will be useful in the next section) \begin{eqnarray}
\vec{M}^t_1 &\coloneq& \max\{M^t_{\infty},\; \vec{\mathfrak{I}}^{t-1}_{\TE^{\bullet}} \circ \vec{\mathfrak{D}}^t_{\TE^{\bullet}}((\SQ^x)_{0 \leq x \leq t+1}), \; \vec{\mathfrak{I}}^t_{\TE^{\bullet}}((\SQ^x)_{0 \leq x \leq t+1})\} \\
\vec{\SD}^t_1 &\coloneq& \SD^t_{\infty}.
\end{eqnarray} Then we have the following theorem.

\begin{theorem}
\label{thm:vectormain}
Assume $p^{\vec{M}^t_1} \mid a$ and $n-a \geq \max\{\vec{\SD}^t_1, C\}$. Then the map \[\vec{\mathfrak{R}}^{n,a,0}_{t, \TE^{\bullet}}: H^{t, 0}(\TE^{\bullet}(n)) \ra H^{t, 0}(\TE^{\bullet}(n-a))\] is an isomorphism.
\end{theorem}
\begin{proof}
Notice that for each $x$, $y$ satisfying $x+y=t$, $\vec{\mathfrak{R}}^{n,a,0}_{t, \TE^{\bullet}}$ induces the maps $\mathfrak{R}^{x,y}_{\infty}(n,a)$ on the graded piece ${_{\uparrow}E}^{x,y}_{\infty}(n)$ of $H^{t, 0}(\TE^{\bullet}(n))$ to the corresponding graded piece ${_{\uparrow}E}^{x,y}_{\infty}(n-a)$ of $H^{t, 0}(\TE^{\bullet}(n-a))$. Hence the result follows by Theorem~\ref{thm:maintheorem}. 
\end{proof}

\begin{remark}
\label{rem:bound1}
By a similar argument as in Lemma~\ref{lem:bound-main}, it can be checked that 
\begin{align*}\vec{M}_{1}^t & \leq (t+3)D, \\
\vec{\SD}_{1}^t & \leq 2(t+ \max_x d_x - 1) +D. \qedhere
\end{align*} 
\end{remark}

\subsection{Generalizations: complexes of finitely generated \texorpdfstring{$\FI$}{FI}-modules} \label{subsection:generalization}
We work with the following arbitrary complex of finitely generated $\FI$-modules with  differential $\delta$. \[0 \ra V^0 \ra V^1 \ra \ldots \ra V^x \ra \ldots \] Consider the corresponding spectral sequences $E^{\bullet, \bullet}(n)$ defined in \eqref{eqn:vecE}. Our aim in this section is to show that for any page $r \geq 1$ of the corresponding upward oriented spectral sequence and position $(x,y)$, ${_{\uparrow}E}^{x,y}_r(n)$, $\im{{_{\uparrow}d}^{x,y}_r(n)}$ and $\ker{{_{\uparrow}d}^{x,y}_r(n)}$ are periodic in $n$. Since these objects depend only on first few columns (depending on $x$ and $y$) we may assume that the complex is supported in finitely many columns say $0 \leq x \leq N$. Hence we can follow the notation from Remark~\ref{remark:structuregeneralization}.

As in Remark~\ref{remark:iotaembedding}, for a fixed $x$ we have an embedding \[\iota^x_{\star}: \Hom_{\fS_{n}}(\mathcal{B}_{\bullet}(\fS_n), V^x_n) \ra \TE^{ x, \bullet}(n) \coloneq \bigoplus_{u_x+v_x= \bullet} \Hom_{\fS_{n}}(\mathcal{B}_{v_x}(\fS_n), J^{u_x,x}_n)\] of complexes. The later complex is the total complex associated to $V^x_n$ as in previous section (see \eqref{eqn:E}) and hence is equipped with the natural differential $d_T^{x,\bullet}(n)$. This yields an embedding $\vec{\iota}_{\star} : E^{\bullet, \bullet}(n) \ra \TE^{\bullet, \bullet}(n)$ given by $\iota^x_{\star} :E^{x,y} \ra \TE^{x,y}$ for a fixed $x$. Here the differentials associated to $\TE^{\bullet, \bullet}$ are given by:
\begin{align}
_{\ra}d_T^{x,y}(n) &: \TE^{x,y}(n) \ra \TE^{x+1, y}(n), \qquad &&\text{induced by } \vec{\delta}^{x,y}\coloneq(\delta^{u_x,x})_{0 \leq u_x \leq y} \text{ and,}& \nonumber \\
{_{\uparrow}d}_T^{x,y}(n) &: \TE^{x,y}(n) \ra \TE^{x, y+1}(n), \qquad &&\text{induced by } d_T^{x,y}.& \nonumber
\end{align}

\begin{remark} 
\label{remark:totalspectral}
Note that, by Remark~\ref{remark:iotaembedding} again, $\vec{\iota}_{\star}$ induces an isomorphism on the first upward oriented page $\vec{\iota}_{\star,1}:{_{\uparrow}E}^{\bullet,\bullet}_1(n) \cong \upTE^{\bullet,\bullet}_1(n)$ and hence on any higher upward oriented page $\vec{\iota}_{\star,r}:{_{\uparrow}E}^{\bullet,\bullet}_r(n) \cong \upTE^{\bullet,\bullet}_r(n)$ given that $n\geq C\coloneq\max_{0 \leq x \leq N} C_x$. Hence for our purpose, it is enough to analyze the complex $\TE^{\bullet, \bullet}(n)$ which is secretly a triple complex with co-ordinates $u_x$, $v_x$ and $y$.
\end{remark}

We may define the vector analogs of quantities defined in previous sections as follows: \begin{align}
\tilde{\TE}^{x,y} &\coloneq \bigoplus_{u_x+v_x= y} \Hom_{\fS_{n}}(\mathcal{B}_{v_x}(\fS_n), \tilde{J}^{u_x,x}_n) \nonumber \\
\vec{\tilde{\delta}}^{x,y} &\coloneq (\tilde{\delta}^{u_x ,x})_{0 \leq u_x \leq y} \nonumber \\
\vec{\delta}^{x,y}_{\star} &\coloneq (\delta^{u_x,x}_{\star})_{0 \leq u_x \leq y} \nonumber
\end{align}

Let $\vec{l} \in \TE^{x,y}$ and let $\vec{\tilde{l}} = (\tilde{l}_{u_x})_{0 \leq u_x \leq y} \in \tilde{\TE}^{x,y}$ be a lift of $\vec{l}$. Assume that $\vec{\tilde{l}}$ is  $(\SQ^{u_x})_{0 \leq u_x \leq y}$-periodic. Then it follows by the sequentialness of the maps $\delta^{u_x, x}$, $0\leq u_x \leq y$ that $_{\ra}d_T^{x,y}(n)(\vec{l})$ has a lift $\vec{\tilde{\delta}}_n^{x,y} \circ \vec{\tilde{l}} = (\tilde{\delta}^{u_x ,x}_n \circ \tilde{l}_{u_x})_{0 \leq u_x \leq y}$ which is $\vec{\delta}^{x,y}_{\star}((\SQ^{u_x})_{0 \leq u_x \leq y}) = (\delta^{u_x,x}_{\star}(\SQ^{u_x}))_{0 \leq u_x \leq y}$-periodic. 

\begin{proof}[{\bf Proof of Theorem~\ref{thm:generalizedmain}}]
We show by induction on $r$ that if $a$ is divisible by a large enough power of $p$ and $n-a$ is sufficiently large then there is an isomorphism $_{\uparrow}\TE^{x, y}_{r}(n) \cong \upTE^{x, y}_{r}(n-a)$. The proof is similar to the proof of Theorem~\ref{thm:maintheorem}. For the base case, observe that by Theorem~\ref{thm:vectormain}, we have the map \[\vec{\mathfrak{R}}^{x,y}_{1}(n,a): \upTE^{x,y}_{1}(n) \ra \upTE^{x, y}_{1}(n-a)\] defined by $\vec{\mathfrak{R}}^{x,y}_{1}(n,a)\coloneq \vec{\mathfrak{R}}^{n,a,0}_{y, \TE^{x, \bullet}}$ that is an isomorphism as long as $p^{\vec{M}^{x,y}_1} \mid a$ and $n-a \geq \vec{\SD}^{x,y}_1$.

Assume, by induction that $\vec{\mathfrak{R}}^{x,y}_{1}(n,a)$ induces a map $\vec{\mathfrak{R}}^{x,y}_{r}(n,a): \upTE^{x,y}_{r}(n) \ra \upTE^{x, y}_{r}(n-a)$ that is an isomorphism as long as $p^{\vec{M}^{x,y}_r} \mid a$ and $n-a \geq \vec{\SD}^{x,y}_r$. Consider the diagram in Figure~\ref{Fig:21}. 

\begin{figure}[h]
\centering
\begin{tikzcd}[row sep = large, column sep = 37pt]
\upTE^{x-r,y+r-1}_r(n) \ar{rr}{{_{\uparrow}d}_{T,r}^{x-r,y+r-1}(n)} \ar{d}{\vec{\mathfrak{R}}^{x-r,y+r-1}_{r}(n,a)} & & \upTE^{x,y}_r(n) \ar{rr}{{_{\uparrow}d}_{T,r}^{x,y}(n)} \ar{d}{\vec{\mathfrak{R}}^{x,y}_{r}(n,a)} & & \upTE^{x+r,y-r+1}_r(n) \ar{d}[swap]{\vec{\mathfrak{R}}^{x+r,y-r+1}_{r}(n,a)} \\ 
\upTE^{x-r,y+r-1}_r(n-a) \ar{rr}{{_{\uparrow}d}_{T,r}^{x-r,y+r-1}(n-a)} & & \upTE^{x,y}_r(n-a) \ar{rr}{{_{\uparrow}d}_{T,r}^{x,y}(n-a)}  & & \upTE^{x+r,y-r+1}_r(n-a)
\end{tikzcd}
\caption{} \label{Fig:21}
\end{figure}

We now find a condition on $n$ and $a$ such that this diagram commute and the vertical arrows are isomorphisms. For that let $\bar{\vec{l}} \in {_{\uparrow}E}^{x,y}_r(n)$. Then there is a corresponding $(x,y)$ strip $\mathfrak{s} = (\vec{l}_i) \in \oplus_{i\geq 0} \TE^{x+i,y-i}(n)$ such that $d\mathfrak{s}$ is a $(x+r, y-r+1)$ strip. This implies that $\vec{l}_0$ is a $\vec{0}$ cycle and $\vec{\delta}^{x+i, y-i}_n \circ l_i = {_{\uparrow}d}_T^{x+i+1,y-i-1}(n)(\vec{l}_{i+1})$, $0 \leq i \leq r-2$. Moreover, the image of $\bar{\vec{l}}$ under ${_{\uparrow}d}_{T,r}^{x,y}(n)$ is determined by $\vec{\delta}^{x+r-1, y- r+1}_n \circ \vec{l}_{r-1}$. The well-definedness of ${_{\uparrow}d}_{T,r}^{x,y}(n)$ allows us to change each $\vec{l}_i$ up to a boundary. Hence, by induction on $i$, we may assume that the $\vec{\tilde{\delta}}^{x+i, y-i} \circ \vec{\tilde{l}}_i$ cycle $\vec{l}_{i+1}$ admits a lift $\vec{\tilde{l}}_{i+1}$ that is nice. It follows that if for each $u_x$, $\SQ^{u_x}$ is the sequence of length = $\length(J^{u_x, x})$ and consisting only of zeros then $\vec{\tilde{l}}_0$ is $\vec{\SQ}^{x,y}_{r,0} \coloneq \vec{\mathfrak{D}}^{y}_{\TE^{x, \bullet}}((\SQ^{u_x})_{0 \leq u_x \leq y+1})$-periodic and that $\vec{\tilde{l}}_{i+1}$ is $\vec{\SQ}^{x,y}_{r,i+1} \coloneq \vec{\mathfrak{D}}^{y-i-1}_{\TE^{x+i+1, \bullet}}(\vec{\delta}^{x+i, y-i}_{\star}(\vec{\SQ}^{x,y}_{r,i}))$-periodic. Define \[\vec{N}^{x,y}_r\coloneq \max\{\vec{\mathfrak{I}}^{y}_{\TE^{x, \bullet}}((\SQ^{u_x})_{0 \leq u_x \leq y+1}), \max_{0 \leq i \leq r-1} \vec{\mathfrak{I}}^{y-i-1}_{\TE^{x+i+1, \bullet}}(\vec{\delta}^{x+i, y-i}_{\star}(\vec{\SQ}^{x,y}_{r,i}))\}.\] Then, as in the proof of Theorem~\ref{thm:maintheorem}, the right square in Figure~\ref{Fig:21} commutes if $p^{N^{x,y}_r} \mid a$. We now define:
\begin{align}
\vec{M}^{x,y}_{r+1} &\coloneq \max(\vec{M}^{x,y}_r, \vec{M}^{x-r,y+r-1}_r, \vec{M}^{x+r,y-r+1}_r, \vec{N}^{x,y}_r, \vec{N}^{x-r,y+r-1}_r), \label{equ:vectorrec1}\\
\vec{\SD}^{x,y}_{r+1} &\coloneq \max(\vec{\SD}^{x,y}_r, \; \vec{\SD}^{x-r,y+r-1}_r, \; \vec{\SD}^{x+r,y-r+1}_r) \label{equ:vectorrec2}
\end{align}

\noindent and notice that the diagram above commutes and vertical arrows are isomorphisms as long as $p^{\vec{M}^{x,y}_{r+1}} \mid a$ and $n-a \geq \vec{\SD}^{x,y}_{r+1}$. Hence, under these assumptions on $n$ and $a$, $\vec{\mathfrak{R}}^{x,y}_r(n,a)$ induces an isomorphism $\vec{\mathfrak{R}}^{x,y}_{r+1}(n,a): \upTE^{x,y}_{r+1}(n) \ra \upTE^{x,y}_{r+1}(n-a)$, completing the inductive step of the proof. The proof is now complete by Remark~\ref{remark:totalspectral} and defining $C^{x,y}$ to be $\max_{0 \leq x \leq N} C_x$ for $N$ sufficiently large such that the $\upTE^{x,y}_{\infty}$ depends only on the first $N$ columns of the complex.
\end{proof}

\begin{remark}
\label{rem:bound-general}
Assume that each $V^x$ is generated in degree at most  $D$ and that for each $x$ and $u_x$, the length of $J^{u_x, x}$ is bounded by $d$. Then a proof similar to the one in Lemma~\ref{lem:bound-main} gives the following bounds:
\begin{align*}
\vec{M}_{\infty}^{x,y} & \leq (x+y+3)D \\
\vec{\SD}_{\infty}^{x,y} & \leq 2(x+y+ d - 1) +D. 
\end{align*} By Remark~\ref{rem:dimension-generation}, we can replace $D$ above by $\max_x{\chi(V^x)}$.
\end{remark}

\begin{corollary} \label{corollary:kerperiodicity}
Let $\delta:U \ra V$ be a map of finitely generated $\FI$-modules and let $\delta_{n,t}: H^t(\fS_n, U_n) \ra H^t(\fS_n, V_n)$ be the induced map on cohomology. Then ${\ker{\delta_{n,t}}}$ and ${\im{\delta_{n,t}}}$ are eventually periodic in $n$, with period a power of $p$. 
\end{corollary}

\begin{proof}
Consider the complex $0 \ra V^0 \ra V^1 \ra 0$ where $V^0 \coloneq U$, $V^1 \coloneq V$ and $\delta^0\coloneq\delta$ and note that ${\ker{\delta_{n,t}}} = {{_{\uparrow}E}^{0,t}_{2}(n)}$. Hence by Theorem~\ref{thm:generalizedmain}, $\ker{\delta_{n,t}}$ is periodic in $n$. The statement about ${\im{\delta_{n,t}}}$ follows from it.
\end{proof}

\begin{corollary}
\label{corollary:connectinghomperiodicity}
Let $0 \ra V^0 \ra V^1 \ra V^2 \ra 0$ be an exact sequence and $c_{n,t}:H^t(\fS_n, V^2_n) \ra H^{t+1}(\fS_n, V^0_n)$ be the connecting homomorphism. Then, ${\ker{c_{n,t}}}$ is eventually periodic in $n$, with period a power of $p$.
\end{corollary}

\begin{proof}
The proof follows by applying Theorem~\ref{thm:generalizedmain} to the exact sequence $0 \ra V^0 \ra V^1 \ra V^2 \ra 0$.
\end{proof}

\section{mod-$p$ cohomology of unordered configuration spaces}
\label{section:application}

We maintain the assumption that $\bk$ is a field of characteristic $p$ and we fix a manifold $\mathcal{M}$. With the notations of \S\ref{subsection:applicationintro}, let $C_{\bullet}$ and $V^{\bullet}$ be the singular chain and cochain complexes  associated to $\Conf(\mathcal{M})$ with differentials $\delta_{\bullet}$ and $\delta^{\bullet}$ respectively. Then for each $x\geq 0$, $C_x$ is a co-$\FI$-module and $V^x$ is an $\FI$-module. The degree $n$ part, $C_{x,n}$ of $C_x$ is $\bk[\fS_n]$ projective and we have $V^x_n = \Hom(C_{x,n}, \bk)$. The following argument stolen from the classical reference \cite{CE} and included for the sake of completeness shows that $H^t(\fS_n, V^x_n) = 0$ for $t>0$.

Since $C_{x,n}$ is projective, then the natural $\bk[\fS_n]$-module exact sequence
\[0 \ra \ker{g_n} \ra \bk[\fS_n] \otimes_{\bk} C_{x,n} \xrightarrow{\makebox[1cm]{$g_n$}} C_{x,n} \ra 0\] splits. By taking $\Hom(., \bk)$, we see that the following sequence is split exact.\[
0 \ra V^x_n
\xrightarrow{\makebox[2cm]{$\Hom(g_n, \bk)$}} \Hom(\bk[\fS_n] \otimes_{\bk} C_{x,n},\bk) \ra \Hom(\ker{g_n}, \bk) \ra 0. \] Then, it follows from the natural isomorphism $\Hom(\bk[\fS_n] \otimes_{\bk} C_{x,n}, \bk) \cong \Hom(\bk[\fS_n], V^x_n)$ that $V_n^x$ is a direct summand of  $\Hom(\bk[\fS_n], V^x_n)$ as a $\bk[\fS_n]$-module. By Shapiro's lemma, we have $H^t(\fS_n, \Hom(\bk[\fS_n], V_n^x))= \Ext^t_{\bk[\fS_n]}(\bk, \Hom(\bk[\fS_n], V_n^x)) \cong \Ext^t_{\bk}(\bk, V_n^x)$. Since $\bk$ is a field, $\Ext^t_{\bk}(\bk, V_n^x) =0$ for $t>0$. This implies $H^t(\fS_n, V^x_n) = 0$ for $t>0$ because $V_n^x$ is a direct summand of  $\Hom(\bk[\fS_n], V^x_n)$.

Consider the natural spectral sequence $F^{\bullet, \bullet}(n)$ given by $F^{x,y}(n) \coloneq \Hom_{\fS_n}(\mathcal{B}_y(\fS_n), V^x_n)$. The previous paragraph shows that ${_{\uparrow}F}^{x,y}_{1}(n) = 0$ for $y>0$ and ${_{\uparrow}F}^{x,0}_{1}(n) = (V^x_n)^{\fS_n}$. Hence ${_{\uparrow}F}^{\bullet, \bullet}_{r}(n) \Longrightarrow H^{\bullet}((V_n^\bullet)^{\fS_n})$, the mod-${p}$ cohomology of $\conf_n(\mathcal{M})$. Thus we must also have ${_{\rightarrow}F}^{\bullet, \bullet}_{r}(n) \Longrightarrow H^{\bullet}((V_n^\bullet)^{\fS_n})$. Note that $\mathcal{B}_y(\fS_n)$ is  free as a $\bk[\fS_n]$-module, which implies that we have a natural isomorphism ${_{\rightarrow}F}^{x, y}_{1}(n) \cong \Hom_{\fS_n}(\mathcal{B}_y(\fS_n), H^x(V^{\bullet})_n)$. We can now prove the main theorem of this section.

\begin{proof}[{\bf Proof of Theorem~\ref{thm:confmain}}] Since $\mathcal{M}$ satisfies the hypotheses of Theorem~\ref{thm:configurationsfg}, $H^x(V^{\bullet})$ is a finitely generated $\FI$-module. Note that $V^x$ may not be a finitely generated $\FI$-module. We construct a sub-complex $U^{\bullet} \hookrightarrow V^{\bullet}$ of finitely generated $\FI$-modules as follows. By induction on $x$ assume that $U^x$ has been constructed and consider the exact sequence\[ 0 \ra \im{\delta^x} \ra \ker{\delta^{x+1}} \ra H^{x+1}(V^{\bullet}) \ra 0. \] Since $\delta^x(U^x)$ and $H^{x+1}(V^{\bullet})$ are finitely generated $\FI$-modules there exists a finitely generated sub $\FI$-module $U^{x+1} \subset \ker{\delta^{x+1}} \subset V^{x+1}$ such that the following sequence is exact.\[ 0 \ra \delta^x(U^x) \ra U^{x+1} \ra H^{x+1}(V^{\bullet}) \ra 0.\]

Consider the spectral sequence $E^{\bullet, \bullet}(n)$ given by $E^{x,y}(n) \coloneq \Hom_{\fS_n}(\mathcal{B}_y(\fS_n), U^x_n)$. We have a natural map $\Psi(n): E^{\bullet, \bullet}(n) \hookrightarrow F^{\bullet, \bullet}(n)$. By construction we have, $H^x(U^{\bullet}) \cong H^x(V^{\bullet})$ for each $x$. Hence $\Psi(n)$ induces an isomorphism ${_{\rightarrow}E}^{x, y}_{1}(n) \cong {_{\rightarrow}F}^{x, y}_{1}(n)$. It follows that ${_{\uparrow}E}^{\bullet, \bullet}_{r}(n) \Longrightarrow H^{\bullet}((V_n^\bullet)^{\fS_n})$, the mod-${p}$ cohomology of $\conf_n(\mathcal{M})$. By  applying Theorem~\ref{thm:generalizedmain} to the complex $U^{\bullet}$ and defining \begin{align}
\vec{M}_{\infty}^t & \coloneq \max_{x+y = t, 1 \leq r < \infty}{\vec{M}^{x,y}_r}, \label{eqn:confperiod} \\
\vec{\SD}_{\infty}^t & \coloneq \max_{x+y = t, 1 \leq r < \infty}{\vec{\SD}^{x,y}_r}, \text{ and}, \label{eqn:confstability} \\
 C^t & \coloneq \max_{x+y = t, 1 \leq r < \infty}{C^{x,y}_r} \label{eqn:confCt} 
\end{align} completes the proof.
\end{proof}

\begin{remark}
\label{rem:bound-config}
The proof of \cite[Theorem~E]{CEFN} implies that $\dim_{\bk} H^t(\Conf_n(\cM), \bk)$ is eventually a polynomial of degree at most $2 t$. Also note that to find out $H^t(\Conf_n(\cM), \bk)$ we only need to consider the complex $V^{\bullet}$ for $x \leq t+1$. Hence Remark~\ref{rem:bound-general} implies that an upper bound on $\vec{M}^t_{\infty}$ is $(t+3)(2 t+2)$.
\end{remark}

\begin{remark}
If $\bk = \mathbb{Z}$ then it is known that the cohomology groups $H^2(\conf(S^2), \bk)$ are not eventually periodic; see \cite{Nap}. This implies that Theorem~\ref{thm:generalizedmain} does not hold with integer coefficients.
\end{remark}

\section{Further questions and comments}
\label{section:questions-comments}
In this section we list some related open problems:

\begin{questionx}
Does Theorem~\ref{thm:filteredstabilityrange} hold when $\bk = \mathbb{Z}$? More precisely, Is it true that if $V$ is a finitely generated $\sharp$-filtered $\FI$-module over $\mathbb{Z}$, then the cohomology groups $H^t(\fS_n, V_n)$ are eventually periodic and $t>0$?
\end{questionx}

\begin{questionx}
Let $V$ and $W$ be finitely generated $\FI$-modules over a field $\bk$ of positive characteristic. Is it true that the groups $\Ext_{\bk[\fS_n]}(V_n, W_n)$ are eventually periodic in $n$?
\end{questionx}

\begin{questionx}
Does there exist a manifold $\cM$ satisfying the hypothesis of Theorem~\ref{thm:configurationsfg} such that, for some $t \geq 0$, the smallest period (eventually) of the mod-$p$ cohomology groups $H^t(\conf_n(\cM), \mathbb{F}_p)$ is $p^2$ or a higher power of $p$? For all the examples we know so far, either the period is $p$ or the stabilization occurs.
\end{questionx}

Let $\FI_d$ be the category whose objects are finite sets and morphisms are injections with $d$-coloring of the complement. More precisely, any $f \in \Hom_{\FI_d}(A,B)$ is a pair $f = (g,h)$ where $g \colon A \ra B$ is an injection and $h \colon B \setminus g(A) \ra [d]$ is any function. We have the following conjecture:

\begin{conjecturex} Let $\bk$ be an arbitrary field and $V$ be a finitely generated $\FI_d$-module over $\bk$. Then $\dim_{\bk}{H^t(\fS_n, V_n)}$ is a quasi-polynomial in $n$ of degree at most $d-1$. 
\end{conjecturex}


\end{document}